\documentclass[letterpaper, twoside, 11pt]{amsproc}

\usepackage{amsmath,amsthm,amsfonts,amssymb,amscd}
\usepackage{tikz}
\usepackage{tikz-cd}
\usepackage{euler}
\usepackage{palatino}
\usepackage[symbol*]{footmisc}
\usepackage{standalone}
\usepackage[hidelinks]{hyperref}
\usetikzlibrary{arrows,chains,matrix,positioning,scopes}
\usepackage[letterpaper, left=2.5cm,right=2.5cm,top=3cm,bottom=3cm]{geometry}

\newcommand{\CC}{{\mathbb C}}
\newcommand{\RR}{{\mathbb R}}

\newcommand{\ZZ}{{\mathbb Z}}
\newcommand{\QQ}{{\mathbb Q}}
\newcommand{\hhom}{{\text{hom}}}
\newcommand{\Spec}{{\text{Spec}}}
\newcommand{\Hom}{{\text{Hom}}}
\newcommand{\CH}{{\text{CH}}}
\newcommand{\Ext}{{\text{Ext}}}
\newcommand{\Jac}{{\text{Jac}}}
\newcommand{\hol}{{\text{hol}}}
\renewcommand{\Im}{{\text{Im}}}
\newcommand{\Per}{{\text{Per}}}

\newcommand{\PD}{{\text{PD}}}
\newcommand{\AJ}{{\text{AJ}}}
\newcommand{\rank}{{\text{rank}}}
\newcommand{\inhom}{{\underline{\text{Hom}}}}
\newcommand{\cupprod}{\mathbin{\smile}}
\theoremstyle{definition}

\theoremstyle{definition}
\newtheorem*{rem}{Remark}

\newcounter{subsecnum}[section]

\newcounter{parnum}[section]

\newtheorem{prop}{Proposition}
\newtheorem{lemma}{Lemma}  
\newtheorem{thm}{Theorem} 
\newtheorem{cor}{Corollary}
\newcounter{subparnum}[parnum]

\tikzset{node distance=2cm, auto} 

\makeatletter
\tikzset{join/.code=\tikzset{after node path={%
\ifx\tikzchainprevious\pgfutil@empty\else(\tikzchainprevious)%
edge[every join]#1(\tikzchaincurrent)\fi}}}
\makeatother
\tikzset{>=stealth',every on chain/.append style={join},
         every join/.style={->}}
\tikzstyle{labeled}=[execute at begin node=$\scriptstyle,
   execute at end node=$]
	
\DefineFNsymbolsTM{myfnsymbols}{% def. from footmisc.sty "bringhurst" symbols
  \textasteriskcentered *
  \textdagger    \dagger
  \textdaggerdbl \ddagger
  \textsection   \mathsection
  \textbardbl    \|%
  \textparagraph \mathparagraph
}
\begin{document}
\title[Algebraic Cycles and $\pi_1$ of a Punctured Curve]{Algebraic Cycles, Fundamental Group of a Punctured Curve, and Applications in Arithmetic}
\author{Payman Eskandari}
\address{Department of Mathematics, University of Toronto, 40 St. George St., Room 6290, Toronto, Ontario, Canada, M5S 2E4}
\email{payman@math.utoronto.ca}
\begin{abstract}
The results of this paper can be divided into two parts, geometric and arithmetic. Let $X$ be a smooth projective curve over $\CC$, and $e,\infty\in X(\CC)$ be distinct points. Let $L_n$ be the mixed Hodge structure of functions on $\pi_1(X-\{\infty\},e)$ given by iterated integrals of length $\leq n$ (as defined by Hain). In the geometric part, inspired by a work of Darmon, Rotger, and Sols \cite{DRS}, we express the mixed Hodge extension $\mathbb{E}^\infty_{n,e}$ given by the weight filtration on $\frac{L_n}{L_{n-2}}$ in terms of certain null-homologous algebraic cycles on $X^{2n-1}$. These cycles are constructed using the diagonal embeddings of $X^{n-1}$ into $X^n$. As a corollary, we show that the extension $\mathbb{E}^\infty_{n,e}$ determines the point $\infty\in X-\{e\}$.\\
\indent The arithmetic part of the paper gives some number-theoretic applications of the geometric part. We assume that $X=X_0\otimes_K\CC$ and $e,\infty\in X_0(K)$, where $K$ is a subfield of $\CC$ and $X_0$ is a projective curve over $K$. Let $\Jac$ be the Jacobian of $X_0$. We use the extension $\mathbb{E}^\infty_{n,e}$ to associate to each $Z\in\CH_{n-1}(X_0^{2n-2})$ a point $P_Z\in\Jac(K)$, which can be described analytically in terms of iterated integrals. The proof of $K$-rationality of $P_Z$ uses that the algebraic cycles constructed in the geometric part of the paper are defined over $K$. Assuming a certain plausible hypothesis on the Hodge filtration on $L_n(X-\{\infty\},e)$ holds, we show that an algebraic cycle $Z$ for which $P_Z$ is torsion, gives rise to relations between periods of $L_2(X-\{\infty\},e)$. Interestingly, these relations are non-trivial even when one takes $Z$ to be the diagonal of $X_0$. In the elliptic curve case, we show unconditionally that a certain relation between periods of $L_2(X-\{\infty\},e)$ (which is induced by the diagonal of $X_0$) exists if and only if $e-\infty$ is torsion.\\ 
\indent The geometric result of the paper in $n=2$ case, and the fact that one can associate to $\mathbb{E}^\infty_{2,e}$ a family of points in $\Jac(K)$, are due to Darmon, Rotger, and Sols \cite{DRS}. Our contribution is in generalizing the picture to higher weights.
\end{abstract}
\maketitle
\vspace{-.4in}
\tableofcontents

\section{Introduction}

Let $U$ be a smooth (connected) variety over $\CC$ and $e\in U(\CC)$. Thanks to the works of Chen, Hain, Deligne, Morgan and others one has, for each $n$, a mixed Hodge structure $L_n(U,e)$ with integral lattice 
\[\left(\frac{\ZZ[\pi_1(U,e)]}{I^{n+1}}\right)^\vee,\] 
where $I\subset\ZZ[\pi_1(U,e)]$ is the augmentation ideal. The filtrations (Hodge and weight) are defined using the characterization of 
\begin{equation}\label{eq1intro}
\left(\frac{\CC[\pi_1(U,e)]}{I^{n+1}}\right)^\vee,
\end{equation}
where $I$ is again the augmentation ideal, as the space of closed (i.e. homotopy invariant) iterated integrals of length $\leq n$ on $U$. One has
\[
L_1(U,e)\simeq \ZZ(0)\oplus H^1(U),
\]
but the $L_n(U,e)$ are more complicated for $n>1$. In particular, they may not be pure even if $U$ is projective.\\

There are two aspects of the Hodge realization of the fundamental group that are of particular interest to us:\\

1. \underline{Connections to null-homologous algebraic cycles}: Over the past few decades, a number of connections have been found between the Hodge theory of the fundamental group and null-homologous algebraic cycles. See for instance \cite{Harris}, \cite{Pulte}, \cite{Colombo}, and the expository paper \cite{Hain2}. More recently, Darmon, Rotger, and Sols in \cite{DRS} considered the extension
\[
0\rightarrow \frac{L_1}{L_0}(U,e)\rightarrow\frac{L_2}{L_0}(U,e)\rightarrow\frac{L_2}{L_1}(U,e)\rightarrow 0,
\] 
where $U$ is obtained from a smooth projective curve $X$ over a subfield $K\subset\CC$ by removing a $K$-rational point, and $e\in U(K)$. They related this extension to the modified diagonal cycle of Gross, Kudla, and Schoen in $X^3$. Using this relation they were able to define a family of rational points on the Jacobian of $X$ parametrized by algebraic cycles in $X^2$. One of the primary goals of this paper is to generalize this picture to higher weights. We will discuss this in more detail shortly.\\

2. \underline{Periods}: Similar to the cohomology case, if $U$ and $e$ are defined over a subfield $K\subset\CC$, $L_n(U,e)$ is endowed with a {\it de Rham lattice}, which is a $K$-lattice inside \eqref{eq1intro}. One then has a $K$-vector space of {\it periods} of $L_n(U,e)$, which contains the periods of $U$ if $n\geq 1$. The new phenomenon here is that because of a formal property of iterated integrals, namely the so called {\it shuffle product}, periods of $\cup L_n(U,e)$ that correspond to the same path in $\pi_1(U,e)$, are closed under multiplication, and form a $K$-subalgebra of $\CC$. One refers to the periods of $\cup L_n(U,e)$ as the periods of $\pi_1(U,e)$. The celebrated multiple zeta values arise as periods of $\pi_1$ of $\mathbb{P}^1-\{0,1,\infty\}$.\\

We proceed to give a review of the results of the paper. The work can be divided into two parts, geometric and arithmetic. Before we discuss the contents of each part, let us fix some notation. We use $\CH_i(-)$ for Chow groups. (As usual, the subscript is the dimension.) By $\CH_i^\hhom(-)$ we mean the subgroup of $\CH_i(-)$ consisting of homologically trivial cycles. We denote by $\inhom$ the internal Hom in the category of mixed Hodge structures, and for a pure Hodge structure $A$ of odd weight $2k-1$, by $JA$ we refer to the ``middle" Carlson Jacobian
\[
JA:=\frac{A_\CC}{F^kA_\CC+A_\ZZ},
\]
where $F^\cdot$ denotes the Hodge filtration. For instance, if $A=H^{2k-1}(U)$ for a smooth projective complex variety $U$, $JA$ is nothing but the Griffiths' intermediate Jacobian.\\

From now on, $X$ is a smooth (connected) projective curve over $\CC$. Let $e,\infty\in X(\CC)$ be distinct. We write $H^1$ for $H^1(X)$, the mixed Hodge structure associated to the degree one cohomology of $X$.\\

{\bf 1. Geometric part:} (Up to Section \ref{geometriccor}) Darmon, Rotger and Sols in \cite{DRS} relate the extension $\mathbb{E}_{2,e}^\infty$
\[
\begin{array}{ccccccccc}
0& \longrightarrow & \displaystyle{\frac{L_{1}}{L_{0}}(X-\{\infty\},e)} & \longrightarrow &\displaystyle{\frac{L_2}{L_0}(X-\{\infty\},e)} &\longrightarrow &\displaystyle{\frac{L_2}{L_{1}}(X-\{\infty\},e)} &\longrightarrow &0,\\
 &             &    \text{\rotatebox{90}{$\cong$}} & & & &\text{\rotatebox{90}{$\cong$}} & &\\
 & & H^1 & & & & (H^1)^{\otimes 2} & &\\
\end{array}
\]
to the modified diagonal cycle of Kudla, Gross and Schoen\footnote[2]{The reason for this non-standard choice of notation will be clear shortly.}
\begin{eqnarray*}
\Delta_{2,e}&:=&\{(x,x,x):x\in X\}-\{(e,x,x):x\in X\}-\{(x,e,x):x\in X\}-\{(x,x,e):x\in X\}\\
&+&~\{(e,e,x):x\in X\}+\{(e,x,e):x\in X\}+\{(x,e,e):x\in X\}\in \CH_1^{\hhom}(X^3)
\end{eqnarray*}
and the cycle
\[
Z^\infty_{2,e}:=\{(x,x,\infty):x\in X\}-\{(x,x,e):x\in X\}\in\CH_1^{\hhom}(X^3).
\]
Let $h_2$ be the composition
\begin{equation}\label{eqsum2}
\CH_1^\hhom(X^3)\stackrel{\text{Abel-Jacobi}}{\longrightarrow}J\inhom(H^3(X^3),\ZZ(0))\stackrel{\text{Kunneth}}{\longrightarrow}J\inhom((H^1)^{\otimes 3},\ZZ(0)),
\end{equation} 
and identify  
\[
\Ext((H^1)^{\otimes2},H^1)\cong J\inhom((H^1)^{\otimes2},H^1)\cong J\inhom((H^1)^{\otimes2}\otimes H^1,\ZZ(0)),
\]
where the first isomorphism is that of Carlson \cite{Carlson}, and the second is given by Poincare duality. Theorem 2.5 of \cite{DRS} asserts\footnote[3]{The result in \cite{DRS} is slightly weaker, but a small modification of its proof implies \eqref{eqsum3}. See Section \ref{section_DRS_result_reformulation}.} that
\begin{equation}\label{eqsum3}
\mathbb{E}_{2,e}^\infty=h_2(-\Delta_{2,e}+Z^\infty_{2,e}).
\end{equation}
Our goal in the first part of the paper is to generalize this result to higher weights. For each $n\geq 2$, we consider the extension $\mathbb{E}^\infty_{n,e}$ 
\[
\begin{array}{ccccccccc}
0& \longrightarrow & \displaystyle{\frac{L_{n-1}}{L_{n-2}}(X-\{\infty\},e)} & \longrightarrow &\displaystyle{\frac{L_n}{L_{n-2}}(X-\{\infty\},e)} &\longrightarrow &\displaystyle{\frac{L_n}{L_{n-1}}(X-\{\infty\},e)} &\longrightarrow &0\\
 &             &    \text{\rotatebox{90}{$\cong$}} & & & &\text{\rotatebox{90}{$\cong$}} & &\\
 & & (H^1)^{\otimes n-1} & & & & (H^1)^{\otimes n} & &\\
\end{array}
\] 
of mixed Hodge structures as an element of $\Ext\left((H^1)^{\otimes n},(H^1)^{\otimes n-1}\right)$. One can show that the weight filtration on 
\[
\frac{L_n}{L_{n-2}}(X-\{\infty\},e)
\] 
is given by 
\[
W_{n-2}=0,\hspace{.2in} W_{n-1}=\frac{L_{n-1}}{L_{n-2}}(X-\{\infty\},e),\hspace{.2in}\text{and}\hspace{.2in}W_n=\frac{L_n}{L_{n-2}}(X-\{\infty\},e),
\] 
so that it gives rise to only one interesting extension, namely $\mathbb{E}^\infty_{n,e}$.\\
 
Let $h_n$ be the composition
\begin{equation*}
\CH_{n-1}^\hhom(X^{2n-1})\stackrel{\text{Abel-Jacobi}}{\longrightarrow}J\inhom(H^{2n-1}(X^{2n-1}),\ZZ(0))\stackrel{\text{Kunneth}}{\longrightarrow}J\inhom((H^1)^{\otimes 2n-1},\ZZ(0)),
\end{equation*}
and identify
\[
\Ext((H^1)^{\otimes n}, (H^1)^{\otimes n-1})\stackrel{\text{Carlson}}{\cong}J\inhom\left((H^1)^{\otimes n}, (H^1)^{\otimes n-1}\right)\stackrel{\text{Poincare duality}}{\cong}J\inhom((H^1)^{\otimes 2n-1},\ZZ(0)).
\]
For each $n$, we define algebraic cycles 
\[\Delta_{n,e}, Z^\infty_{n,e}\in \CH^\hhom_{n-1}(X^{2n-1})\]
such that \eqref{eqsum3} generalizes to the following result.
 \begin{thm}\label{thm1intro}
\[
\mathbb{E}^\infty_{n,e}=(-1)^{\frac{n(n-1)}{2}} h_n\left(\Delta_{n,e}-Z^\infty_{n,e}\right)
\]\\
\end{thm}
The cycle $\Delta_{n,e}$ is constructed by first taking an alternating sum
\[\sum\limits_i(-1)^{i-1}~{}^t\Gamma_{\delta_i}\]
of the transposes of the graphs of the diagonal embeddings $\delta_i:X^{n-1}\longrightarrow X^n$ defined by
\begin{equation}\label{eqsum4}
(x_1,\ldots,x_{n-1})\mapsto(x_1,\ldots,x_i,x_i,\ldots,x_{n-1}),
\end{equation}
and then using the method of Gross and Schoen \cite{GS} to produce a null-homologous cycle. The cycle $Z^\infty_{n,e}$ is defined as
\[
\sum\limits_{i=1}^{n-1} (-1)^{i-1} \left((\pi_{n+i,\infty})_\ast-(\pi_{n+i,e})_\ast\right) ({}^t\Gamma_{\delta_i}),
\]
where for $x\in X$, $\pi_{i,x}$ is the map $X^{2n-1}\rightarrow X^{2n-1}$ that replaces the $i^\text{th}$ coordinate by $x$, and leaves the other coordinated unchanged.\\

Note that the fact that the diagonal embeddings $\delta_i: X^{n-1}\rightarrow X^n$ appear in the constructions is not surprising. Wojtkowiak used these maps in \cite{Wojt} to form a cosimplicial scheme that gives rise to the de Rham fundamental group, and Deligne and Goncharov used these maps in \cite{DG} to construct their motivic fundamental group.\\

Theorem \ref{thm1intro} has the following corollaries: 
\begin{itemize}
\item[(1)] The function
\[
X(\CC)-\{e\}\rightarrow \Ext((H^1)^{\otimes n}, (H^1)^{\otimes n-1})\hspace{.2in} \infty\mapsto \mathbb{E}^\infty_{n,e}
\]
is injective. 
\item[(2)] If $X$ is of genus 1, $\mathbb{E}^\infty_{n,e}$ is torsion if and only if $\infty-e\in\CH_0^\hhom(X)$ is torsion.
\end{itemize}

We should mention that one motivation for considering extensions of the form
\[
\begin{array}{ccccccccc}
0& \longrightarrow & \displaystyle{\frac{L_{n-1}}{L_{n-2}}} & \longrightarrow &\displaystyle{\frac{L_n}{L_{n-2}}} &\longrightarrow &\displaystyle{\frac{L_n}{L_{n-1}}} &\longrightarrow &0,\\
\end{array}
\]
rather than 
\[
\begin{array}{ccccccccc}
0& \longrightarrow & \displaystyle{L_{n-1}} & \longrightarrow &\displaystyle{L_n} &\longrightarrow & \displaystyle{\frac{L_n}{L_{n-1}}} &\longrightarrow &0,\\
\end{array}
\]
is that the quotients $\{\displaystyle{\frac{L_n}{L_{n-1}}}\}$ are independent of the base point, so that we can think of extensions coming from different base points as elements of the same Ext group. The reason for looking at extensions coming from $\pi_1$ of the punctured curve, rather than the curve $X$ itself, is that the successive quotients $\displaystyle{\frac{L_n}{L_{n-1}}(X,e)}$ for $n>2$ are much more complicated than their counterparts for $X-\{\infty\}$. (See \cite{Stallings}.)\\

{\bf 2. Arithmetic part:} Here we give some number theoretic applications for Theorem \ref{thm1intro}. Suppose $K\subset\CC$ is a subfield, $X=X_0\otimes_K\CC$, where $X_0$ is a (smooth) projective curve over $K$, and $e,\infty\in X_0(K)$. Let $g$ be the genus. Denote the Jacobian of $X_0$ by $\Jac$.\\

2A. \underline{Application to rational points on the Jacobian}: (Section \ref{ch4}) Following the ideas of \cite{DRS}, we associate to the extension $\mathbb{E}^\infty_{n,e}$ a family of points in $\Jac(K)$ parametrized by algebraic cycles of the appropriate dimension in a certain power of $X_0$. Our approach is in line with Darmon's general philosophy of constructing rational points on Jacobians of curves using algebraic cycles on higher dimensional varieties.\\
Throughout, we identify
\[
\Jac(\CC)\cong J\inhom((H^1)^{\otimes 2n-1},\ZZ(0)).
\]
For a Hodge class
\[
\xi\in (H^1)^{\otimes 2n-2},
\] 
let $\xi^{-1}$ be the map
\[
J\inhom((H^1)^{\otimes 2n-1},\ZZ(0))\rightarrow J\inhom(H^1,\ZZ(0))\cong\Jac(\CC)
\]
defined by
\[
\bigm(\text{class of $f:(H^1_\CC)^{\otimes 2n-1}\rightarrow\CC$}\bigm)~~\mapsto~~ \bigm(\text{class of $f(\xi\otimes-)$}\bigm).
\]
For $Z\in\CH_{n-1}(X_0^{2n-2})$, let $\xi_Z$ be the $(H^1)^{\otimes 2n-2}$ Kunneth component of the class of $Z$. In Section \ref{ch4} we prove the following result.
\begin{thm}\label{thm2intro}
Let $Z\in\CH_{n-1}(X_0^{2n-2})$. Then $\xi_Z^{-1}(\mathbb{E}^\infty_{n,e})\in\Jac(K)$. 
\end{thm}
Note that this is not a priori obvious, as to define $\mathbb{E}^\infty_{n,e}$ one first goes to analytic topology. The result is a consequence of Theorem \ref{thm1intro} in view of the following two facts:
\begin{itemize}
\item[(i)] The map $\xi_Z^{-1}$ is given by a correspondence. More precisely, it is induced by an element of 
\[
\CH_{n}(X_0^{2n})=\CH_n(X_0^{2n-1}\times X_0)
\]
whose class is the $(H^1)^{\otimes 2n}$ component of
\[
Z\times \Delta(X_0),
\]
where $\Delta(X_0)$ is the diagonal of $X_0$. Denoting the composition
\[
\CH_{n-1}^\hhom(X_0^{2n-1})\stackrel{\text{natural map}}{\rightarrow} \CH_{n-1}^\hhom(X^{2n-1})\stackrel{h_n}{\rightarrow} J\inhom((H^1)^{\otimes 2n-1},\ZZ(0)) 
\] 
also by $h_n$, this gives us a commutative diagram
\begin{equation}
\begin{tikzpicture}
  \matrix (m) [matrix of math nodes, column sep=2em, row sep=2.5em]
    {	  \CH_{n-1}^\hhom(X_0^{2n-1})& J((H^1)^{\otimes 2n-1})^\vee\\
		\CH_{0}^\hhom(X_0)& J(H^1)^\vee.\\};
  { [start chain] \chainin (m-1-1);
    \chainin (m-1-2) [join={node[above,labeled] {h_n}}];}
  {[start chain] \chainin (m-2-1);
    \chainin (m-2-2) [join={node[above,labeled] {\text{Abel-Jacobi}}}];}
	{[start chain] \chainin (m-1-2);
		\chainin (m-2-2) [join={node[right,labeled] {\xi_Z^{-1}}}];}
	{[start chain] \chainin (m-1-1);
		\chainin (m-2-1) [join={node[left,labeled] {}}];}
  \end{tikzpicture}
\end{equation}
\item[(ii)] The algebraic cycles $\Delta_{n,e}$ and $Z^\infty_{n,e}$ are defined over $K$.
\end{itemize}
Theorem \ref{thm2intro} is due to Darmon, Rotger, and Sols \cite{DRS} in the case $n=2$. For each $n$, it associates to the extension $\mathbb{E}^\infty_{n,e}$ a family of rational points on $\Jac$ parametrized by $\CH_{n-1}(X_0^{2n-2})$.\\

To simplify the notation, we will write $P_\xi$ for $\xi^{-1}(\mathbb{E}^\infty_{n,e})$ and $P_Z$ for $P_{\xi_Z}$. The point $P_\xi$ (and in particular $P_Z$) can be described analytically using iterated integrals. Ideally, we would like to have a description in terms of algebraic 1-forms on $X_0$. Let $\Omega^1_\hol(X)$ be the space of holomorphic 1-forms on $X$. Identify
\[
\Jac(\CC)\cong \frac{\Omega^1_\hol(X)^\vee}{H_1(X,\ZZ)}.
\]  
Let $\alpha_1,\ldots,\alpha_{2g}$ be regular algebraic 1-forms on $X_0-\{\infty\}$ whose classes form a basis $H^1_{dR}(X_0)$. Moreover, suppose $\alpha_1,\ldots,\alpha_g$ are holomorphic on $X$. Let $d_1,\ldots,d_{2g}$ form a basis of $H^1_\ZZ$ such that 
\[
\int_X d_i\wedge d_j=1\hspace{.3in}\text{if $i<j$.}
\]
Let $\omega_i$ be the representative of $d_i$ in $\sum\limits_j\CC\alpha_j$. Write
\[
\alpha_i=\sum\limits_j p_{ij}\omega_j.
\]
For each $i$, let $\beta_i\in\pi_1(X-\{\infty\},e)$ be such that
\[
\int\limits_{\beta_i}-=\int\limits_X d_i\wedge-
\] 
on $H^1$. Then, assuming the $\alpha_i$ satisfy a certain hypothesis, which we refer to as Hypothesis $\star=\star(n)$ (see Paragraph \ref{analytic_description_par1}), the point 
\[
P_\xi\in\frac{\Omega^1_\hol(X)^\vee}{H_1(X,\ZZ)}
\]
is represented by
\[
f_\xi:\alpha_l\mapsto \sum\limits_{i,j,k\leq 2g}\mu'_{i,j,k}(\xi;\alpha_l) \int\limits_{\beta_k}\omega_i\omega_j.
\]
Here the coefficients
\[
\mu'_{i,j,k}(\xi;\alpha_l)\in \Per_\QQ(\alpha_l):=\sum\limits_r \QQ\int\limits_{\beta_r}\alpha_l=\sum\limits_{r}p_{lr}\QQ
\]
are explicit linear combinations (in fact, with integer coefficients) of the $p_{lr}$.\\

It will be interesting to investigate when Hypothesis $\star$ holds. We show that in the case $g=1$, the hypothesis is indeed satisfied if $\alpha_2$ has a pole of order 2 at $\infty$. For instance, if $X_0$ is given by the affine equation
\begin{equation}\label{eq4introrevision1}
y^2=4x^3-g_2x-g_3
\end{equation}
and $\infty$ is the point at infinity, Hypothesis $\star$ holds if $\alpha_2=\frac{xdx}{y}$.\\

2B. \underline{Application to periods}: (Sections \ref{ch5} and \ref{ch5section_examples}) Assume for the moment that the Mordell-Weil group $\Jac(K)$ has rank $\geq 1$. A natural question one can ask is whether the families
\[\{P_Z:Z\in\CH_{n-1}(X_0^{2n-2})\subset\Jac(K)\] 
contain non-torsion points.\footnote[2]{One should keep in mind that for different $n$ these families arise from different parts of the weight filtration on the mixed Hodge structure on $\pi_1(X-\{\infty\},e)$.} This led us to ask whether $P_Z$ being torsion will have any interesting consequences.\\

It is well-known that Hodge classes in tensor powers of $H^1$ induce polynomial relations (with integer coefficients) between the periods of $X_0$. In Section \ref{ch5}, we observe that a Hodge class $\xi$ for which $P_\xi$ is torsion, might induce relations between periods of $L_2(X-\{\infty\},e)$. This is an easy consequence of the analytic description of $P_\xi$. Indeed, setting
\[
\mu_{i,j,k}(\xi;\alpha_l)=\mu'_{i,j,k}(\xi;\alpha_l)-\mu'_{j,i,k}(\xi;\alpha_l)\hspace{.2in}(i,j,k\leq2g, i<j),
\]
it is easy to see that if the $\alpha_i$ satisfy Hypothesis $\star$ and $P_\xi$ is torsion, then
\begin{equation}\label{eq1introB}
\sum\limits_{\stackrel{i,j,k\leq 2g}{i<j}}\mu_{i,j,k}(\xi;\alpha_l) \int\limits_{\beta_k}\omega_i\omega_j\in \Per_\QQ(\alpha_l)\hspace{.2in}(l\leq g). 
\end{equation}
The reason for writing these only in terms of the triples $(i,j,k)$ satisfying $i<j$ is that thanks to the shuffle product property of iterated integrals,
\[
\int\limits_{\beta_k}\omega_i\omega_j+\int\limits_{\beta_k}\omega_j\omega_i=\int\limits_{\beta_k}\omega_i \int\limits_{\beta_k}\omega_i.
\] 
Let $\QQ(X_0)$ be the field generated over $\QQ$ by all the numbers $p_{ij}$ ($i,j\leq 2g$). The relations \eqref{eq1introB} can be considered as linear relations in
\begin{equation}\label{eq2introB}
1, \int\limits_{\beta_k}\omega_i\omega_j\hspace{.2in} (i,j,k\leq 2g, i<j)
\end{equation}
with coefficients in $\Per_\QQ(X_0)$. By multi-linearity of iterated integrals, they can be rewritten as linear relations between
\[
1, \int\limits_{\beta_k}\alpha_i\alpha_j\hspace{.2in} (i,j,k\leq 2g, i<j)
\]
with coefficients in $\QQ(X_0)$.\\

We then proceed in Section \ref{ch5section_examples} to specialize to the Hodge classes coming from the diagonal of $X_0$ and $X_0^2$. Even these simplest cases lead to interesting statements. 

\begin{prop}\label{propintro}
Suppose the $\alpha_i$ are chosen so that they satisfy Hypothesis $\star$.\\
(a) Suppose $P_{\Delta(X_0)}$ is torsion. Then the $g$ relations \eqref{eq1introB}, which in this case take the form 
\[
\sum\limits_{\stackrel{i,j,k\leq2g}{i<j}}(-1)^{i+j}p_{lk}\int\limits_{\beta_k}\omega_i\omega_j\in \Per_\QQ(\alpha_l)\hspace{.3in}(l\leq g),
\]
are independent (as linear relations among \eqref{eq2introB} with coefficients in $\QQ(X_0)$).\\
\noindent (b) For $1\leq i,j\leq 2g$ define the numbers $\lambda_{ij}$ by $\lambda_{ij}=(-1)^{i+j}$ if $i<j$ and $\lambda_{ij}=-\lambda_{ji}$. Suppose $P_{\Delta(X_0^2)}$ is torsion. Then the relations \eqref{eq1introB}, which in this case are 
\[
\sum\limits_{\stackrel{i,j,k}{i<j}}\biggm(\lambda_{jk}p_{li}-\lambda_{ik}p_{lj}-2(-1)^{i+j}p_{lk}\biggm)\int\limits_{\beta_k}\omega_i\omega_j~\in\Per_\QQ(\alpha_l)\hspace{.3in}(l\leq g),
\]
are independent.\\
\noindent (c) Let $g=2$. Suppose $P_{\Delta(X_0)}$ and $P_{\Delta(X^2_0)}$ are torsion. Then at least three of the relations given in (a) and (b) are independent.\\ 
\end{prop}
 
Part (c) of the proposition is particularly interesting, as it shows that by digging deeper into the weight filtration the method might indeed give new information about the periods. Also note that thanks to Theorem \ref{thm2intro}, $P_{\Delta(X_0)}$ and $P_{\Delta(X^2_0)}$ are $K$-rational, so that they are guaranteed to be torsion if it happens that $\Jac(K)$ is finite. This happens for instance when $K=\QQ$ and $X_0$ is a Fermat curve of degree an odd prime $\leq 7$ \cite{Faddeev}.\\

In the elliptic curve case, one can be more precise:  
\begin{thm}\label{thm3intro}
Let $g=1$. Suppose that $\alpha_2$ has a pole of order 2 at $\infty$. Then
\begin{equation}\label{eq3introB}
p_{11}\int\limits_{\beta_1}\omega_1\omega_2+p_{12}\int\limits_{\beta_2}\omega_1\omega_2 \equiv \int\limits_e^\infty \alpha_1\hspace{.1in}\mod \frac{1}{4}\Per_\ZZ(\alpha_1),
\end{equation}
where $\Per_\ZZ(\alpha_1)=\sum\limits_r \ZZ\int\limits_{\beta_r}\alpha_1$.
\end{thm}

The condition on the order of the pole at $\infty$ is included only to guarantee that Hypothesis $\star$ is satisfied. To prove Theorem \ref{thm3intro}, one applies $\xi_{\Delta(X_0)}^{-1}$ to \eqref{eqsum3} and uses the fact that when $g=1$, $2h_2(\Delta_{2,e})=0$.\\

Let $X_0$ be given by the affine equation \eqref{eq4introrevision1} and $\infty$ be the point at infinity. Take $\alpha_1=\frac{dx}{y}$ and $\alpha_2=\frac{xdx}{y}$. Then the classical Legendre relation says $p_{11}p_{22}-p_{12}p_{21}=2\pi i$, and \eqref{eq3introB} can be rewritten as
\[
\int\limits_{\beta_1}\alpha_1\int\limits_{\beta_2} (\alpha_1\alpha_2-\alpha_2\alpha_1)~-~\int\limits_{\beta_2}\alpha_1\int\limits_{\beta_1} (\alpha_1\alpha_2-\alpha_2\alpha_1)\equiv 4\pi i\int\limits_e^\infty \alpha_1\hspace{.1in}\mod \pi i\cdot \Per_\ZZ(\alpha_1).
\]

We close this introduction with a word on the structure of the paper. We recall some background material in Sections \ref{Hodgetheorybackground} and \ref{ch1}. Nothing in these two sections is original. Sections \ref{bar2}-\ref{geometriccor} contain the geometric component of the paper. The goal in Sections \ref{bar2}-\ref{proofgeneralcase} is to state and prove Theorem \ref{thm1intro}. In Section \ref{geometriccor} we give two corollaries of Theorem \ref{thm1intro}. The last three sections contain the arithmetic part of the paper. In Section \ref{ch4}, we prove Theorem \ref{thm2intro} and give an analytic description for the point $P_\xi$. Sections \ref{ch5} and \ref{ch5section_examples} apply the earlier results of the paper to periods. Paragraph \ref{ch5philosophy} explains the methodology in detail, namely how Hodge classes may induce relations between periods of $L_2(X-\{\infty\},e)$. Section \ref{ch5section_examples} discusses Proposition \ref{propintro} and Theorem \ref{thm3intro} above.\\

{\bf Acknowledgment.} This article is based on my PhD thesis at the University of Toronto. I am very grateful to my advisor, Professor Kumar Murty, for his continuous encouragement and guidance. I am also grateful to Professor Henri Darmon for reading and providing feedback on my thesis, and to him and Professor Steven Kudla for some enlightening discussions. Finally, I would like to thank Professor Richard Hain for a helpful correspondence. 

\numberwithin{prop}{subsection}
\numberwithin{lemma}{subsection}
\numberwithin{thm}{subsection}

\section{Recollections from Hodge theory}\label{Hodgetheorybackground}
In this section we briefly recall a few basic definitions and facts about mixed Hodge structures.

\subsection{} Unless otherwise stated, by a (pure or mixed) Hodge structure we mean one that is over $\ZZ$. We use the standard notation $F^\cdot$ and $W_\cdot$ for the Hodge and weight filtrations. We denote the category of mixed (resp. pure) Hodge structures by $\mathbf{MHS}$ (resp. $\mathbf{HS}$). We will often denote a Hodge or mixed Hodge structure by a capital English letter, and then decorate it with the subscript $\mathbb{K}\in\{\ZZ,\QQ,\CC\}$ to refer to its corresponding $\mathbb{K}$-module. For example, if $H$ is a mixed Hodge structure, by $H_\ZZ$, $H_\QQ$, and $H_\CC$ we refer to the corresponding $\ZZ$, $\QQ$, and $\CC$ modules. For each integer $n$, we denote by $\ZZ(-n)$ the unique Hodge structure of weight $2n$ with the underlying abelian group $\ZZ$.\\  

Given a mixed Hodge structure $H$, we set $W_nH_\ZZ$ to be the pre-image of $W_nH_\QQ$ under the natural map
\[
H_\ZZ\rightarrow H_\QQ.
\]
This convention is adopted so that the $W_n$ are functors $\mathbf{MHS}\rightarrow\mathbf{MHS}$. The highest (resp. lowest) weight of a mixed Hodge structure $H$ is defined to be the smallest $n$ for which $W_nH=H$ (resp. $W_nH\neq0$).\\

\subsection{Tensor product and internal Homs} Given mixed Hodge structures $A$ and $B$, one has an object $A\otimes B$ in $\mathbf{MHS}$ defined in the obvious way. For each $n$, the {\it twist} $A(n):=A\otimes \mathbb{Z}(n)$ is obtained from $A$ by shifting the filtrations. One clearly has $A(0)=A$. The category $\mathbf{MHS}$ is a tensor abelian category with $\mathbb{Z}(0)$ as the identity of the tensor product.\\

Given objects $A$ and $B$ of $\mathbf{MHS}$, their {\it internal hom} $\inhom(A,B)$ is a mixed Hodge structure defined as follows: Its underlying abelian group is $\Hom_\ZZ(A_\ZZ,B_\ZZ)$, and the filtrations are given by
\[
W_n\Hom_\QQ(A_\QQ,B_\QQ)=\{f:A_\QQ\rightarrow B_\QQ~|~f(W_lA_\mathbb{Q})\subset W_{n+l}B_\mathbb{Q}\hspace{.1in}\text{for all $l$}\}
\]
and
\[
F^p\Hom_\CC(A_\CC,B_\CC)=\{f:A_\CC\rightarrow B_\CC~|~f(F^lA_\CC)\subset F^{p+l}B_\CC\hspace{.1in}\text{for all $l$}\}.
\]
If $A$ and $B$ are pure of weights $a$ and $b$, $\inhom(A,B)$ is pure of weight $b-a$. The {\it dual} to a mixed Hodge structure $A$ is defined to be $A^\vee:=\inhom(A,\mathbb{Z}(0))$. We adopt the convention $A^{\otimes n}:=(A^{\otimes -n})^\vee$ for $n$ negative. One clearly has $\mathbb{Z}(n)=\mathbb{Z}(1)^{\otimes n}$ for all $n$.\\
%For references on the material so far, see \cite{HT2}, \cite{Voisin}, and \cite{2LNM900}.\\

\subsection{Carlson Jacobians} Motivated by Griffiths' intermediate Jacobians of a variety, given a mixed Hodge structure $A$, Carlson \cite{Carlson} defined its $n^{\text{th}}$ {\it Jacobian}\footnote[2]{One should not be misled by the use of the word Jacobian here: Carlson Jacobians of a mixed Hodge structure are often not algebraic.} by 
\[J^n(A):=\frac{A_{\CC}}{F^nA_{\CC}+A_{\ZZ}},\]
where by $A_\ZZ$ we obviously mean its image in $A_\CC$. It is easy to see that for $n$ bigger than half the highest weight of $A$, the natural map
\begin{equation}\label{eqHSbasics1}
A_\RR:=A_\ZZ\otimes\RR\rightarrow \frac{A_{\CC}}{F^nA_{\CC}}
\end{equation}
(given by the inclusion $A_{\RR}\subset A_{\CC}$) is injective, whence $J^n(A)$ is the quotient of a complex vector space by a discrete subgroup. It is easy to see that in general $J^n$ is a functor from $\mathbf{MHS}$ to the category of abelian groups that respects direct sums.\\

\noindent Of special interest to us is the case of the ``middle Jacobian" $JA:=J^nA$ of a pure Hodge structure $A$ of weight $2n-1$ (possibly negative). It is easy to see that in this case, the map \eqref{eqHSbasics1} is an isomorphism, and hence induces an isomorphism of real tori 
\begin{equation}\label{eqHSbasics2}
\frac{A_{\RR}}{A_{\ZZ}}\cong JA.
\end{equation}

We record, for future reference, a few easy statements in the following lemma.

\begin{lemma} \label{basichodge}
Let $A$, $B$ and $C$ be mixed Hodge structures.
\begin{itemize}
\item[(a)] If $B_{\ZZ}$ is free, the canonical isomorphism $\Hom_\ZZ(A_\ZZ,B_\ZZ\otimes C_\ZZ)\cong \Hom_\ZZ\left(A_\ZZ\otimes B_\ZZ^\vee,C_\ZZ\right)$ induces an isomorphism $\displaystyle{\underline{\Hom}(A,B\otimes C)\cong \underline{\Hom}\left(A\otimes B^\vee,C\right)}$.
\item[(b)] The canonical isomorphism $\Hom_\ZZ(A_\ZZ,B_\ZZ)\otimes C_\ZZ\cong \Hom_\ZZ\left(A_\ZZ, B_\ZZ\otimes C_\ZZ\right)$ induces an isomorphism $\underline{\Hom}(A,B)\otimes C\cong \underline{\Hom}\left(A, B\otimes C\right)$.
\item[(c)] $J^nA(-p)=J^{n-p}A$
\item[(d)] If $A$ is pure of odd weight, $JA(-p)=JA$.
\item[(e)] $J^n\underline{\Hom}(A(-p),B)=J^{n+p}\underline{\Hom}(A,B)$.
\item[(f)] If $A$ and $B$ are pure of opposite parity weights, then $J\underline{\Hom}(A(-p),B)=J\underline{\Hom}(A,B)$.
\end{itemize}
\end{lemma}

The proofs are all straightforward. For (a) (resp. (b)) one notes that the canonical isomorphisms $\Hom_\mathbb{K}(A_\mathbb{K},B_\mathbb{K}\otimes C_\mathbb{K})\cong \Hom_\mathbb{K}\left(A_\mathbb{K}\otimes B_\mathbb{K}^\vee,C_\mathbb{K}\right)$ (resp. $\Hom_\mathbb{K}(A_\mathbb{K},B_\mathbb{K})\otimes C_\mathbb{K}\cong \Hom_\mathbb{K}\left(A_\mathbb{K}, B_\mathbb{K}\otimes C_\mathbb{K}\right)$) for $\mathbb{K}=\QQ,\CC$ come from their $\mathbb{K}=\ZZ$ counterpart by extending the scalars, and then checks that the isomorphisms respect the filtrations $W$ and $F$. Parts (c) and (e) are immediate from that $F^nA(-p)_\CC=F^{n-p}A_\CC$. Part (d) (resp. (f)) is a special case of (c) (resp. (e)).

\subsection{Carlson's theorem on classifying extensions in $\mathbf{MHS}$}\label{Carlsoniso}
Let $A$ and $B$ be mixed Hodge structures. By $\Ext(A,B)$ we mean the group of extensions of $A$ by $B$ in the category $\mathbf{MHS}$. Suppose the highest weight of $B$ is less than the lowest weight of $A$. Carlson \cite{Carlson} gave a functorial isomorphism 
\[\text{Ext}(A,B)\cong J^0\underline{\text{Hom}}(A,B).\]
Given an extension $\mathbb{E}$ given by a short exact sequence
\begin{center}
\begin{tikzpicture}
  \matrix (m) [matrix of math nodes, column sep=2em]
    { 0 & B  & E  & A  & 0~,\\};
  { [start chain] \chainin (m-1-1);
    \chainin (m-1-2);
    \chainin (m-1-3);
    \chainin (m-1-4);
    \chainin (m-1-5);}
  \end{tikzpicture}
\end{center}
one way to describe the corresponding element in the Jacobian is as follows: Choose a Hodge section $\sigma_F$ of $E_{\CC}\rightarrow A_{\CC}$, and an integral retraction (i.e. left inverse) $\rho_{\ZZ}$ of $B_{\CC}\rightarrow E_{\CC}$. The extension $\mathbb{E}$ corresponds to the class of $\rho_{\ZZ}\circ\sigma_F$. (By a Hodge section we mean a section that is compatible with the Hodge filtrations, and by integral we mean a map that is induced by a map between the underlying $\ZZ$-modules.)\\

In the interest of simplifying notation, we shall identify $\text{Ext}(A,B)$ and $J^0\underline{\text{Hom}}(A,B)$ via the isomorphism of Carlson.

\subsection{Cohomology of a complex variety}\label{cohomologyvariety} 

Let $U$ be a complex variety. If $U$ is smooth and projective, its degree $n$ cohomology is a pure Hodge structure of weight $n$: The underlying abelian group is the Betti (singular) cohomology group $H^n(U,\ZZ)$ ($U$ with analytic topology). Identifying
\begin{equation}\label{eq1rev2ch2}
H^n(U,\ZZ)\otimes \CC\cong H^n(U,\CC)\stackrel{\text{de Rham iso.}}{\cong} H^n_{dR}(U),
\end{equation}
where $H^n_{dR}(U)$ here is the $n^\text{th}$ cohomology of the complex $E^\cdot_\CC(U)$ of complex-valued smooth differential forms on $U$, the Hodge decomposition is given by the classical result
\[
H^n_{dR}(U)=\bigoplus\limits_{p+q=n} H^{p,q},
\] 
where elements of $H^{p,q}$ are represented by forms of type $(p,q)$.\\

More generally, thanks to a theorem of Deligne, the degree $n$ cohomology of $U$ (which is no longer assumed to be projective or smooth), naturally carries a mixed Hodge structure, which we denote by $H^n(U)$. If $U$ is smooth, $H^n(U)$ can be described as follows: The underlying abelian group is again Betti cohomology with integral coefficients. Via the identifications \eqref{eq1rev2ch2}, we define the weight and Hodge filtrations on $H^n_{dR}(U)$. Realize $U$ as $Y\setminus D$, where $Y$ is smooth projective and $D$ is a normal crossing divisor. Then the complex $E^\cdot(Y\log D)$ of smooth differential forms on $U$ with at most logarithmic singularity along $D$ calculates the cohomology of $U$, i.e. the inclusion
\[
E^\cdot(Y\log D)\hookrightarrow E^\cdot_\CC(U)
\]  
is a quasi-isomorphism. One defines a filtration $F^\cdot$ (resp. $W_\cdot$) on the complex $E^\cdot(Y\log D)$ by holomorphic degree (resp. order of poles along $D$). Then the Hodge and Weight filtration on $H^n_{dR}(U)$ are given by
\[
F^pH^n_{dR}(U)=\Im\bigm(H^nF^pE^\cdot(Y\log D)\longrightarrow H^n_{dR}(U)\bigm)
\]    
and 
\[
W_lH^n_{dR}(U)=\Im \bigm(H^nW_{l-n}E^\cdot(Y\log D)\longrightarrow H^n_{dR}(U)\bigm).
\]
One can show that $W_\cdot$ is defined over $\QQ$, i.e. is induced by a filtration on
\[
H^n(U,\QQ)\subset H^n(U,\CC)\cong H^n_{dR}(U),
\]
and that the structure just defined only depends on $U$ (and not the compactification used in the process).\\

For references on mixed Hodge structures on the cohomology of a variety, the original articles are Deligne's \cite{HT2} (for the smooth case) and \cite{HT3} (for the general case). The reader can also consult \cite{Morgan} and \cite{Voisin}. For more details on the complex $E^\cdot(Y\log D)$ see \cite{Morgan}.\\

Throughout the paper, our varieties will all be smooth and for such a variety $U$ we continue to identify $H^n(U,\CC)$ and $H^n_{dR}(U)$ via the isomorphism of de Rham.

\section{Hodge Theory of $\pi_1$- Recollections from the general theory}\label{ch1}

\subsection{Review of the reduced bar construction}

In this paragraph, we briefly review certain aspects of the reduced bar construction on a differential graded algebra. The construction is due to K.T. Chen, and the reader can refer to \cite{Chen2} and \cite{Hain1} for references. We only discuss a special case that is of interest to us. Throughout this paragraph $\mathbb{K}$ is a field of characteristic 0.\\

By a differential graded algebra over $\mathbb{K}$ we mean one that is concentrated in degree $\geq0$. More precisely, this is a graded $\mathbb{K}$-algebra $\displaystyle{A^\cdot=\bigoplus\limits_{n\geq 0} A^n}$, equipped with a differential $d$ of degree 1 (so that one has a complex    
\[
A^0\stackrel{d}{\longrightarrow}A^1\stackrel{d}{\longrightarrow}A^2\stackrel{d}{\longrightarrow}\cdots
\] 
of $\mathbb{K}$-vector spaces) such that the graded Leibniz rule holds, i.e.
\[
d(ab)=(da)b+(-1)^{\deg(a)}a(db)
\] 
for homogeneous elements $a,b\in A^\cdot$, where $\deg$ is the degree. Moreover, we say $A^\cdot$ is commutative if 
\[
ab=(-1)^{\deg(a)\deg(b)}ba
\]
for all homogeneous $a,b$.\\  

Note that $\mathbb{K}$ itself can be thought of as a differential graded algebra over $\mathbb{K}$ in an obvious way. Suppose $\displaystyle{A^\cdot=\bigoplus\limits_{n\geq 0} A^n}$ is a differential graded algebra over $\mathbb{K}$, with the differential denoted by $d$. Denote the positive degree part by $A^+$. Let $\epsilon:A^\cdot\rightarrow \mathbb{K}$ be an augmentation (i.e. a morphism of differential graded algebras in to $\mathbb{K}$). For any integers $r,s$ ($r\geq0$), let $T^{-r,s}(A^\cdot)$ be the degree $s$ part of $(A^+)^{\otimes r}$, i.e. the $\mathbb{K}$-span of all terms of the form 
\begin{equation}\label{eq3backgroundrevision}
a_1\otimes\ldots\otimes a_r,
\end{equation}   
where $a_i\in A^+$ and $\sum \deg a_i=s$. (By convention, $(A^+)^{\otimes 0}=\mathbb{K}$.) It is customary to use the notation 
\[
[a_1|\ldots|a_r]
\]
for the element \eqref{eq3backgroundrevision}. The $T^{-r,s}(A^\cdot)$ form a second quadrant bicomplex $T^{\cdot,\cdot}(A^\cdot)$, with $T^{-r,s}(A^\cdot)$ being the $(-r,s)$ bidegree component, and anti-commuting differentials both of degree 1 defined below. Here $Ja=(-1)^{\deg a}a$ for any homogeneous element $a\in A^\cdot$.  
\begin{itemize}
\item The horizontal differential $d_h$:
\[
d_h([a_1|\ldots|a_r])= \sum\limits_{i=1}^{r-1}(-1)^{i+1}[Ja_1|\ldots|Ja_{i-1}|(Ja_i)a_{i+1}|a_{i+2}|\ldots|a_r]
\]
\item The vertical differential $d_v$:
\[
d_v([a_1|\ldots|a_r])=\sum\limits_{i=1}^{r}(-1)^i[Ja_1|\ldots|Ja_{i-1}|da_i|a_{i+1}|\ldots|a_r].
\]
\end{itemize}
\noindent The formulas for the differentials are particularly important for us when all the $a_i$ are of degree 1. In this case the formulas simplify to 
\begin{equation}\label{eq_dh}
d_h[a_1\ldots|a_r]=-\sum\limits_i[a_1|\ldots|a_ia_{i+1}|\ldots|a_r]
\end{equation}
and
\begin{equation}\label{eq_dv}
d_v[a_1\ldots|a_r]=-\sum\limits_i [a_1|\ldots|da_i|\ldots|a_r].
\end{equation}
 
\noindent The associated total complex $\text{Tot}\left(T^{\cdot,\cdot}(A^\cdot)\right)$ is concentrated in non-negative degrees, and its degree zero part is $\bigoplus\limits_{s\geq 0}T^{-s,s}(A^\cdot)=\bigoplus\limits_{s\geq 0} (A^1)^{\otimes s}$. The reduced bar construction $\overline{B}(A^{\cdot},\epsilon)=\bigoplus\limits_{n\geq 0}\overline{B}^n(A^{\cdot},\epsilon)$ of $A^{\cdot}$ relative to $\epsilon$ is by definition a certain quotient of $\text{Tot}\left(T^{\cdot,\cdot}(A^\cdot)\right)$, where the subcomplex by which one quotients depends on $\epsilon$. The image of $[a_1|\ldots|a_r]$ is denoted by $(a_1|\ldots|a_r)$. If $A^0=\mathbb{K}$, then $\overline{B}(A^{\cdot},\epsilon)$ is simply $\text{Tot}\left(T^{\cdot,\cdot}(A^\cdot)\right)$. From now on we drop the augmentation $\epsilon$ from our notation for $\overline{B}$ if it will not lead to any confusion.\\

The reduced bar construction is naturally filtered by tensor length: Let 
\[\mathcal{T}_{n}=\bigoplus\limits_{r\leq n}\left(T^{-r,s}(A^\cdot)\right).\]
The filtration $\{\mathcal{T}_n\}$ of the double complex $\left(T^{\cdot,\cdot}(A^\cdot)\right)$ induces a filtration $\{\mathcal{B}_n\}$ on the reduced bar construction. We denote the filtration induced on the cohomology of $\overline{B}(A^{\cdot})$ also by $\{\mathcal{B}_n\}$.\\ 

The reduced bar construction is functorial. In particular, if $A^\cdot$ and $\tilde{A}^\cdot$ are differential graded $\mathbb{K}$-algebras, and $\epsilon:A^\cdot\rightarrow \mathbb{K}$ and $\tilde{\epsilon}:\tilde{A}^\cdot\rightarrow \mathbb{K}$ are augmentations, a morphism $f:A^\cdot\rightarrow \tilde{A}^\cdot$ of differential graded algebras satisfying $\tilde{\epsilon}\circ f=\epsilon$ induces a morphisms of complexes $\overline{B}(A^\cdot)\rightarrow \overline{B}(\tilde{A}^\cdot)$ compatible with the filtrations $\{\mathcal{B}_n\}$. Moreover, if $f$ is a quasi-isomorphism and $H^0(A^\cdot)=\mathbb{K}$, then the induced maps between the reduced bar constructions or the $\mathcal{B}_n$ are also quasi-isomorphisms.\\ 

If $A^\cdot$ is commutative, then $\overline{B}(A^{\cdot})$ is in fact a commutative differential graded algebra\footnote[2]{Actually it is a commutative Hopf algebra, with comultiplication defined by
\[
(a_1|\ldots|a_r)\mapsto \sum_i (a_1|\ldots|a_i)\otimes(a_{i+1}|\ldots|a_r).
\]
We shall not explicitly work with the coalgebra structure in this paper.}, with multiplication given by the so called {\it shuffle product}. For degree zero elements, the multiplication is given by the formula\footnote[3]{Recall that $\sigma\in S_{r+s}$ is an (r,s) shuffle if 
\[
\sigma^{-1}(1)<\cdots<\sigma^{-1}(r)\hspace{.1in}\text{and}\hspace{.1in} \sigma^{-1}(r+1)<\cdots<\sigma^{-1}(r+s)
\].}    
\[
(a_1|\ldots|a_r)\cdot (a_{r+1}|\ldots|a_{r+s})=\sum\limits_{\text{(r,s) shuffles $\sigma$}}(a_{\sigma(1)}|\ldots|a_{\sigma(r+s)}).
\]
The general formula is an alternating sum of the $(a_{\sigma(1)}|\ldots|a_{\sigma(r+s)})$, where the coefficients take into account the signs of the $\sigma$ and the degrees of the $a_i$. In particular, when $A^\cdot$ is commutative, $H^0\overline{B}(A^{\cdot})$ is also a commutative algebra. If $f:A^\cdot\rightarrow B^\cdot$ is a morphism of commutative differential graded algebras, then the induced map between the reduced bar constructions respects the multiplications.

\subsection{} Let $G$ be a finitely generated group and $\mathbb{K}$ a field of characteristic zero. The Malcev or (pro)-unipotent completion of $G$ over $\mathbb{K}$ is a pro-unipotent algebraic group $G_{\mathbb{K}}^{un}$ over $\mathbb{K}$, together with a homomorphism $G\rightarrow G_{\mathbb{K}}^{un}(\mathbb{K})$, such that for any pro-unipotent group $U$ over $\mathbb{K}$ and any homomorphism $G\rightarrow U(\mathbb{K})$, there is a unique morphism $G_{\mathbb{K}}^{un}\rightarrow U$ of group schemes over $\mathbb{K}$ making the obvious diagram commute. It follows immediately that the image of $G$ is dense in $G_{\mathbb{K}}^{un}$. The group $G_{\mathbb{K}}^{un}$ can be defined explicitly as $\Spec(\mathcal{O}_{G_\mathbb{K}^{un}})$, where
\[\mathcal{O}_{G_\mathbb{K}^{un}}=\lim_{\longrightarrow}\left(\frac{\mathbb{K}[G]}{I^{m+1}}\right)^\vee,\]
and $I$ is the augmentation ideal. One can think of 
\[\left(\frac{\mathbb{K}[G]}{I^{m+1}}\right)^\vee\] 
as the space of $\mathbb{K}$-valued functions on $G$ which (after being extended linearly to $\mathbb{K}[G]$) vanish on $I^{m+1}$. For the very last sentence, $\mathbb{K}$ can be a ring.

\subsection{Chen's theory of iterated integrals and the description of $\mathcal{O}({\pi_1}_{\CC}^{un})$}\label{backgrounditint}
We review some results of K.T Chen in this paragraph. For details and proofs, see \cite{Chen1}, \cite{Chen2} and \cite{Chen3}. Throughout this paragraph $\mathbb{K}\in\{\RR,\CC\}$.\\

As a generalization of the notion of a manifold, Chen in \cite{Chen3} defines the notion of a differentiable space. He associates to each differentiable space a {\it de Rham} commutative differential graded algebra of {\it $\mathbb{K}$-valued differential forms}. The degree $0$ forms are, as expected, ``differentiable" functions, and the multiplication on them is simply point-wise multiplication of functions.\\

Let $U$ be a path-connected (smooth) manifold, $e\in U$, and $\Omega_e$ be the (smooth) loop space at $e$. Let $E_\mathbb{K}^\cdot(U)$ be the complex of $\mathbb{K}$-valued differential forms on $U$. The loop space $\Omega_e$ is naturally made into a differentiable space. For every $\omega_1,\ldots,\omega_r\in E_\mathbb{K}^\cdot(U)$ of positive degree, Chen defines a $d$-form on $\Omega_e$ denoted by $\int\omega_1\ldots\omega_r$. A $\mathbb{K}$-valued {\it iterated integral of degree $d$} is by definition a linear combination of the $d$-forms of the form $\int\omega_1\ldots\omega_r$. In the case that $\omega_1,\ldots,\omega_r$ are all 1-forms on $U$, the zero form, i.e. function, $\int\omega_1\ldots\omega_r$ on the loop space is defined by
\[
\biggm(\gamma:[0,1]\rightarrow U\biggm)~\mapsto\hspace{.5in} \int\limits_{0\leq t_1\leq\ldots\leq t_r\leq 1} f_1(t_1)dt_1 \ldots f_r(t_r)dt_r, 
\] 
where $f_i(t)dt=\gamma^\ast(\omega_i)$. If $r=0$, the ``empty" iterated integral is defined to be the constant function 1. The value of $\displaystyle{\int\omega_1\ldots\omega_r}$ on $\gamma$ is denoted by $\displaystyle{\int\limits_\gamma\omega_1\ldots\omega_r}$. It is clear that for $r=1$, this coincides with the usual integral.\\

Following \cite{Chen2}, we denote the space of  $\mathbb{K}$-valued iterated integral of degree $d$ by ${A'_\mathbb{K}}^d$. The space $A'_\mathbb{K}:=\oplus {A'_\mathbb{K}}^d$ is a sub-complex of the de Rham complex on the loop space $\Omega_e$. It is also closed under multiplication (and hence is a sub-differential graded algebra). For degree 0 iterated integrals, this is thanks to the so-called {\it shuffle product} property given by the formula     
\begin{equation}\label{eqshuffleitint}
\int\limits_\gamma\omega_1\ldots\omega_r \int\limits_\gamma\omega_{r+1}\ldots\omega_{r+s}=\sum\limits_{\text{(r,s) shuffles $\sigma$}}\int\limits_\gamma\omega_{\sigma(1)}\ldots\omega_{\sigma(r+s)},
\end{equation}
where $\gamma$ is a loop at $e$.\\

An element of $A'^d_\mathbb{K}$ that can be expressed as a linear combination of $\int\omega_1\ldots\omega_r$ with $r\leq m$ is said to be of length $\leq m$. The elements of $A'_\mathbb{K}$ of length $\leq m$ form a subcomplex $A'_\mathbb{K}(m)$. The complex $A'_\mathbb{K}$ is naturally filtered by length. Since $A'_\mathbb{K}$ is concentrated in degree $\geq 0$, one has
\[
H^0(A'_\mathbb{K}(m))\subset H^0(A'_\mathbb{K}),
\]
and the $\{H^0(A'_\mathbb{K}(m))\}$ is a filtration on $H^0(A'_\mathbb{K})$.\\

From now on, by an iterated integral we mean one of degree zero. The following formula describes how iterated integrals behave relative to composition of paths. Here $\alpha$ and $\beta$ are loops at $e$.
\begin{equation}\label{itintprop}
\int\limits_{\alpha\beta}\omega_1\ldots\omega_r=\sum\limits_{i=0}^{r}\int\limits_\alpha\omega_1\ldots\omega_i~\int\limits_\beta\omega_{i+1}\ldots\omega_r
\end{equation}

\noindent One can show that iterated integrals also satisfy the following relations (as functions on $\Omega_e$). Here $f$ is a (smooth) function on $U$.
\begin{eqnarray}
\int(df)\omega_2\ldots\omega_r&=&\int(f\omega_2)\ldots\omega_r-f(e)\int\omega_2\ldots\omega_r\notag\\
\int\omega_1\ldots\omega_{i-1}(df)\omega_{i+1}\ldots\omega_r&=&\int\omega_1\ldots\omega_{i-1}(f\omega_{i+1})\ldots\omega_r-\int\omega_1\ldots(f\omega_{i-1})\omega_{i+1}\ldots\omega_r\notag\\
\int\omega_1\ldots\omega_{r-1}(df)&=&f(e)\int\omega_1\ldots\omega_{r-1}-\int\omega_1\ldots(f\omega_{r-1})\label{itrel}
\end{eqnarray}

An iterated integral induces a function on $G=\pi_1(U,e)$ if and only if it is locally constant on the loop space if and only if it is closed (as an element of the complex $A'_\mathbb{K}$). It follows from \eqref{itintprop} that a closed iterated integral of length $\leq m$ vanishes on $I^{m+1}\subset \mathbb{K}[G]$, so that one has a natural inclusion 
\[H^0(A'_\mathbb{K}(m))\subset\left(\frac{\mathbb{K}[G]}{I^{m+1}}\right)^\vee.\] 
The main theorem of \cite{Chen1} (Theorem 5.3) asserts that indeed 
\[H^0(A'_\mathbb{K}(m))=\left(\frac{\mathbb{K}[G]}{I^{m+1}}\right)^\vee.\]

The algebraic structure of $H^0(A'_\mathbb{K}(m))$ can be described using the reduced bar construction on the complex $E^\cdot_\mathbb{K}(U)$ of smooth $\mathbb{K}$-valued differential forms on $U$, augmented by ``evaluation at $e$". One has a natural map of differential graded algebras $\overline{B}(E^\cdot_\mathbb{K}(U))\rightarrow A'_\mathbb{K}$ given by integration 
\[
(\omega_1|\ldots|\omega_r)~\mapsto~\int\omega_1\ldots\omega_r.
\]
This map\footnote[2]{The relations by which one mods out $\text{Tot}\left(T^{\cdot,\cdot}(E^\cdot_\mathbb{K}(U)\right)$ to get $\overline{B}(E^\cdot_\mathbb{K}(U))$ are defined exactly based on relations \eqref{itrel} satisfied by iterated integrals, so that the map just described is well-defined.} induces an isomorphism $H^0\overline{B}(E^\cdot_\mathbb{K}(U))\rightarrow H^0(A'_\mathbb{K})$ strictly compatible with the length filtrations, i.e. we have a natural isomorphism
\[
\mathcal{B}_mH^0\overline{B}(E^\cdot_\mathbb{K}(U))\stackrel{\int}{\rightarrow}H^0(A'_\mathbb{K}(m))=\left(\frac{\mathbb{K}[G]}{I^{m+1}}\right)^\vee. 
\]

\begin{rem} If $U$ is (the associated complex manifold to) a smooth complex variety, and $U=Y\setminus D$ where $Y$ is smooth projective and $D$ is a normal crossing divisor, one can replace $E^\cdot_\CC(U)$ by the complex $E^\cdot(Y\log D)$. (See Paragraph \ref{cohomologyvariety}.)
\end{rem}

\subsection{Mixed Hodge structure on $\pi_1$ of a smooth complex variety}\label{par4} 
Let $U$ be a smooth variety over $\CC$, $e\in U(\CC)$, $G=\pi_1(U,e)$, where with abuse of notation we denote a smooth complex variety and its associated complex manifold by the same symbol. Here we briefly recall Hain's mixed Hodge structure on the integral lattice 
\[\left(\frac{\ZZ[G]}{I^{m+1}}\right)^\vee,\]
which we denote by $L_m=L_m(U,e)$. For details and proofs, see \cite{Hain1}.\\

Let $U=Y\setminus D$, where $Y$ is a smooth projective variety and $D$ is a normal crossing divisor. In view of the isomorphism
\[\mathcal{B}_mH^0\overline{B}(E^\cdot(Y\log D))  \stackrel{\int}{\longrightarrow} \displaystyle{\left(\frac{\CC[G]}{I^{m+1}}\right)^\vee={(L_m)}_\CC}\]
the weight and Hodge filtrations on $L_m$ are described as follows:\\ 

\begin{itemize}
\item The weight filtration: $W_n{(L_m)}_\CC$ is the space of those closed iterated integrals that can be expressed as a sum of (not necessarily closed) iterated integrals of the form $\int\omega_1\ldots\omega_r$, with $r\leq m$ and $\omega_i\in E^1(Y\log D)$, such that at most $n-r$ of the $\omega_i$ are not smooth along $D$. One can prove that this filtration is indeed defined over $\QQ$. It is easy to see that $W_n(L_m)\subset L_n$.\\
\item The Hodge filtration: $F^p(L_m)_\CC$ is the space of those closed iterated integrals that can be expressed as a sum of (not necessarily closed) iterated integrals of the form $\int\omega_1\ldots\omega_r$, where $r\leq m$ and $\omega_i\in E^1(Y\log D)$, such that at least $p$ of the $\omega_i$ are of type (1,0).
\end{itemize}

Note that the $L_m$ form a direct system of mixed Hodge structures.

\begin{rem} (1) One can show that $L_m$ only depends on the pair $(U,e)$, and not on the embedding of $U$ as $Y\setminus D$. As in the case of mixed Hodge structure on cohomology, to explicitly describe the Hodge and weight filtrations on $L_m$ one usually embeds $U$ as $Y\setminus D$ as above.\\
(2) $L_m(U,e)$ is functorial in $(U,e)$.
\end{rem}

\subsection{De Rham lattice in $\mathcal{O}(\pi_1^{un})$ and periods of the fundamental group}

Let $K$ be a subfield of $\CC$, $U_0$ be a smooth variety over $K$, $e\in U_0(K)$, and $U=U_0\otimes_K\CC$. We assume moreover that $U_0$ is affine. Let $\Omega^\cdot(U_0)$ (resp. $\Omega^\cdot(U)$) be the complex of global (regular) differential forms on $U_0$ (resp. $U$). Since $U$ is affine, the complex $\Omega^\cdot(U)$ calculates the cohomology of $U$. More precisely, the natural map
\[
\Omega^\cdot(U_0)\otimes_K\CC=\Omega^\cdot(U)\rightarrow E^\cdot(U)
\]
is a quasi-isomorphism. It follows that one has a natural isomorphism
\[
H^0\overline{B}(\Omega^\cdot(U))\cong H^0\overline{B}(E^\cdot(U))
\]
strictly compatible with the filtrations. The de Rham fundamental group $\pi_1^{dR}(U_0,e)$ of $U_0$ with base point $e$ is an affine group scheme over $K$ with coordinate ring 
\[
\mathcal{O}(\pi_1^{dR}(U_0,e))=H^0\overline{B}(\Omega^\cdot(U_0)). 
\]
We refer to the image of $\mathcal{B}_nH^0\overline{B}(\Omega^\cdot(U_0))$ under
\[
\mathcal{B}_nH^0\overline{B}(\Omega^\cdot(U_0))\subset \mathcal{B}_nH^0\overline{B}(\Omega^\cdot(U))\cong \mathcal{B}_nH^0\overline{B}(E^\cdot(U))\stackrel{\int}{\cong} L_n(U,e)_\CC
\]
as the {\it de Rham lattice} in $L_n(U,e)_\CC$. It is easy to see that it is the space of all iterated integrals of length $\leq n$ formed by elements of $\Omega^1(U_0)$. (They will automatically be closed.) We use the notation $L_n(U_0,e)$ to refer to $L_n(U,e)$ together with the data of the de Rham lattice. The space of {\it periods} of $\pi_1(U_0,e)$ is the $K$-span of all the numbers of the form
\[
\int_\gamma \omega_1\ldots\omega_r,
\] 
where the $\omega_i$ are in $\Omega^1(U_0)$ and $\gamma\in\pi_1(U,e)$. The subspace generated by those integrals above with $r\leq n$ is the space of periods of $L_n(U_0,e)$.

\section{Construction of certain elements in the Bar construction}\label{bar2}

In this section, given an augmented differential graded algebra satisfying certain properties, we give a procedure that constructs elements is $H^0\overline{B}$ with prescribed highest length terms. This construction will be particularly important in Section \ref{Hodgetheorypi1}.\\ 

We assume that $A^\cdot$ is an augmented differential graded algebra, and that
\begin{itemize}
\item[(i)] $d(A^1)=(A^1)^2$,
\item[(ii)] for each pair $(a,b)$ of elements of $A^1$, $s(a,b)\in A^1$ is such that $d(s(a,b))=-ab$.
\end{itemize}
Let $a_1,\ldots,a_n\in A^1$ be closed. Our goal is to give a closed element of $\overline{B}^0(A^\cdot)$ of the form
\[
(a_1|\ldots|a_n)~+~\text{lower length terms}.
\]
For this, it suffices to construct a closed element of $\oplus T^{-r,r}(A^\cdot)$ of the form
\[
[a_1|\ldots|a_n]~+~\text{lower length terms}.
\] 
Set $\lambda_n=[a_1|\ldots|a_n]$. Then $d_v(\lambda_n)=0$, and $d_h(\lambda_n)\in T^{-n+1,n}$. The idea is to define, for each $r=n-1,\ldots,1$, an element $\lambda_r\in T^{-r,r}$ such that $d_v(\lambda_r)=-d_h(\lambda_{r+1})$. The element 
\[
\lambda_n+\lambda_{n-1}+\ldots+\lambda_1
\]
will then be closed.\\
For $r=n-1,\ldots,1$, define $\lambda_r$ to be the sum of all simple tensors in $T^{-r,r}$ of the form
\begin{equation}\label{eq1geometryrevision}
[...|...|...|...~...~...|...],
\end{equation}
where each block is formed by (possibly 0) successions of $s(~,~)$, and such that when we remove the symbols ``$|$" and ``$s(~,~)$", we are left with 
\begin{equation}\label{eq2geometryrevision}
[a_1~a_2~\ldots~a_n].
\end{equation}    
For example,
\[
\lambda_{n-1}=\sum\limits_{i=1}^{n-1}[a_1|\ldots|s(a_i,a_{i+1})|\ldots|a_n],
\]
and
\begin{eqnarray*}
\lambda_{n-2}~&=&\sum\limits_{1\leq i<j-1\leq n-2} [a_1|\ldots|s(a_i,a_{i+1})|\ldots|s(a_j,a_{j+1})|\ldots|a_n]\\
              &+&\sum\limits_{i=1}^{n-2} [a_1|\ldots|s(s(a_i,a_{i+1}),a_{i+2})|\ldots|a_n]\\
							&+&\sum\limits_{i=1}^{n-2} [a_1|\ldots|s(a_i,s(a_{i+1},a_{i+2}))|\ldots|a_n].
\end{eqnarray*}
There will be much more variety for $\lambda_{n-3}$: 
\begin{eqnarray*}
\lambda_{n-3}~=&&\sum\limits [a_1|\ldots|s(a_i,a_{i+1})|\ldots|s(a_j,a_{j+1})|\ldots|s(a_k,a_{k+1})|\ldots|a_n]\\
              &+&\sum [a_1|\ldots|s(a_i,a_{i+1})|\ldots|s(s(a_j,a_{j+1}),a_{j+2})|\ldots|a_n]\\
							&+&\sum [a_1|\ldots|s(a_i,a_{i+1})|\ldots|s(a_j,s(a_{j+1},a_{j+2}))|\ldots|a_n]\\
							&+&\sum [a_1|\ldots|s(s(a_i,a_{i+1}),a_{i+2})|\ldots|s(a_j,a_{j+1})|\ldots|a_n]\\
						  &+&\sum [a_1|\ldots|s(a_i,s(a_{i+1},a_{i+2}))|\ldots|s(a_j,a_{j+1})|\ldots|a_n]\\
							&+&\sum [a_1|\ldots|s(s(s(a_i,a_{i+1}),a_{i+2}),a_{i+3})|\ldots|a_n]\\
							&+&\sum [a_1|\ldots|s(s(a_i,s(a_{i+1},a_{i+2})),a_{i+3})|\ldots|a_n]\\
              &+&\sum [a_1|\ldots|s(a_i,s(s(a_{i+1},a_{i+2}),a_{i+3}))|\ldots|a_n]\\
							&+&\sum [a_1|\ldots|s(a_i,s(a_{i+1},s(a_{i+2},a_{i+3})))|\ldots|a_n]\\
							&+&\sum [a_1|\ldots|s(s(a_i,a_{i+1}),s(a_{i+2},a_{i+3}))|\ldots|a_n].
\end{eqnarray*}
Note that in every summand of $\lambda_r$, there are exactly $n-r$ occurrences of $s$.

\begin{lemma}
The element $\lambda_n+\ldots+\lambda_1$ is closed.
\end{lemma}

\begin{proof}
Note that $d_v(\lambda_n)=d_h(\lambda_1)=0$. It remains to check that for each $r$, $-d_h(\lambda_{r+1})=d_v(\lambda_r)$. But in view of the formulas \eqref{eq_dh} and \eqref{eq_dv}, both $-d_h(\lambda_{r+1})$ and $d_v(\lambda_r)$ are the sum of all simple tensors in $T^{-r,r+1}$ of the form \eqref{eq1geometryrevision} where each block is formed by (possibly 0) successions of $s(~,~)$, and such that when we remove the symbols ``$|$" and ``$s(~,~)$", we are left with \eqref{eq2geometryrevision}. That each $a_i$ is closed is important to make sure $d_v(\lambda_r)$ is equal to the aforementioned sum.
\end{proof}

\begin{rem} It is easy to see that if $s:A^1\times A^1\rightarrow A^1$ is bilinear, then the above construction gives a linear map $(A^1_\text{closed})^{\otimes n}\rightarrow \mathcal{B}_nH^0\overline{B}(A)$.
\end{rem}

\section{Hodge Theory of $\pi_1$- The case of a punctured curve}\label{Hodgetheorypi1}
From here until the end of the paper, $X$ is a smooth (connected) projective curve over $\CC$ of genus $g$, and $\infty,e\in X(\CC)$ are distinct points. Our main objective in this section is to construct a map (see Lemma \ref{s_F}) which will play a crucial role later on. 

\subsection{}\label{Hodge_filt_L_n_paragraph} Let $S\subset X(\CC)$ be of finite cardinality $|S|\geq1$, $U=X-S$, and $e\in U(\CC)$. Let $G=\pi_1(U,e)$ and $L_m=L_m(U,e)$. Our goal in this paragraph is to study $(L_m)_\CC$ more closely.\\

It is well-known that in this case there are holomorphic differential forms $\alpha_i$ ($1\leq i\leq 2g+|S|-1$) on $U$ whose classes form a basis of $H^1_{dR}(U)$. We can, and will, take these such that $\alpha_1,\ldots,\alpha_{g}$ are of first kind (i.e. holomorphic on $X$), $\alpha_{g+1},\ldots\alpha_{2g}$ are of second kind (i.e. meromorphic on $X$ with zero residue along $S$), and $\alpha_{2g+1},\ldots,\alpha_{2g+|S|-1}$ are of third kind with simple poles at points in $S$. Let $R^\cdot$ be the sub-object of $E^\cdot_\CC(U)$ given by $R^0=\CC$, $R^1=\sum\limits_{i=1}^{2g+|S|-1}\alpha_i\CC$, and $R^2=0$. The inclusion map $R^\cdot\rightarrow E^\cdot_\CC(U)$ is a quasi-isomorphism, so that in particular
\[
\mathcal{B}_mH^0\overline{B}(R^\cdot)\cong \mathcal{B}_mH^0\overline{B}(E^\cdot_\CC(U))\hspace{.5in}\text{and}\hspace{.5in}H^0\overline{B}(R^\cdot)\cong H^0\overline{B}(E^\cdot_\CC(U)).
\]
It is easy to see that $H^0\overline{B}(R^\cdot)$, as a vector space, is the (underlying vector space of the) tensor algebra on $R^1$, and the multiplication is the shuffle product. In other words, $H^0\overline{B}(R^\cdot)$ is the shuffle algebra on the letters $\alpha_i$ ($1\leq i\leq 2g+|S|-1$). The filtration $\mathcal{B}_\cdot$ is the tensor length filtration. The following description of $L_m$ is now immediate.         

\begin{prop}\label{L_nusingholforms} 
The integration map $\displaystyle{H^0\overline{B}(R^\cdot)\rightarrow\lim\limits_{\longrightarrow}\left(\frac{\CC[G]}{I^{m+1}}\right)^\vee}$ which maps
\[
[\alpha_{i_1}|\ldots|\alpha_{i_r}]\mapsto \int\alpha_{i_1}\ldots\alpha_{i_r} 
\] 
is an isomorphism, which maps $\mathcal{B}_m$ onto $(L_m)_\CC$. In particular, any complex valued function on $G$ that (after extending linearly to $\CC[G]$) vanishes on $I^{m+1}$ is given by a unique (linear combination of) iterated integral(s) of length $\leq m$ in the forms $\alpha_i$.
\end{prop}  

\subsection{}  From now on, let $S=\{\infty\}$. (Thus $U=X-\{\infty\}$ and $L_n=L_n(X-\{\infty\},e)$.) The complex $F^1E^\cdot(X\log\infty)$ is exact in degree 2. For each $a,a'\in E^1(X\log\infty)$, let $s(a,a')\in F^1E^1(X\log\infty)$ be such that $d(s(a,a'))=-a\wedge a'$. If $a\wedge a'=0$, we specifically take $s(a,a')=0$.\\

\noindent The differential graded algebra $E^\cdot(X\log\infty)$ meets the condition of Section \ref{bar2}, and hence for $\omega_1,\ldots,\omega_n$ closed smooth 1-forms on $X$, the construction given in that section gives us a closed element of $\overline{B}^0E^\cdot(X\log\infty)$ of the form
\[
(\omega_1|\ldots|\omega_n)+\text{lower length terms},
\] 
and thus a closed iterated integral on $X-\{\infty\}$ of the form
\begin{equation}\label{lift1}
\int \omega_1\ldots\omega_n+\text{lower length terms},
\end{equation}  
where all the 1-forms involved are in $E^1(X\log\infty)$. Moreover, by construction, in each term of length $r$ above there are $n-r$ occurrences of $s$, and hence at most $n-r$ forms with a pole at $\infty$. In view of the description of the weight filtration given in Paragraph \ref{par4}, this implies the following lemma.

\begin{lemma}\label{lem3.1}
Given closed smooth 1-forms $\omega_1,\ldots,\omega_n$ on $X$, there is an element of $W_n(L_n)_\CC$ of the form \eqref{lift1}.
\end{lemma}

\subsection{} \underline{The Weight Filtration of $L_m$}: We now show that the weight filtration on $L_m$ coincides with the length filtration.

\begin{prop}\label{weightdescription}
For $n\leq m$, $W_nL_m=L_n$.
\end{prop} 

\begin{proof}
It is enough to show $W_nL_n=L_n$ for all $n$, for then, if $n\leq m$, we see in view of $W_nL_m\subset L_n$ that $W_nL_m=L_n$. We argue by induction on $n$. This is trivial for $n=0$. Suppose $W_{n-1}L_{n-1}=L_{n-1}$. In view of Proposition \ref{L_nusingholforms}, it suffices to show that 
\[
\int \alpha_{j_1}\ldots\alpha_{j_n}\in W_n(L_n)_\CC.
\] 
For each $i$, let $\omega_i\in E_\CC^1(X)$ be such that $\alpha_{j_i}=\omega_i+df_i$ on $U$, where $f_i$ is a smooth function on $U$; this can be done because inclusion of $U$ in $X$ gives an isomorphism on the level of $H^1$. Thanks to the relations \eqref{itrel} satisfied by iterated integrals, we have 
\[
\int \alpha_{j_1}\ldots\alpha_{j_n}=\int\omega_1\ldots\omega_n + \text{lower length terms}. 
\]
In view of Lemma \ref{lem3.1} we can write
\begin{eqnarray*}
\int \alpha_{j_1}\ldots\alpha_{j_n}=&&\text{\bigg(an element of $W_n(L_n)_\CC$ of the form}\\
&&\text{$\int\omega_1\ldots\omega_n$ + lower length terms\bigg)}\\
&&+\int\text{terms of length $\leq n-1$}.
\end{eqnarray*}
The left hand side and the first integral on the right are both closed, so that the second integral on the right also has to be closed, hence in $(L_{n-1})_\CC$, and by the induction hypothesis in $W_{n-1}(L_{n-1})_\CC\subset W_n(L_n)_\CC$. The desired conclusion follows.\\
\end{proof}

\subsection{} \label{parq} In this paragraph we review some facts from group theory and then apply them to our setting. Let $\Gamma$ be a finitely generated group, $\mathbb{K}\in\{\ZZ,\QQ,\CC\}$, and $I$ be the augmentation ideal in $\mathbb{K}[\Gamma]$. Let $\Gamma^\text{ab}:=\frac{\Gamma}{[\Gamma,\Gamma]}$. It is well-known that
\begin{equation}\label{ImodI^2}
\frac{I}{I^2}\rightarrow \Gamma^\text{ab}\otimes \mathbb{K}\hspace{.5in}[\gamma-1]\mapsto [\gamma]
\end{equation}   
is an isomorphism. For $n>1$ however, the quotients $\frac{I^n}{I^{n+1}}$ become increasingly more complicated in general. (See Stallings \cite{Stallings}.) On the other hand, if $\Gamma$ is free, these quotients are easy to describe: One has an isomorphism
\begin{equation}\label{quotientsofI}
\frac{I^n}{I^{n+1}}\rightarrow \left(\frac{I}{I^2}\right)^{\otimes n}
\end{equation}
given by
\[
[(\gamma_1-1)\ldots(\gamma_n-1)]\mapsto [\gamma_1-1]\otimes\ldots\otimes[\gamma_n-1].
\]
Let $\Gamma$ be free. Then $\frac{I}{I^2}$, and hence $\frac{I^n}{I^{n+1}}$ for every $n$, is a free $\mathbb{K}$-module. (Of course, this is only interesting when $\mathbb{K}=\ZZ$.) One has for each $n$ an obvious exact sequence (of $\mathbb{K}$-modules)
\begin{equation*}
0\longrightarrow\frac{I^n}{I^{n+1}}\longrightarrow\frac{\mathbb{K}[\Gamma]}{I^{n+1}}\longrightarrow\frac{\mathbb{K}[\Gamma]}{I^n}\longrightarrow 0.
\end{equation*}
We see by induction that each $\displaystyle{\frac{\mathbb{K}[\Gamma]}{I^n}}$ is free, and hence dualizing the previous sequence we get exact
\[
0\longrightarrow\left(\frac{\mathbb{K}[\Gamma]}{I^n}\right)^\vee\longrightarrow\left(\frac{\mathbb{K}[\Gamma]}{I^{n+1}}\right)^\vee\longrightarrow\left(\frac{I^n}{I^{n+1}}\right)^\vee\longrightarrow 0.
\]
Via
\[
\left(\frac{I^n}{I^{n+1}}\right)^\vee\stackrel{\eqref{quotientsofI}}{\simeq}\left(\left(\frac{I}{I^2}\right)^{\otimes n}\right)^\vee\stackrel{\eqref{ImodI^2}}{\simeq}\left(\left(\Gamma^\text{ab}\otimes \mathbb{K}\right)^{\otimes n}\right)^\vee,
\]  
we get a short exact sequence
\begin{equation}\label{q'_k}
0\longrightarrow\left(\frac{\mathbb{K}[\Gamma]}{I^n}\right)^\vee\longrightarrow\left(\frac{\mathbb{K}[\Gamma]}{I^{n+1}}\right)^\vee\stackrel{q_\mathbb{K}}{\longrightarrow} \left((\Gamma^\text{ab}\otimes \mathbb{K})^{\otimes n}\right)^\vee\longrightarrow 0.
\end{equation}
Unwinding definitions, it is easy to see that $q_\mathbb{K}$ sends $\displaystyle{f\in\left(\frac{\mathbb{K}[\Gamma]}{I^{n+1}}\right)^\vee}$ to the map 
\[
[\gamma_1]\otimes\ldots\otimes[\gamma_n]~\mapsto~ f\left([(\gamma_1-1)\ldots(\gamma_n-1)]\right).
\]
It is clear that \eqref{q'_k} is compatible with extending $\mathbb{K}$.\\

We apply this to the group $G=\pi_1(U,e)$. In view of the definition of $(L_n)_\mathbb{K}$, the isomorphism $G^\text{ab}\otimes \mathbb{K}\simeq H_1(U,\mathbb{K})$ given by $[\gamma]\mapsto[\gamma]$, and 
\[
\left(H_1(U,\mathbb{K})^{\otimes n}\right)^\vee\cong \left(H_1(U,\mathbb{K})^\vee\right)^{\otimes n}\cong\left(H^1(U)_\mathbb{K}\right)^{\otimes n},
\]
the sequence \eqref{q'_k} reads
\begin{equation}\label{seq1}
0\longrightarrow (L_{n-1})_\mathbb{K}\stackrel{\text{inclusion}}{\longrightarrow}(L_n)_\mathbb{K}\stackrel{q_\mathbb{K}}{\longrightarrow} \left(H^1(U)_\mathbb{K}\right)^{\otimes n} \longrightarrow 0.
\end{equation}
Compatibility with extending $\mathbb{K}$ implies the maps in this sequence when $\mathbb{K}=\CC$ are defined over $\ZZ$ (i.e. take integral lattices to integral lattices, and hence rationals to rationals), and the sequence when $\mathbb{K}=\ZZ$ (resp. $\mathbb{K}=\QQ$) is the restriction of the sequence for $\mathbb{K}=\CC$ to integral (resp. rational) lattices. In particular, these restrictions are exact.\\  

The inclusion $U\subset X$ gives an isomorphism $H^1(X)\rightarrow H^1(U)$. We will always identify the two Hodge structures via this map, and from now on simply write $H^1$ for $H^1(U)=H^1(X)$. Unwinding definitions, in view of  
\begin{equation}\label{eqnilit}
\int\limits_{(\gamma_1-1)\ldots(\gamma_n-1)} \omega_1\ldots\omega_n+\text{lower length terms}=\int\limits_{\gamma_1}\omega_1~\ldots~\int\limits_{\gamma_n}\omega_n,
\end{equation}
we see that the map $q_\CC$ sends
\begin{equation}\label{eqq}
\int\omega_1\ldots\omega_n+\text{lower length terms}~~\mapsto [\omega_1]\otimes\ldots\otimes[\omega_n],
\end{equation}
where the integral on the left is closed, each $\omega_i$ is a closed smooth 1-form on $U$, and $[\omega_i]$ denotes the cohomology class of $\omega_i$. Note that \eqref{eqnilit} is a consequence of \eqref{itintprop}.\\

It is clear from the description of the weight filtration on $L_n$ given in Proposition \ref{weightdescription} that the map $q_\CC$ is compatible with the weight filtrations. We shall shortly see that it is also compatible with the Hodge filtrations, so that it gives an isomorphism of mixed Hodge structures
\[\frac{L_n}{L_{n-1}}\rightarrow (H^1)^{\otimes n}.\]
We will not try to take the fastest route to this end. Rather, we will conclude this as a consequence of existence of a section of $q_\CC$ respecting the Hodge filtrations. Over the next three paragraphs, we will construct a particular section $s_F$ of $q_\CC$. This map enjoys some nice properties and will play an important role later on.

\subsection{} \label{greenfacts} In this paragraph, we review some basic facts about Green functions. For the proofs and further details, see \cite{lang}.\\

Let $\varphi$ be a real non-exact smooth form of type (1,1) on $X$, $D$ be a nonzero divisor on $X$, and $\text{supp}(D)$ be the support of $D$. Then $\varphi$ is exact on $X-\text{supp}(D)$. Indeed, one can prove that there is a unique (smooth) function $g_{_{D,\varphi}}: X-\text{supp}(D)\rightarrow \RR$, called the Green function for $\varphi$ relative to $D$, satisfying the following properties: 
\begin{itemize}
\item[(1)] If $D$ is represented by a meromorphic function $f$ on an open set (in analytic topology) $V$ of $X$, then the function $V-\text{supp}(D)\rightarrow\RR$ defined by\footnote[2]{The appearance of the extra factor $\int\limits_X\varphi$ compared to Lang comes from the fact that $\varphi$ is not normalized here.} 
\[P\mapsto g_{_{D,\varphi}}(P)+(\int\limits_X\varphi)~\log|f(P)|^2\]
extends smoothly to $V$.
\item[(2)] $dd^cg_{_{D,\varphi}}=(\deg D)\varphi$ on $X-\text{supp}(D)$, where $d^c=\frac{1}{4\pi i}(\partial-\bar{\partial})$ with the $\partial$, $\bar{\partial}$ the usual operators.
\item[(3)] $\displaystyle{\int\limits_X g_{_{D,\varphi}}\varphi=0}$.      
\end{itemize}

\noindent One can show that a function satisfying (1) and (2) is unique up to a constant. Condition (3) is included to guarantee uniqueness. Conditions (1) and (2) are the important ones for us.  Take $D=\infty$. It follows from (1) that locally near the point $\infty$, with a chart taken such that $\infty$ corresponds to $z=0$, the function $g_{\infty,\varphi}$ looks like
\[
-(\int\limits_X\varphi)~\log z\bar{z}+\text{a smooth function}.
\]  
It follows that $\partial g_{\infty,\varphi}$ near $\infty$ (again with $z=0$ corresponding to the point $\infty$) is of the form
\[
-(\int\limits_X\varphi)~\frac{dz}{z}+\text{a smooth 1-form},
\]
so that $\partial g_{\infty,\varphi}$ is in $E^1(X\log\infty)$. By condition (2), $d(\frac{1}{2\pi i}\partial g_{\infty,\varphi})=\varphi$ on $U$. To sum up, given a a non-exact real two-form $\varphi$ on $X$, we have a specific 1-form $\frac{1}{2\pi i}\partial g_{\infty,\varphi}$ of type (1,0) in $E^1(X\log\infty)$ with residue $\displaystyle{-\frac{1}{2\pi i}\int\limits_X\varphi}$ at $\infty$ whose $d$ is $\varphi$ on $U$.

\subsection{}  Throughout this paragraph, $\mathbb{K}=\RR$ or $\CC$. Let $\mathcal{H}_\mathbb{K}^1(X)$ be the space of $\mathbb{K}$-valued harmonic 1-forms on $X$. One has a commutative diagram
\[
\begin{array}{ccc}
\mathcal{H}^1_\RR(X)&\cong&H_\RR^1\\
\cap&&\cap\\
\mathcal{H}^1_\CC(X)&\cong&H_\CC^1.
\end{array}
\]

Via the horizontal isomorphisms we get a pure real Hodge structure $\mathcal{H}^1(X)$ of weight one with $\mathbb{K}$-vector space $\mathcal{H}_\mathbb{K}^1(X)$. The subspace $F^1\mathcal{H}^1_\CC(X)$ is the space of holomorphic 1-forms on $X$. Let $\wedge:\mathcal{H}^1_\mathbb{K}\otimes\mathcal{H}^1_\mathbb{K}\rightarrow E_\mathbb{K}^2(X)$ be the ``wedge product" map, i.e. given by $\wedge(\omega_1\otimes\omega_2)=\omega_1\wedge\omega_2$. The following lemma combines some ideas of Pulte \cite{Pulte} and Darmon, Rotger and Sols \cite{DRS}. 

\begin{lemma}\label{lemlift1}
There is a $\CC$-linear map 
\[
\nu: \mathcal{H}_\CC^1(X)\otimes\mathcal{H}_\CC^1(X)\rightarrow E^1(X\log\infty)
\] 
such that 
\begin{itemize}
\item[(i)] for each $w\in\mathcal{H}_\CC^1(X)\otimes\mathcal{H}_\CC^1(X)$, $d(\nu(w))=-\wedge(w)$ on $U$,
\item[(ii)] $\nu$ respects the Hodge filtration $F$,
\item[(iii)] for each $w\in\mathcal{H}_\RR^1(X)\otimes\mathcal{H}_\RR^1(X)$, there is a smooth real 1-form $\nu_\RR=\nu_\RR(w)$ on $U$ such that $\nu(w)-\nu_\RR$ is exact on $U$, 
\item[(iv)] for every $w\in\mathcal{H}_\CC^1(X)\otimes\mathcal{H}_\CC^1(X)$, the residue of $\nu(w)$ at $\infty$ is $\displaystyle{\frac{1}{2\pi i}\int\limits_X\wedge(w)}$.
\end{itemize}	
\end{lemma}

\begin{proof}
The cup product $H^1\otimes H^1\stackrel{\cupprod}{\rightarrow}H^2(X)$ is a morphism of Hodge structures. Let $K$ be its kernel. Ignoring the rational structures, we can think of $K$ as a sub-Hodge structure of the real Hodge structure $H^1\otimes H^1$. Let $\mathcal{K}$ be its copy in $\mathcal{H}^1(X)\otimes\mathcal{H}^1(X)$. Thus $\mathcal{K}_\mathbb{K}$ consists of those $w\in\mathcal{H}_\mathbb{K}^1(X)\otimes\mathcal{H}_\mathbb{K}^1(X)$ for which $\wedge(w)\in E_\mathbb{K}^2(X)$ is exact. One has a short exact sequence of real Hodge structures
\[
0\longrightarrow\mathcal{K}\stackrel{\text{inclusion}}{\longrightarrow}\mathcal{H}^1(X)\otimes\mathcal{H}^1(X)\cong H^1\otimes H^1\stackrel{\cupprod}{\longrightarrow}H^2(X)\longrightarrow 0.
\]    
The category of pure real Hodge structures is semi-simple, so that there is 
\[\phi\in \mathcal{H}_\RR^1(X)\otimes\mathcal{H}_\RR^1(X)~\cap~F^1(\mathcal{H}_\CC^1(X)\otimes\mathcal{H}_\CC^1(X))\] 
giving rise to a decomposition of $\mathcal{H}^1(X)\otimes\mathcal{H}^1(X)$ as an internal direct sum
\[
\mathcal{H}^1(X)\otimes\mathcal{H}^1(X)=\mathcal{K}~\oplus~\mathcal{L},
\]
where $\mathcal{L}$ is the one dimensional sub-object of $\mathcal{H}^1(X)\otimes\mathcal{H}^1(X)$ generated by $\phi$\footnote[2]{We could have instead worked over $\QQ$ here, as the Mumford-Tate group of $X$ is reductive. But this would not result in any major simplification.}. Because of the linear nature of the requirements, it suffices to define $\nu$ on $\mathcal{K}_\CC$ and $\mathcal{L}_\CC$ satisfying (i)-(iv).\\

\noindent \underline{Definition of $\nu$ on $\mathcal{K}_\CC$}: This part is due to Pulte \cite{Pulte}. The operator $d$ on $X$ is strict with respect to the Hodge filtration, so that one can choose 
\[
\nu':\mathcal{K}_\CC\rightarrow E_\CC^1(X)
\]   
respecting the Hodge filtration such that $d\nu'(w)=-\wedge(w)$ on $X$. Now recall that one has a decomposition $E^1_\mathbb{K}(X)=\mathcal{H}_\mathbb{K}^1(X)\oplus\mathcal{H}_\mathbb{K}^1(X)^\perp$, where $\mathcal{H}_\mathbb{K}^1(X)^\perp$ is the space of $\mathbb{K}$-valued 1-forms orthogonal to $\mathcal{H}_\mathbb{K}^1(X)$ with respect to the inner product defined using the Hodge $\ast$ operator. Recall also that the projections $E^1_\CC(X)\rightarrow\mathcal{H}^1_\CC(X)$ and $E^1_\CC(X)\rightarrow\mathcal{H}^1_\CC(X)^\perp$ preserve type. Define $\nu$ to be the composition of $\nu'$ and the latter projection. Since harmonic forms are closed, we have $d\nu(w)=d\nu'(w)=-\wedge(w)$. Note that condition (iv) holds trivially. We claim that $\nu$ satisfies property (iii) as well. Let $w\in\mathcal{K}_\RR$. Then $\wedge(w)$ is exact and real, so that there is $\nu'_\RR\in E^1_\RR(X)$ such that $d\nu'_\RR=-\wedge(w)$. Let $\nu_\RR$ be the component of $\nu'_\RR$ in ${\mathcal{H}_\RR^1(X)}^\perp$. Then $d\nu_\RR=d\nu'_\RR=-\wedge(w)$, so that $\nu(w)-\nu_\RR\in{\mathcal{H}_\CC^1(X)}^\perp$ is closed. The desired conclusion follows from the general fact that a closed element of $\mathcal{H}_\mathbb{K}^1(X)^\perp$ is necessarily exact. Note that on the subspace $\mathcal{K}_\CC$ the requirements of the lemma hold on all of $X$, not just $U$.\\   

\noindent \underline{Definition of $\nu$ on $\mathcal{L}_\CC$}: Define $\nu$ on the subspace $\mathcal{L}_\CC=\CC\phi$ by $\nu(\phi)=-\frac{1}{2\pi i}\partial g_{\infty,\wedge(\phi)}$. Conditions (i), (ii) and (iv) hold by Paragraph \ref{greenfacts}. As for condition (iii), note that $-d^cg_{\infty,\wedge(\phi)}$ is real, and 
\[
-\frac{1}{2\pi i}\partial g_{\infty,\wedge(\phi)}+d_cg_{\infty,\wedge(\phi)}=-\frac{1}{4\pi i}dg_{\infty,\wedge(\phi)}.
\]
\end{proof}

If the point $\infty$ is not clear from the context, we will write $\nu_\infty$ instead of $\nu$. Note that the map $\nu$ is not natural; it depends on the choices of $\phi$ and $\nu'$.

\subsection{} In this paragraph, we use Lemma \ref{lemlift1} to construct a section $s_{F}$ of $q_\CC:(L_n)_\CC\rightarrow (H^1_\CC)^{\otimes n}$ that is compatible with the Hodge filtrations, and also such that its composition with $(L_n)_\CC\rightarrow\left(\frac{L_n}{L_{n-2}}\right)_\CC$ is defined over $\RR$. This map is of crucial importance in the later parts of the paper.\\ 
%This will give us a real Hodge section of $\mathfrak{q}$.\\ 

By exactness of $F^1E^\cdot(X\log\infty)$ in degree 2, one can (non-uniquely) extend the map $\nu$ of the previous paragraph to a map 
\[\tilde{\nu}: E^1(X\log\infty)\otimes E^1(X\log\infty)\rightarrow E^1(X\log\infty)\]
respecting the Hodge filtrations and satisfying $d(\tilde{\nu}(w))=-\wedge(w)$ for every $w\in E^1(X\log\infty)\otimes E^1(X\log\infty)$. The differential graded algebra $E^\cdot(X\log\infty)$ with the data of $s(a,a')=\tilde{\nu}(a\otimes a')$ for each $a,a'\in E^1(X\log\infty)$ satisfies the conditions of Section \ref{bar2}, and hence in particular for $\omega_1,\ldots,\omega_n\in \mathcal{H}_\CC^1(X)$, we have a closed iterated integral on $U$ of the form 
\begin{equation}\label{eq1ch3revision}
\int~\omega_1\ldots\omega_n+\sum\limits_{i=1}^{n-1}\omega_1\ldots\nu(\omega_i\otimes\omega_{i+1})\ldots\omega_n+\text{terms of length at most $n-2$.}
\end{equation}
(See the construction of Section \ref{bar2}.) In view of $(H^1_\CC)^{\otimes n}\cong(\mathcal{H}^1_\CC)^{\otimes n}$, we define the map $s_F:(H^1_\CC)^{\otimes n}\rightarrow (L_n)_\CC$ by
\[
[\omega_1]\otimes\ldots\otimes[\omega_n]\mapsto \text{the iterated integral described above,}
\]
where $\omega_i\in \mathcal{H}_\CC^1(X)$ and $[\omega_i]$ denotes the cohomology class of $\omega_i$. This is well-defined and linear (see the final remark of Section \ref{bar2}), and in view of \eqref{eqq} it is a section of $q_\CC$ (of Paragraph \ref{parq}). Also, it is apparent from the construction of Section \ref{bar2} that since $\tilde{\nu}$ preserves the Hodge filtration $F$, so does $s_F$. That $s_F$ respects the weight filtration (over $\CC$) is obvious from $W_n(L_n)_\CC=(L_n)_\CC$. We have proved parts (i)-(iii) of the following lemma.

\begin{lemma}\label{s_F}
There is a $\CC$-linear map $s_F:(H^1_\CC)^{\otimes n}\rightarrow (L_n)_\CC$ that satisfies the following properties:
\begin{itemize}
\item[(i)] Given $\omega_1,\ldots,\omega_n\in \mathcal{H}_\CC^1(X)$, $s_F([\omega_1]\otimes\ldots\otimes[\omega_n])$ is of the form \eqref{eq1ch3revision}.
\item[(ii)] $s_F$ is a section of $q_\CC:(L_n)_\CC\rightarrow (H^1_\CC)^{\otimes n}$.
\item[(iii)] $s_F$ respects the Hodge and weight filtrations.
\item[(iv)] The composition
\[
\mathfrak{s}_F:~(H_\CC^1)^{\otimes n}\stackrel{s_F}{\longrightarrow}(L_n)_\CC\stackrel{\text{quotient}}{\longrightarrow}(\frac{L_n}{L_{n-2}})_\CC
\] 
is defined over $\RR$.
\end{itemize}
\end{lemma}

\begin{proof}
(of (iv)) We must show that if $w\in (H^1_\RR)^{\otimes n}$, then 
\[\mathfrak{s}_F(w)\in(\frac{L_n}{L_{n-2}})_\RR\subset (\frac{L_n}{L_{n-2}})_\CC,\]
or equivalently, $s_F(w)\in (L_n)_\RR+(L_{n-2})_\CC$. It suffices to consider $w=[\omega_1]\otimes\ldots\otimes[\omega_n]$, where the $\omega_i\in\mathcal{H}^1_\RR(X)$. In view of Lemma \ref{lemlift1}(iii) and the relations \eqref{itrel} satisfied by iterated integrals, we have
\[
s_F(w)=\int~\omega_1\ldots\omega_n+\sum\limits_{i=1}^{n-1}\omega_1\ldots\nu_\RR(\omega_i\otimes\omega_{i+1})\ldots\omega_n+\text{terms of length $\leq n-2$.}
\]
Applying the construction of Section \ref{bar2} to the differential graded algebra $E^\cdot_\RR(U)$ with $s(-,-)$ chosen such that $s(\omega_i,\omega_{i+1})=\nu_\RR(\omega_i\otimes\omega_{i+1})$, we get a closed element of $\overline{B}^0(E^\cdot_\RR(U))$ of the form
\[
(\omega_1|\ldots|\omega_n)+\sum\limits_{i=1}^{n-1}(\omega_1|\ldots|\nu_\RR(\omega_i\otimes\omega_{i+1})|\ldots|\omega_n)+\text{terms of length $\leq n-2$,}
\]
and hence an element of $(L_n)_\RR$ of the form
\[
\int~\omega_1\ldots\omega_n+\sum\limits_{i=1}^{n-1}\omega_1\ldots\nu_\RR(\omega_i\otimes\omega_{i+1})\ldots\omega_n+\text{terms of length $\leq n-2$.}
\]
This differs from $s_F(w)$ by an element of $(L_{n-2})_\CC$, giving the desired conclusion.
\end{proof}

\subsection{}\label{qbarparnew} 
Let $\overline{q}_\CC$ be the isomorphism of vector spaces
\[
\left(\frac{L_n}{L_{n-1}}\right)_\CC\rightarrow (H^1_\CC)^{\otimes n}
\]
induced by $q_\CC$. Let $\overline{s}_F$ be the composition
\[
(H^1_\CC)^{\otimes n}\stackrel{s_F}{\rightarrow} (L_n)_\CC\stackrel{\text{quotient}}{\rightarrow}\left(\frac{L_n}{L_{n-1}}\right)_\CC.
\]
Then $\overline{s}_F$ is the inverse of $\overline{q}_\CC$. By the discussion of Paragraph \ref{parq}, $\overline{q}_\CC$ restricts to an isomorphism of the integral lattices. It follows that the same is true for $\overline{s}_F$. Moreover, $\overline{s}_F$ is compatible with the Hodge and weight filtrations (because so is $s_F$), and hence is a morphism of mixed Hodge structures. In view of strictness of morphisms in $\mathbf{MHS}$ with respect to the Hodge filtration, $\overline{q}_\CC$ is also compatible with the Hodge filtration. The following statement follows. (Compatibility of $\overline{q}_\CC$ with the weight filtration is obvious.)  

\begin{prop}\label{qbar}
The map $q_\CC$ induces an isomorphism of mixed Hodge structures 
\[\overline{q}: \frac{L_n}{L_{n-1}}\rightarrow (H^1)^{\otimes n}.\]
\end{prop}

In the interest of keeping the notation simple, here we did not incorporate $n$ in the notation for $\overline{q}$. When there is a possibility of confusion, we will instead use the decorated notation $\overline{q}_n$ for the isomorphism given in Proposition \ref{qbar}. 

\begin{rem} (1) Note that in particular this says even though the mixed Hodge structure $L_m$ may depend on the base point $e$, the quotient $\text{Gr}^W_nL_m=\frac{L_n}{L_{n-1}}$ does not. In fact, it does not even depend on the point $\infty$ we removed from $X$. It is true in general that for any smooth connected complex variety the quotients $\frac{L_n}{L_{n-1}}$ are independent of the base point. See (3.22) Remark (iii) of \cite{HZ}.\\
(2) It follows from the above that the map $q_\CC$ is also compatible with the Hodge filtration, and that \eqref{seq1} is a short exact sequence of mixed Hodge structures.\\
(3) We should clarify that Proposition \ref{qbar} is not a new result. For instance, it can be deduced from the ideas behind Remark (iii) of Paragraph (3.22) of \cite{HZ}. Here we included a proof as it was easy to do so with the section $s_F$ at hand, and in the interest of making the paper more self-contained. 
\end{rem}

\section{The extension $\mathbb{E}^\infty_{n,p}$}\label{section_extension_E}
\subsection{}\label{RmodZ} Let $A$ be a mixed Hodge structure with torsion-free $A_\ZZ$. The kernel of the surjective map
\[
\Hom_{\ZZ}(A_{\ZZ},\RR)\rightarrow \Hom_{\ZZ}(A_{\ZZ},\RR/\ZZ)
\]
induced by the natural quotient map $\RR\rightarrow\RR/\ZZ$ is $\Hom_{\ZZ}(A_{\ZZ},\ZZ)$. Putting this together with
\[\Hom_{\RR}(A_{\RR},\RR)\cong\Hom_{\ZZ}(A_{\ZZ},\RR),\]
we obtain
\[
\frac{\Hom_{\RR}(A_{\RR},\RR)}{\Hom_{\ZZ}(A_{\ZZ},\ZZ)}\cong \Hom_{\ZZ}(A_{\ZZ},\RR/\ZZ).
\]

\noindent Now suppose $A$ is pure of odd weight. Then so is $A^\vee$, and 
\[
JA^\vee\stackrel{\eqref{eqHSbasics2}}{\cong} \frac{\Hom_{\RR}(A_\RR,\RR)}{\Hom_{\ZZ}(A_\ZZ,\ZZ)}\cong \Hom_{\ZZ}(A_{\ZZ},\RR/\ZZ).
\]
Unwinding definitions, we see that given $f:A_{\CC}\rightarrow\CC$ defined over $\RR$, the class of $f$ in $JA^\vee$ corresponds under the identification to the composition 
\begin{equation}\label{eqR/Zcorel}
A_{\ZZ}\stackrel{\text{inclusion}}{\rightarrow}A_{\RR}\stackrel{f}{\rightarrow}\RR\rightarrow\RR/\ZZ
\end{equation}
in $\Hom_{\ZZ}(A_{\ZZ},\RR/\ZZ)$.

\begin{rem} Here we make an observation that will be useful later on. Let $A$ be of weight $2n-1$, and $f:A_\CC\rightarrow\CC$ be defined over $\RR$. It follows from the above that $f(A_\ZZ)\subset\ZZ$ if and only if the restriction of $f$ to $F^{n}A_\CC$ is equal to that of an element of $\Hom_\ZZ(A_\ZZ,\ZZ)\subset\Hom(A_\CC,\CC)$. Indeed, the first statement is equivalent to that the composition \eqref{eqR/Zcorel} is trivial, which is equivalent to that the class of $f$ is trivial in $JA^\vee$, i.e. $f\in F^{1-n}(A^\vee)_\CC+\Hom_\ZZ(A_\ZZ,\ZZ)$, which, in view of
\[
F^{1-n}(A^\vee)_\CC=\{g:A_\CC\rightarrow\CC: g(F^{n}A_\CC)=0\},
\] 
is equivalent to the second statement. Note that the ``only if" part of the statement is trivial.
\end{rem} 

\subsection{}\label{PhiandPsi} Let $H_1:=(H^1)^\vee$. We identify $(H_1)_\ZZ$ with $H_1(X,\ZZ)$ (the singular homology). One has an isomorphism of Hodge structures $H^1(1)\cong H_1$ given by Poincare duality
\[
PD: H^1(1)\stackrel{\simeq}{\longrightarrow}H_1,\hspace{.5in} [\omega]~\mapsto~\int\limits_X[\omega]\wedge -, 
\]  
where $\omega$ is a smooth closed 1-form on $X$. This gives for each positive $n$ an isomorphism 
\[
PD^{\otimes n}:(H^1)^{\otimes n}(n)\longrightarrow H_1^{\otimes n}\cong (H^1)^{\otimes -n}, 
\]
given by 
\[
[\omega_1]\otimes\ldots\otimes[\omega_n]~\mapsto~PD([\omega_1])\otimes\ldots\otimes PD([\omega_n])=\left([\omega'_1]\otimes\ldots\otimes[\omega'_n]\mapsto \prod\limits_i\int\limits_X[\omega_i]\wedge[\omega'_i]\right).
\]
We have
\begin{eqnarray}
\Ext\left((H^1)^{\otimes n},(H^1)^{\otimes n-1}\right)&\stackrel{\text{Carlson (Par. \ref{Carlsoniso})}}{\cong}&J\inhom\left((H^1)^{\otimes n},(H^1)^{\otimes n-1}\right)\notag\\
&\stackrel{\text{Lemma \ref{basichodge}(a)}}{\cong}&J\inhom\left((H^1)^{\otimes n}\otimes (H^1)^{\otimes 1-n},\ZZ(0)\right)\notag\\
&\stackrel{PD^{\otimes n-1}}{\cong}&J\inhom\left((H^1)^{\otimes n}\otimes {(H^1)}^{\otimes n-1}(n-1),\ZZ(0)\right)\notag\\
%&\stackrel{\text{Lemma \ref{basichodge}(f)}}{\cong}&J\inhom\left((H^1)^{\otimes n}\otimes {H^1}^{\otimes n-1},\ZZ(0)\right)\\
&\stackrel{\text{Lemma \ref{basichodge}(f)}}{\cong}&J((H^1)^{\otimes 2n-1})^\vee\label{eqPsisteps}.
%&\stackrel{\text{Paragraph \ref{RmodZ}}}{\cong}&\Hom_{\ZZ}\left((H_\ZZ^1)^{\otimes 2n-1}, \RR/\ZZ\right).
\end{eqnarray}
Let $\Psi$ be the composition isomorphism
\[
\Ext\left((H^1)^{\otimes n},(H^1)^{\otimes n-1}\right)\longrightarrow J((H^1)^{\otimes 2n-1})^\vee.
\] 
We denote by $\Phi$ the isomorphism
\[
J((H^1)^{\otimes 2n-1})^\vee\longrightarrow \Hom_{\ZZ}\left((H_\ZZ^1)^{\otimes 2n-1}, \RR/\ZZ\right)
\]
given by Paragraph \ref{RmodZ}. (To make the notation slightly simpler we did not include $n$ as a part of the symbol for the maps $\Phi$ and $\Psi$. This should not cause any confusion as $n$ will be clear from the context.)\\

\subsection{Definition of $\mathbb{E}^\infty_{n,e}$}\label{defe} Let $n\geq 2$. In this paragraph, we use $\displaystyle{\frac{L_n}{L_{n-2}}}$ to define an element 
\[
\mathbb{E}_{n,e}^{\infty}\in \Ext((H^1)^{\otimes n},(H^1)^{\otimes n-1}). 
\]
It follows from Proposition \ref{weightdescription} that the weight filtration on $\displaystyle{\frac{L_n}{L_{n-2}}}$ is given by
\[
W_{n-2}=0,\hspace{.5in} W_{n-1}=\frac{L_{n-1}}{L_{n-2}}~,\hspace{.5in}\text{and}\hspace{.2in} W_n=\frac{L_n}{L_{n-2}}. 
\]  
The filtration gives rise to the exact sequence
\[
0\longrightarrow \frac{L_{n-1}}{L_{n-2}}\stackrel{\iota}{\longrightarrow} \frac{L_n}{L_{n-2}}\stackrel{\text{quotient}}{\longrightarrow} \frac{L_n}{L_{n-1}}\longrightarrow 0,
\]
where $\iota$ is the inclusion map. Using the isomorphism of Proposition \ref{qbar} to replace $\frac{L_{n-1}}{L_{n-2}}$ (resp. $\frac{L_n}{L_{n-1}}$) by $(H^1)^{\otimes n-1}$ (resp. $(H^1)^{\otimes n}$), we get the exact sequence  
\begin{equation}\label{seq2}
0\longrightarrow (H^1)^{\otimes n-1}\stackrel{\mathfrak{i}}{\longrightarrow} \frac{L_n}{L_{n-2}}\stackrel{\mathfrak{q}}{\longrightarrow} (H^1)^{\otimes n}\longrightarrow 0.
\end{equation}
Here $\mathfrak{i}=\iota{\overline{q}}~^{-1}$, and $\mathfrak{q}$ is the composition 
\[
\frac{L_n}{L_{n-2}}\stackrel{\text{quotient}}{\longrightarrow} \frac{L_n}{L_{n-1}}\stackrel{\overline{q}}{\longrightarrow}(H^1)^{\otimes n}.
\]
Let $\mathbb{E}_{n,e}^{\infty}\in\Ext((H^1)^{\otimes n},(H^1)^{\otimes n-1})$ be the extension defined by the sequence \eqref{seq2}.\\ 

\begin{rem} One can deduce from a theorem of Pulte \cite{Pulte} that the map
\[
X(\CC)-\{\infty\}\rightarrow \Ext((H^1)^{\otimes 2},H^1)
\] 
defined by $e\mapsto \mathbb{E}_{2,e}^{\infty}$ is injective.
\end{rem}

Our goal in the remainder of this section is to describe the images of $\mathbb{E}^\infty_{n,e}$ under $\Psi$ and $\Phi\circ\Psi$. To this end, in view of Paragraph \ref{RmodZ} and Paragraph \ref{Carlsoniso}, we will define an integral retraction of $\mathfrak{i}$ and a Hodge section of $\mathfrak{q}$ defined over $\RR$. (See \eqref{seq2}.)

\subsection{}\label{intretpar} 
\underline{An integral retraction of $\mathfrak{i}$}: In this paragraph, we define an integral retraction $r_\ZZ$ of $\mathfrak{i}$, i.e. a linear map
\[
r_\ZZ: (\frac{L_n}{L_{n-2}})_\CC\longrightarrow (H_\CC^1)^{\otimes n-1}
\]  
defined over $\ZZ$, that is left inverse to $\mathfrak{i}$.\\

Choose $\beta_1,\ldots,\beta_{2g}\in\pi_1(U,e)$ such that the $[\beta_j]\in H_1(X,\ZZ)$ form a basis. To define an element of $(H_\CC^1)^{\otimes n-1}$, it suffices to specify how it pairs with the elements $[\beta_{j_1}]\otimes\ldots\otimes [\beta_{j_{n-1}}]$ of $H_1(X,\ZZ)^{\otimes n-1}$. Moreover, an element of $(H_\CC^1)^{\otimes n-1}$ is in $(H_\ZZ^1)^{\otimes n-1}$ if and only if it produces integer values when pairing with the $[\beta_{j_1}]\otimes\ldots\otimes [\beta_{j_{n-1}}]$. Given an element
\[
f=\int \sum\limits_{i\leq n} w_i~+~(L_{n-2})_\CC~\in~(\frac{L_n}{L_{n-2}})_\CC,     
\]
where $w_i$ is a sum of terms of length $i$ and the iterated integral is closed, set $r_\ZZ(f)$ to be the element of $(H_\CC^1)^{\otimes n-1}$ satisfying
\begin{equation}\label{eqr_Z}
[\beta_{j_1}]\otimes\ldots\otimes [\beta_{j_{n-1}}](r_\ZZ(f))=\int\limits_{(\beta_{j_1}-1)\ldots(\beta_{j_{n-1}}-1)} \sum\limits_{i\leq n} w_i.
\end{equation}
%\begin{equation}\label{eqr_Z}
%r_\ZZ(f)=~\left([\beta_{j_1}]\otimes\ldots\otimes [\beta_{j_{n-1}}]~\mapsto~\int\limits_{(\beta_{j_1}-1)\ldots(\beta_{j_{n-1}}-1)} \sum\limits_{i\leq n} w_i~\right).
%\end{equation}
Note that
\[
\int\limits_{(\beta_{j_1}-1)\ldots(\beta_{j_{n-1}}-1)} \sum\limits_{i\leq n} w_i=\int\limits_{(\beta_{j_1}-1)\ldots(\beta_{j_{n-1}}-1)} w_n+w_{n-1}.
\]
Since $(L_{n-2})_\CC$ vanishes on $I^{n-1}$, $r_\ZZ$ is well-defined. Moreover, $r_\ZZ$ is defined over $\ZZ$, for if $\displaystyle{f\in(\frac{L_n}{L_{n-2}})_\ZZ}$, the iterated integral $\int\sum w_i$ can be chosen to be integer-valued on $\pi_1(U,e)$, and hence \eqref{eqr_Z} is an integer. Finally, we check that $r_\ZZ$ is a retraction of $\mathfrak{i}$. In view of Lemma \ref{lem3.1} and the formula \eqref{eqq} for $q_\CC$, if $\omega_1,\ldots,\omega_{n-1}$ are smooth closed 1-forms on $X$, $\mathfrak{i}([\omega_1]\otimes\ldots\otimes[\omega_{n-1}])$ is of the form
\[
\int \omega_1\ldots\omega_{n-1}+\text{lower length terms} \mod (L_{n-2})_\CC,
\]
where the iterated integral is closed. We have
\begin{eqnarray*}
[\beta_{j_1}]\otimes\ldots\otimes [\beta_{j_{n-1}}]~\bigm(r_\ZZ\circ\mathfrak{i}([\omega_1]\otimes\ldots\otimes[\omega_{n-1}])\bigm)&=&\int\limits_{(\beta_{j_1}-1)\ldots(\beta_{j_{n-1}}-1)}\omega_1\ldots\omega_{n-1}\\
&=&\int\limits_{\beta_{j_1}}\omega_1\ldots\int\limits_{\beta_{j_{n-1}}}\omega_{n-1},
\end{eqnarray*}
which is the same as
\[[\beta_{j_1}]\otimes\ldots\otimes [\beta_{j_{n-1}}]\left([\omega_1]\otimes\ldots\otimes[\omega_{n-1}]\right),\]
as desired.

\begin{rem} 
The retraction $r_\ZZ$ is by no means natural, as it depends on the choice of the $\beta_j$.
\end{rem}

\subsection{} \underline{A real Hodge section of $\mathfrak{q}$}: The first assertion of the following lemma is immediate from Lemma \ref{s_F} (ii), (iii) and (iv). In view of the sequence \eqref{seq2}, the second assertion follows immediately from the first.

\begin{lemma}\label{realsection} 
The map $\mathfrak{s}_F$ (defined in Lemma \ref{s_F}(iv)) is a section of $\mathfrak{q}: (\frac{L_n}{L_{n-2}})_\CC\longrightarrow (H_\CC^1)^{\otimes n}$ defined over $\RR$ that respects the Hodge and weight filtrations. In particular, it gives an isomorphism 
\[\frac{L_n}{L_{n-2}}\simeq (H^1)^{\otimes n}\oplus (H^1)^{\otimes n-1}\]
as real mixed Hodge structures.
\end{lemma}

\subsection{} In this paragraph, we describe the images of the extension $\mathbb{E}_{n,e}^\infty$ under $\Psi$ and $\Phi\circ\Psi$.

\begin{prop}\label{harmonicvolume}
\begin{itemize}
\item[(a)] $\Psi(\mathbb{E}_{n,e}^\infty)$ is the class of the map that given $c\in (H_\CC^1)^{\otimes n}$, $d\in (H_\CC^1)^{\otimes n-1}$, it sends $c\otimes d$ to $PD^{\otimes n-1}(d)(r_\ZZ\circ\mathfrak{s}_F(c))$. More explicitly, if $\beta_j\in\pi_1(U,e)$ ($1\leq j\leq 2g$) are such that $\{[\beta_j]\}$ is a basis of $H_1(X,\ZZ)$, and $\omega_1,\ldots,\omega_n\in\mathcal{H}_\CC^1(X)$, $\Psi(\mathbb{E}_{n,e}^\infty)$ is the class of the map that sends
\[
[\omega_1]\otimes\ldots\otimes[\omega_n]\otimes (PD^{\otimes n-1})^{-1}([\beta_{j_1}]\otimes\ldots\otimes[\beta_{j_{n-1}}])\]
to
\[\int\limits_{(\beta_{j_1}-1)\ldots(\beta_{j_{n-1}}-1)}\omega_1\ldots\omega_n+\sum\limits_i\omega_1\ldots\nu(\omega_i\otimes\omega_{i+1})\ldots\omega_n.\]
\item[(b)] $\Phi\circ\Psi(\mathbb{E}_{n,e}^\infty)$ is the map that given $c\in (H_\ZZ^1)^{\otimes n}$, $d\in (H_\ZZ^1)^{\otimes n-1}$, it sends $c\otimes d$ to $PD^{\otimes n-1}(d)(r_\ZZ\circ\mathfrak{s}_F(c))$ $\mod \ZZ$. More explicitly, for $\gamma_j\in\pi_1(U,e)$ ($1\leq j\leq n-1$), and $\omega_1,\ldots,\omega_n\in\mathcal{H}_\RR^1(X)$ with integral periods, $\Phi\circ\Psi(\mathbb{E}_{n,e}^\infty)$ sends
\[
[\omega_1]\otimes\ldots\otimes[\omega_n]\otimes (PD^{\otimes n-1})^{-1}([\gamma_1]\otimes\ldots\otimes[\gamma_{n-1}])\]
to
\[\int\limits_{(\gamma_{1}-1)\ldots(\gamma_{n-1}-1)}\omega_1\ldots\omega_n+\sum\limits_i\omega_1\ldots\nu(\omega_i\otimes\omega_{i+1})\ldots\omega_n~\mod\ZZ.\]
\end{itemize}
\end{prop}
 
\begin{proof}
(a) We track $\mathbb{E}^\infty_{n,e}$ through different steps of \eqref{eqPsisteps}. The element in $J\inhom\left((H^1)^{\otimes n},(H^1)^{\otimes n-1}\right)$ corresponding to $\mathbb{E}_{n,e}^\infty$ under the isomorphism of Carlson is the class of $r_\ZZ\circ\mathfrak{s}_F$. (See Paragraph \ref{Carlsoniso}.) That the latter goes to the described element of $J((H^1)^{\otimes 2n-1})^\vee$ is clear. For the second assertion, define $r_\ZZ$ using the basis $\{[\beta_j]\}$, and then the assertion follows on noting that $r_\ZZ\circ\mathfrak{s}_F([\omega_1]\otimes\ldots\otimes[\omega_n])$, by its definition, pairs with the element $[\beta_{j_1}]\otimes\ldots\otimes[\beta_{j_{n-1}}]\in (H_1)_\ZZ^{\otimes n-1}$ in the desired fashion. (See \eqref{eqr_Z} and Lemma \ref{s_F}(i),(iv).)\\
(b) The section $\mathfrak{s}_F$ is defined over $\RR$, and hence so is $r_\ZZ\circ\mathfrak{s}_F$. Thus the map
\[
c\otimes d\mapsto PD^{\otimes n-1}(d)(r_\ZZ\circ\mathfrak{s}_F(c))
\]
of Part (a) is also defined over $\RR$. The first assertion follows. The explicit description of Part (a) implies that (with $\beta_j$ as in Part (a)) $\Phi\circ\Psi(\mathbb{E}_{n,e}^\infty)$ sends
\[
[\omega_1]\otimes\ldots\otimes[\omega_n]\otimes (PD^{\otimes n-1})^{-1}([\beta_{j_1}]\otimes\ldots\otimes[\beta_{j_{n-1}}])\]
to
\[\int\limits_{(\beta_{j_1}-1)\ldots(\beta_{j_{n-1}}-1)}\omega_1\ldots\omega_n+\sum\limits_i\omega_1\ldots\nu(\omega_i\otimes\omega_{i+1})\ldots\omega_n~\mod\ZZ.\]
To get the basis-independent formula, in $H_1(X,\ZZ)^{\otimes n-1}$ we write 
\[
[\gamma_1]\otimes\ldots\otimes[\gamma_{n-1}]=\sum\limits_{j_1,\ldots,j_{n-1}}c_{j_1,\ldots,j_{n-1}}[\beta_{j_1}]\otimes\ldots\otimes[\beta_{j_{n-1}}],
\]
where the coefficients are all integers. In view of the isomorphisms \eqref{ImodI^2} and \eqref{quotientsofI}, the element
\[
\lambda:=(\gamma_1-1)\ldots(\gamma_{n-1}-1)-\sum\limits_{j_1,\ldots,j_{n-1}}c_{j_1,\ldots,j_{n-1}}(\beta_{j_1}-1)\ldots(\beta_{j_{n-1}}-1)\in I^{n-1},
\]
where $I\in\ZZ[\pi_1(U,e)]$ is the augmentation ideal, actually belongs to $I^n$. Thus
\begin{eqnarray*}
\int\limits_\lambda~\omega_1\ldots\omega_n+\sum\limits_i\omega_1\ldots\nu(\omega_i\otimes\omega_{i+1})\ldots\omega_n=\int\limits_{\lambda}\omega_1\ldots\omega_n~\in\ZZ,
\end{eqnarray*}
as $\lambda\in I^n$ and the $\omega_i$ have integer periods. This gives the desired conclusion.
\end{proof}

\begin{rem} (1) The use of a basis in Part (a) of the proposition is just to make the map well-defined.\\ 
(2) Let $K$ be the kernel of the cup product $H^1\otimes H^1\rightarrow H^2(X)$. The map $\Phi\circ\Psi(\mathbb{E}_{n,e}^\infty)$ can be thought of as an analog of the pointed harmonic volume 
\[I_e\in\Hom(K_\ZZ\otimes H^1_\ZZ,\RR/\ZZ)\]
of B. Harris \cite{Harris}. Pulte \cite{Pulte} showed that $I_e$ corresponds under the isomorphisms 
\[
\Ext(K,H^1)\stackrel{\text{Carlson}}{\cong}J\inhom(K,H^1)\stackrel{\text{Poincare duality}}{\cong} J\inhom(K\otimes H^1,\ZZ(0))\cong \Hom(K_\ZZ\otimes H^1_\ZZ,\RR/\ZZ) 
\] 
to the extension given by the sequence
\begin{equation}\label{Pulteseq}
\begin{array}{ccccccccc}
0& \longrightarrow & \displaystyle{\frac{L_{1}}{L_{0}}(X,e)} & \longrightarrow &\displaystyle{\frac{L_2}{L_{0}}(X,e)} &\longrightarrow &\displaystyle{\frac{L_2}{L_{1}}(X,e)} &\longrightarrow &0.\\
 &             &    \text{\rotatebox{90}{$\cong$}} & & & &\text{\rotatebox{90}{$\cong$}} & &\\
 & & H^1 & & & & K & &
\end{array}
\end{equation}
\end{rem}

\section{Algebraic cycles $\Delta_{n,e}$ and $Z^\infty_{n,e}$}\label{construction_of_cycles}

\subsection{Notation} Given a variety $Y$ over a field $K$, $\mathcal{Z}_i(Y)$ (resp. $\mathcal{Z}^i(Y)$) denotes the group of algebraic cycles of dimension (resp. codimension) $i$, and $\CH_i(Y)$ (resp. $\CH^i(Y)$) is $\mathcal{Z}_i(Y)$ (resp. $\mathcal{Z}^i(Y)$) modulo rational equivalence.\footnote[2]{Note that in our notation, $\CH_i(Y)$ is merely an abelian group, and not a functor from $K$-schemes to abelian groups.} As usual $\mathcal{Z}(Y):=\bigoplus\mathcal{Z}^i(Y)$ and $\CH(Y):=\bigoplus \CH^i(Y)$.  Notation-wise, we do not distinguish between an algebraic cycle and its class in the corresponding Chow group. Given $Y$ and $Y'$ of dimensions $d$ and $d'$, the group of degree zero correspondences from $Y$ to $Y'$ is $\text{Cor}(Y,Y'):=\mathcal{Z}_d(Y\times Y')$. If $f:Y\rightarrow Y'$ is a morphism, the graph of $f$ is denoted by $\Gamma_f$; it is an element of $\text{Cor}(Y,Y')$. We use the standard notation (lower star) for push-forwards along morphisms. Given algebraic cycles $Z\in\mathcal{Z}_i(Y)$ and $Z'\in\mathcal{Z}_j(Y')$, $Z\times Z'\in\mathcal{Z}_{i+j}(Y\times Y')$ denotes the Cartesian product. Given $Z\in \mathcal{Z}_i(Y\times Y')$, ${}^tZ$ is the transpose of $Z$; it is an element of $\mathcal{Z}_i(Y'\times Y)$. Finally, if $Y$ is a smooth variety over a subfield of $\CC$, $\mathcal{Z}^\hhom_i(Y)$ (resp. $\CH^\hhom_i(Y)$) refers to the subgroup of null-homologous cycles in $\mathcal{Z}_i(Y)$ (resp. $\CH_i(Y)$).

\subsection{A construction of Gross and Schoen}\label{parGSrecollection}

In this paragraph, we recall a construction of Gross and Shoen \cite{GS}. Let $m$ be a positive integer. By convention, we set $X^0=\Spec~\CC$. For (possibly empty) $T\subset\{1,\ldots,m\}$, let $p_T:X^m\rightarrow X^{|T|}$ be the projection map onto the coordinates in $T$, and $q_T:X^{|T|}\rightarrow X^m$ be the embedding that is a right inverse to $p_T$ and fills the coordinates that are not in $T$ by $e$. For instance, if $m=3$ and $T=\{2,3\}$, 
\[
(x_1,x_2,x_3)\stackrel{p_T}{\mapsto} (x_2,x_3)\hspace{.3in}\text{and}\hspace{.5in} (x_1,x_2)\stackrel{q_T}{\mapsto} (e,x_1,x_2).
\]
In general, the composition $q_T\circ p_T:X^m\rightarrow X^m$ is the morphism that keeps the $T$ coordinates unchanged, and replaces the rest by $e$. Let 
\[P_e=\sum\limits_{T}(-1)^{|T^c|}\Gamma_{q_T\circ p_T} \in \text{Cor}(X^m, X^m),\]
where $T^c$ denotes the complement of $T$. For the proof of the following result, see \cite{GS}.

\begin{thm} 
If $i<m$, the map $(P_e)^h_\ast: H_i(X^m)\rightarrow H_i(X^m)$ induced by $P_e$ on homology is zero.
\end{thm}

\noindent Let $(P_e)_\ast$ be the push forward map $\mathcal{Z}(X^m)\rightarrow \mathcal{Z}(X^m)$ defined by the correspondence $P_e$. Then
\[
(P_e)_\ast=\sum\limits_{T}(-1)^{|T^c|} (q_T\circ p_T)_\ast.
\] 
In view of commutativity of the diagram
\[
\begin{tikzpicture}
  \matrix (m) [matrix of math nodes, column sep=2em, row sep=2.5em]
    {	  \mathcal{Z}_i(X^m)& \mathcal{Z}_i(X^m)\\
		H_{2i}(X^m,\CC)& H_{2i}(X^m,\CC),\\};
  { [start chain] \chainin (m-1-1);
    \chainin (m-1-2) [join={node[above,labeled] {(P_e)_\ast}}];}
  {[start chain] \chainin (m-2-1);
    \chainin (m-2-2) [join={node[above,labeled] {(P_e)^h_\ast}}];}
	{[start chain] \chainin (m-1-2);
		\chainin (m-2-2) [join={node[right,labeled] {}}];}
	{[start chain] \chainin (m-1-1);
		\chainin (m-2-1) [join={node[left,labeled] {}}];}
  \end{tikzpicture}
\]
where the vertical maps are class maps, it follows from the previous theorem that if $2i<m$, then
\[(P_e)_\ast(\mathcal{Z}_i(X^m))\subset \mathcal{Z}_i^\hhom(X^m).\]
This gives a way of constructing null-homologous cycles.\\

\noindent{\bf Example.} For $m\geq 2$, denote by $\Delta^{(m)}(X)$ the diagonal copy of $X$ in $X^m$, i.e. 
\[
\{(x,x,\ldots,x):x\in X\}\in \mathcal{Z}_1(X^m).
\] 
For $m\geq 3$, by the previous observation, the {\it modified diagonal cycle} $(P_e)_\ast(\Delta^{(m)}(X))$ is null-homologous. As it is pointed out in \cite{GS}, this cycle has zero Abel-Jacobi image if $m>3$. On the other hand, if $m=3$, this cycle, which was first defined by Gross and Kudla in \cite{GK} and then studied more by Gross and Schoen in \cite{GS}, is well-known to be interesting. It is easy to see from its definition that
\begin{eqnarray*}
(P_e)_\ast(\Delta^{(3)}(X))&=&\Delta^{(3)}(X)-\{(e,x,x):x\in X\}-\{(x,e,x):x\in X\}-\{(x,x,e):x\in X\}\\
&&+~\{(e,e,x):x\in X\}+\{(e,x,e):x\in X\}+\{(x,e,e):x\in X\}.
\end{eqnarray*} 
We denote this cycle by $\Delta_{KGS,e}$ , the modified diagonal cycle of Kudla, Gross and Schoen.\\

\noindent Note that
\[
(P_e)_\ast(\Delta^{(2)}(X))=\Delta^{(2)}(X)-\{e\}\times X-X\times\{e\},
\]
which is homologically nontrivial.

\subsection{}\label{pardefineDelta_n} Let $n\geq 2$. In this paragraph, we use the construction of Gross and Schoen to define a null-homologous cycle $\Delta_{n,e}\in\mathcal{Z}_{n-1}(X^{2n-1})$, which will play a crucial role in the remainder of the paper. We use the notation of Paragraph \ref{parGSrecollection} with $m=2n-1$.\\

For $0<i<n$, let $\delta_i:X^{n-1}\rightarrow X^{n}$ be the embedding 
\[
(x_1,\ldots,x_{n-1})\mapsto (x_1,\ldots,x_i,x_i,\ldots,x_{n-1}). 
\] 
Then ${}^t\Gamma_{\delta_i}\in\mathcal{Z}_{n-1}(X^{2n-1})$, and thus $(P_e)_\ast ({}^t\Gamma_{\delta_i})$ is null-homologous. We define 
\[
\Delta_{n,e}:=(P_e)_\ast\left(\sum\limits_i(-1)^{i-1}~{}^t\Gamma_{\delta_i}\right)=\sum\limits_i(-1)^{i-1} (P_e)_\ast ({}^t\Gamma_{\delta_i})~\in\mathcal{Z}_{n-1}^\hhom(X^{2n-1}).
\]
It is clear from the definition that $\Delta_{2,e}$ is simply the modified diagonal cycle $\Delta_{KGS,e}$ of Gross, Kudla, and Schoen in $X^3$.

\subsection{}\label{boundry_inverse_Delta_2} In this paragraph, we realize the cycle $\Delta_{n,e}$ as the boundary of a chain. This will be particularly important when later we study the image of $\Delta_{n,e}$ under the Abel-Jacobi map.\\

Let $\Lambda_n$ be the closed subvariety 
\[
\{(x_1,x_1,x_1,x_2,x_2,\ldots,x_{n-1},x_{n-1}): x_i\in X\}
\] 
of $X^{2n-1}$, where each $x_i$ ($i>1$) is appearing in exactly two coordinates. It is a copy of $X^{n-1}$ embedded in $X^{2n-1}$ via the map 
\[
(x_1,\ldots,x_{n-1})\mapsto (x_1,x_1,x_1,x_2,x_2,\ldots,x_{n-1},x_{n-1}),
\]
and can also be thought of as an element of $\mathcal{Z}_{n-1}(X^{2n-1})$. It is easy to see that
\begin{equation}\label{eq1july17}
(P_e)_\ast(\Lambda_n)=\Delta_{2,e}\times \overbrace{(P_e)_\ast(\Delta^{(2)}(X))\times\ldots\times (P_e)_\ast(\Delta^{(2)}(X))}^{\text{$n-2$ factors}}.
\end{equation}
Let $\partial^{-1}(\Delta_{2,e})$ be an integral 3-chain in $X^3$ whose boundary is $\Delta_{2,e}$. (See for instance the proof of Lemma 2.3 of \cite{DRS} for such a 3-chain.) Then $(P_e)_\ast(\Lambda_n)$ is the boundary of 
\[
\partial^{-1}(\Delta_{2,e})\times \overbrace{(P_e)_\ast(\Delta^{(2)}(X))\times\ldots\times (P_e)_\ast(\Delta^{(2)}(X))}^{\text{$n-2$ factors}}~=:\partial^{-1}(P_e)_\ast(\Lambda_n).
\]
\noindent It is clear that each ${}^t\Gamma_{\delta_i}$ is a copy of $\Lambda_n$. Specifically, ${}^t\Gamma_{\delta_i}=(\sigma_i)_\ast(\Lambda_n)$ where $\sigma_i$ is the automorphism of $X^{2n-1}$ that sends $(x_1,\ldots,x_{2n-1})$ to
\[ 
(x_4,x_6,\ldots,x_{2i},x_1,x_2,x_{2i+2},x_{2i+4},\ldots,x_{2n-2},x_5,x_7,\ldots,x_{2i+1}, x_3,x_{2i+3},x_{2i+5},\ldots,x_{2n-1}).  
\]

\begin{lemma}
$(P_e)_\ast$ and $(\sigma_i)_\ast$ commute (as maps $\mathcal{Z}(X^{2n-1})\rightarrow \mathcal{Z}(X^{2n-1})$). 
\end{lemma}
\begin{proof}
With abuse of notation, suppose $\sigma_i$ is the permutation of $1,2,\ldots,2n-1$ such that 
\[
\sigma_i(x_1,\ldots,x_{2n-1})=(x_{\sigma_i^{-1}(1)},x_{\sigma_i^{-1}(2)},\ldots,x_{\sigma_i^{-1}(2n-1)}).
\]
Then for each subset $T$ of $\{1,2,\ldots,2n-1\}$, $q_T\circ p_T\circ\sigma_i=\sigma_i\circ q_{\sigma_i^{-1}T}\circ p_{\sigma_i^{-1}T}$. We have
\begin{eqnarray*}
(P_e)_\ast\circ(\sigma_i)_\ast&=&\left(\sum\limits_T(-1)^{|T^c|}(q_T\circ p_T)_\ast\right)(\sigma_i)_\ast\\
&=&\sum\limits_T(-1)^{|T^c|}(q_T\circ p_T\circ \sigma_i)_\ast\\
&=&\sum\limits_T(-1)^{|T^c|}(\sigma_i\circ q_{\sigma_i^{-1}T}\circ p_{\sigma_i^{-1}T})_\ast\\
&=&(\sigma_i)_\ast\left(\sum\limits_T(-1)^{|T^c|}(q_{\sigma_i^{-1}T}\circ p_{\sigma_i^{-1}T})_\ast\right)\\
&=&(\sigma_i)_\ast\circ (P_e)_\ast.
\end{eqnarray*}
\end{proof}

It follows from the lemma that 
\begin{equation}\label{eq2july17}
(\sigma_i)_\ast\left((P_e)_\ast(\Lambda_n)\right)=(P_e)_\ast~({}^t\Gamma_{\delta_i}),
\end{equation}
and hence 
\[
\partial \left((\sigma_i)_\ast (\partial^{-1}(P_e)_\ast(\Lambda_n))\right)=(P_e)_\ast~({}^t\Gamma_{\delta_i}).
\]
We set
\[
\partial^{-1}\Delta_{n,e}:=\sum\limits_i(-1)^{i-1} (\sigma_i)_\ast (\partial^{-1}(P_e)_\ast(\Lambda_n)).
\]
It is immediate from the above that the boundary of this chain is $\Delta_{n,e}$.

\begin{rem}
In view of \eqref{eq1july17}, \eqref{eq2july17} and definition of $\Delta_{n,e}$, if $\Delta_{2,e}$ is torsion in $\CH_1^\hhom(X^3)$, then so is $\Delta_{n,e}$ in $\CH_{n-1}^\hhom(X^{2n-1})$ for every $n$.
\end{rem}

\subsection{}\label{pardefZ} In this paragraph, we define another family of null-homologous cycles that will be used later on. Let $n\geq 2$. Given $y\in X(\CC)$, for $0<i<n$, let $Z^y_{n,i}\in\mathcal{Z}_{n-1}(X^{2n-1})$ be 
\[
\{(x_1,\ldots,x_{i-1},x_i,x_i,x_{i+1},\ldots,x_{n-1},x_1,\ldots,x_{i-1},y,x_{i+1},\ldots,x_{n-1}):x_1,\ldots,x_{n-1}\in X\}.
\] 
Here each $x_j$ appears in exactly two coordinates. There are different ways of thinking about this cycle. For instance,
\[
Z^y_{n,i}=(\pi_{n+i,y})_\ast {}^t\Gamma_{\delta_i},
\]
where $\pi_{n+i,y}$ is the map $X^{2n-1}\rightarrow X^{2n-1}$ that replaces the $(n+i)$-th coordinate by $y$, and keeps the other coordinates unchanged.\\

It is clear that the cycle $Z^\infty_{n,i}-Z^e_{n,i}$ is null-homologous. For future reference, here we explicitly define a chain whose boundary is $Z^\infty_{n,i}-Z^e_{n,i}$. Choose a path $\gamma^\infty_e$ in $X$ from $e$ to $\infty$, and let 
\[
C^\infty_{n,e}:=\Delta^{(2)}(X)^{n-1}\times \gamma^\infty_e=\{(x_1,x_1,\ldots,x_{n-1},x_{n-1},\gamma^\infty_e(t)):x_i\in X, t\in[0,1]\}.
\]
One clearly has
\[
\partial C^\infty_{n,e}=\Delta^{(2)}(X)^{n-1}\times\{\infty\}-\Delta^{(2)}(X)^{n-1}\times\{e\}.
\]
For $0<i<n$, let $\tau_i$ be the automorphism of $X^{2n-1}$ that maps $(x_1,\ldots,x_{2n-1})$ to 
\[(x_1,x_3,\ldots,x_{2(i-1)-1},x_{2i-1},x_{2i},x_{2(i+1)-1},\ldots,x_{2(n-1)-1},x_2,x_4,\ldots,x_{2(i-1)},x_{2n-1},x_{2(i+1)},\ldots,x_{2(n-1)}),  
\] 
which is designed so that 
\[(\tau_i)_\ast\left(\Delta^{(2)}(X)^{n-1}\times \{y\}\right)=Z^y_{n,i}\] 
for every $y$. Then
\begin{equation}\label{eqpartialZ}
\partial(\tau_i)_\ast (C^\infty_{n,e})=Z^\infty_{n,i}-Z^e_{n,i}.
\end{equation}
We put together all the $Z^\infty_{n,i}-Z^e_{n,i}$ and define
\[
Z^\infty_{n,e}:=\sum\limits_{i=1}^{n-1}(-1)^{i-1}(Z^\infty_{n,i}-Z^e_{n,i})~\in\mathcal{Z}_{n-1}^{\hhom}(X^{2n-1}).
\]

\subsection{Remark}\label{remcyclesoverk} While we worked over $\CC$ in this section, it is clear that the constructions of $\Delta_{n,e}$ and $Z^\infty_{n,e}$ remain valid over any field $K$ that can be embedded into $\CC$. More precisely, if $X_0$ is a geometrically connected smooth projective curve over $K$, and $e,\infty\in X_0(K)$, the above constructions give null-homologous cycles $\Delta_{n,e}$ and $Z^\infty_{n,e}$ in $\mathcal{Z}_{n-1}(X_0^{2n-1})$ (or in $\CH_{n-1}(X_0^{2n-1})$).

\section{Statement of the main theorem}
Our goal in this section is to state the main result of the paper, which expresses the extension $\mathbb{E}^\infty_{n,e}$ in terms of the Abel-Jacobi images of the cycles $\Delta_{n,e}$ and $Z^\infty_{n,e}$.

\subsection{Review of Griffiths' Abel-Jacobi maps}
Let $Y$ be a smooth projective variety over $\CC$. The $n$-th Abel-Jacobi map associated to $Y$ is the map\footnote[2]{That our notation for this map does not incorporate $Y$ or $n$ should not lead to any confusion.}
\[
\AJ: \mathcal{Z}_n^{\hhom}(Y)\rightarrow JH^{2n+1}(Y)^\vee
\] 
defined as follows. First note that the restriction map $\left(H^{2n+1}_\CC(Y)\right)^\vee\rightarrow \left(F^{n+1}H^{2n+1}(Y)\right)^\vee$ gives an isomorphism
\[
JH^{2n+1}(Y)^\vee\cong \frac{\left(F^{n+1}H^{2n+1}(Y)\right)^\vee}{H_{2n+1}(Y,\ZZ)},
\]
where an element of $H_{2n+1}(Y,\ZZ)$ is considered as an element of $\left(F^{n+1}H^{2n+1}(Y)\right)^\vee$ via integration. Thus we can equivalently define $\AJ$ as a map into 
\[\frac{\left(F^{n+1}H^{2n+1}(Y)\right)^\vee}{H_{2n+1}(Y,\ZZ)}.\]
Given a null-homologous $n$-dimensional cycle $Z$ on $Y$, there is an integral chain $C$ such that $\partial C=Z$. Given $c\in F^{n+1}H^{2n+1}(Y)$, take a representative $\omega\in F^{n+1}E^{2n+1}_\CC(Y)$, and set
\[\int\limits_{C}c=\int\limits_C\omega.
\]
One can show that this is independent of the choice of $\omega$. Then 
\[\AJ(Z)\in\frac{\left(F^{n+1}H^{2n+1}(Y)\right)^\vee}{H_{2n+1}(Y,\ZZ)}\]
is defined to be the class of the map 
\[c\mapsto \int\limits_{C}c.\]
The ambiguity in having to choose $C$ is resolved by modding out by $H_{2n+1}(Y,\ZZ)$. If one insists on having $\AJ(Z)\in JH^{2n+1}(Y)^\vee$, it is the class of any map $H_\CC^{2n+1}(Y)\rightarrow \CC$ whose restriction to $F^{n+1}H^{2n+1}(Y)$ is the map $\displaystyle{\int\limits_C}$ above.\\

One can show that $\AJ$ factors through $\CH_n^{\hhom}(Y)$. The induced map 
\[\CH_n^{\hhom}(Y)\rightarrow JH^{2n+1}(Y)^\vee\] 
is also called Abel-Jacobi, and with abuse of notation we denote it by $\AJ$ as well.

\subsection{Notation} We adopt the following notation for the Kunneth decomposition of cohomology. Given manifolds $M$ and $N$, we think of $H^i(M)\otimes H^j(N)$ (singular or de Rham cohomology) as a subspace of $H^{i+j}(M\times N)$. Given $c\in H^i(M)$, $d\in H^j(N)$, the element $c\otimes d$ of $H^{i+j}(M\times N)$ is $pr_1^\ast(c)\wedge pr_2^\ast(d)$, where $pr_i$ the the projection of $M\times N$ onto its $i^\text{th}$ factor. We adopt a similar notation for differential forms: given $\omega$ and $\phi$ differential forms on $M$ and $N$, we refer to the differential form $pr_1^\ast(\omega)\wedge pr_2^\ast(\phi)$ on $M\times N$ by $\omega\otimes\phi$. Similar notation is used for more than two factors.

\subsection{}\label{par2secmaps} For $n\geq 1$, let $h_n$ be the composition of the Abel-Jacobi map 
\[
CH^{\hhom}_{n-1}(X^{2n-1})\longrightarrow J H^{2n-1}(X^{2n-1})^\vee 
\]  
with the map
\[
JH^{2n-1}(X^{2n-1})^\vee\longrightarrow J((H^1)^{\otimes 2n-1})^\vee
\]
induced by the Kunneth inclusion $(H^1)^{\otimes 2n-1}\subset H^{2n-1}(X^{2n-1})$. It is easy to see from definitions that if $Z\in \mathcal{Z}^{\hhom}_{n-1}(X^{2n-1})$ and $C$ is an integral chain in $X^{2n-1}$ whose boundary is $Z$, $h_n(Z)$ is the class of the map that, given harmonic 1-forms $\omega_1,\ldots,\omega_{2n-1}$ on $X$, it sends 
\begin{equation}\label{eq100}
[\omega_1]\otimes\ldots\otimes[\omega_{2n-1}]\mapsto \int\limits_C \omega_1\otimes\ldots\otimes\omega_{2n-1}.
\end{equation} 
Note that $h_1$ is just the ``classical" Abel-Jacobi map $\CH^\hhom_0(X)\rightarrow J (H^1)^\vee$.\\

If $Z$ and $C$ are as above, since the map \eqref{eq100} is defined over $\RR$, 
\[\Phi(h_n(Z)): (H^1_\ZZ)^{\otimes 2n-1}\rightarrow\RR/\ZZ\] 
is the map that, given harmonic forms $\omega_1,\ldots,\omega_{2n-1}$ on $X$ with integral periods, it maps
\[
[\omega_1]\otimes\ldots\otimes[\omega_{2n-1}]\mapsto \int\limits_C \omega_1\otimes\ldots\otimes\omega_{2n-1}~~\mod\ZZ.
\]    
(See Paragraph \ref{RmodZ} and Paragraph \ref{PhiandPsi}.)

\subsection{}
Now we are ready to state the main result. 

\begin{thm}\label{main1}
Let $n\geq 2$. We have
\begin{equation}\label{eqmainthm1}
\Psi(\mathbb{E}^\infty_{n,e})=(-1)^{\frac{n(n-1)}{2}} h_n\left(\Delta_{n,e}-Z^\infty_{n,e}\right).
\end{equation}
\end{thm}

When $n=2$, a slightly weaker of this is due to Darmon, Rotger, and Sols \cite{DRS}. (See the next section.) 

\section{$n=2$ case of Theorem \ref{main1} - A formula of Darmon et al revisited}\label{section_DRS_result_reformulation}

\subsection{Independence of $-\Psi(\mathbb{E}^\infty_{2,e})+h_2(Z^\infty_{2,e})$ from $\infty$}

\begin{lemma}\label{indep_from_infty}
The element $-\Psi(\mathbb{E}^\infty_{2,e})+h_2(Z^\infty_{2,e})$ is independent of the point $\infty\neq e$, i.e. if $\infty_1,\infty_2\neq e$, then 
\[
-\Psi(\mathbb{E}^{\infty_1}_{2,e})+h_2(Z^{\infty_1}_{2,e})=-\Psi(e^{\infty_2}_{2,e})+h_2(Z^{\infty_2}_{2,e}).
\]
\end{lemma}

\begin{proof}
Let $\infty_1,\infty_2\neq e$ be distinct. After passing to $\Hom\left((H^1_\ZZ)^{\otimes 3},\RR/\ZZ\right)$ via $\Phi$, in view of Proposition \ref{harmonicvolume}(b), we need to show that if $\omega,\rho,\eta$ are harmonic forms with integral periods on $X$, and $\gamma_\eta\in\pi_1(X-\{\infty_1,\infty_2\},e)$ is such that its homology class in $H_1(X,\ZZ)$ is $\PD([\eta])$, then 
\[
-\int\limits_{\gamma_\eta}\omega\rho+\nu_{\infty_1}(\omega\otimes\rho)~+~\int\limits_X\omega\wedge\rho~\int\limits_e^{\infty_1}\eta\stackrel{\ZZ}{\equiv} -\int\limits_{\gamma_\eta}\omega\rho+\nu_{\infty_2}(\omega\otimes\rho)~+~\int\limits_X\omega\wedge\rho~\int\limits_e^{\infty_2}\eta,
\]
or equivalently
\begin{equation}\label{eq11}
-\int\limits_{\gamma_\eta}\nu_{\infty_1}(\omega\otimes\rho)-\nu_{\infty_2}(\omega\otimes\rho)~+~\int\limits_X\omega\wedge\rho\int\limits_{\infty_2}^{\infty_1}\eta \in\ZZ,
\end{equation}
where the integrals of $\eta$ are over any path in $X$ with the specified end points. Fix $\omega$ and $\rho$. For brevity we write $\nu_i$ for $\nu_{\infty_i}(\omega\otimes\rho)$. Note that if $\omega\wedge\rho$ is exact on $X$, then the statement clearly holds, as then $\nu_i\in\mathcal{H}^1_\CC(X)^\perp$ and $\nu_1-\nu_2$, being a closed element of $\mathcal{H}^1_\CC(X)^\perp$, is exact, so that the number above is simply zero. (See the proof of Lemma \ref{lemlift1}.) So we may assume $\omega\wedge\rho$ is not exact on $X$. Then the 1-form $\nu_1-\nu_2$ satisfies the following properties:
\begin{itemize}
\item[(i)] It is meromorphic on $X$, holomorphic on $X-\{\infty_1,\infty_2\}$, with logarithmic poles at $\infty_1$ and $\infty_2$ with residues $\frac{a}{2\pi i}$ and $-\frac{a}{2\pi i}$ respectively for some integer $a\neq 0$.
\item[(ii)] Its cohomology class in $H^1(X-\{\infty_1,\infty_2\})$ is real, i.e. it can be written on $X-\{\infty_1,\infty_2\}$ as the sum of an exact form and a real closed form.
\end{itemize}
Indeed, (i) follows from that both $\nu_1$ and $\nu_2$ are of type (1,0), and $d\nu_1=d\nu_2=-\omega\wedge\rho$ on $X-\{\infty_1,\infty_2\}$, so that $\nu_1-\nu_2$ is holomorphic on $X-\{\infty_1,\infty_2\}$. For the behavior at $\infty_i$, note that $\nu_i\in E^1(X\log\infty_i)$. The statement about the residues is immediate from Lemma \ref{lemlift1}(iv) ($a=\int\limits_X\omega\wedge\rho$). Statement (ii) follows from that each form $\nu_i$ can be written as a real form on $X-\{\infty_i\}$ plus an exact form on the same space. (See Lemma \ref{lemlift1}(iii).)\\     

The statement \eqref{eq11} now follows from the following lemma. 
\end{proof}
\begin{lemma}
Let $\infty_1,\infty_2\neq e$, and $\zeta$ be any 1-form satisfying conditions (i) and (ii) above. Then for any harmonic 1-form $\eta$ on $X$ with integral periods, 
\[
-\int\limits_{\gamma_\eta}\zeta~+a\int\limits_{\infty_2}^{\infty_1}\eta \in\ZZ,
\] 
where $\gamma_\eta\in\pi_1(X-\{\infty_1,\infty_2\},e)$ satisfies $\PD([\eta])=[\gamma_\eta]$.
\end{lemma}
\begin{proof}
First note that the integral $\int\limits_X\zeta\wedge\eta$ converges for any $\eta\in\mathcal{H}_\CC^1(X)$, as the integral of $\frac{dzd\bar{z}}{z}$ converges on the unit disk in $\CC$. Thus one gets a map $h:H^1_\CC\rightarrow\CC$ given by $[\eta]\mapsto\int\limits_X\zeta\wedge\eta$. We claim that this map takes integer values on $H^1_\ZZ$. Note that since $h$ vanishes on $F^1H^1$, by the remark in Paragraph \ref{RmodZ}, it suffices to show that it is defined over $\RR$. Suppose $\eta\in\mathcal{H}^1_\RR(X)$ has integer periods. The claim is established if we show $h([\eta])$ is real. We may assume that the map
\begin{equation}\label{eq5ch3revision1}
\int\eta :H_1(X,\ZZ)\rightarrow \ZZ
\end{equation}
is surjective, and that $\gamma_\eta\in\pi_1(X-\{\infty_1,\infty_2\},e)$ (Poincare dual to $[\eta]$ in $H_1(X,\ZZ)$) has a simple representative loop, which we also denote by $\gamma_\eta$. One can show that there is a Riemann surface $\tilde{X}$, a covering projection $\pi:\tilde{X}\rightarrow X$, and a deck transformation $T$ of $\pi$ such that
\begin{itemize}
\item[-] $\pi^\ast\eta=df$ for a real function $f$ on $\tilde{X}$.
\item[-] $fT-f$ is the constant function 1. 
\item[-] There is a lift $\tilde{\gamma}_\eta$ of $\gamma_\eta$, and a submanifold with boundary $X^{(0)}$ of $\tilde{X}$ such that $\partial X^{(0)}=T\tilde{\gamma}_\eta-\tilde{\gamma}_\eta$, and the restriction of $\pi$ to $X^{(0)}-\partial X^{(0)}$ is an isomorphism of Riemann surfaces onto $X-\gamma_\eta$.\footnote[2]{Such a covering projection is obtained by taking a copy $X^{(i)}$ of $X$ for each integer $i$, ``cutting" the $X^{(i)}$ along $\gamma_\eta$, and then gluing $X^{(i)}$ to $X^{(i+1)}$ appropriately along $\gamma_\eta$. The deck transformation simply sends a point in $X^{(i)}$ to its counterpart in $X^{(i+1)}$.}
\end{itemize}
Now let for each $i$, $D_i$ be an open disk around $\infty_i$ in $X$, small enough so that $\overline{D_1}\cap \overline{D_2}=\emptyset$ and $\overline{D_i}\cap \gamma_\eta=\emptyset$ (bar denoting closure). Denote by $\tilde{\infty}_i$ and $\tilde{D}_i$ the lift of $\infty_i$ and $D_i$ in $X^{(0)}$. Then we have
\begin{eqnarray*}
\int\limits_{X-{D_1\cup D_2}} \zeta\wedge\eta=\int\limits_{X^{(0)}-{\tilde{D}_1\cup \tilde{D}_2}} \pi^\ast\zeta\wedge\pi^\ast\eta&=&\int\limits_{X^{(0)}-{\tilde{D}_1\cup \tilde{D}_2}} -df\wedge\pi^\ast\zeta\\
&=&\int\limits_{X^{(0)}-{\tilde{D}_1\cup \tilde{D}_2}} -d(f\pi^\ast\zeta)\\
&=&-\int\limits_{\partial(X^{(0)}-{\tilde{D}_1\cup \tilde{D}_2})}f\pi^\ast\zeta\\
&=& \int\limits_{\tilde{\gamma}_\eta-T\tilde{\gamma}_\eta+\partial\tilde{D}_1+\partial\tilde{D}_2} ~f\pi^\ast\zeta\\
&=& \int\limits_{\tilde{\gamma}_\eta-T\tilde{\gamma}_\eta}f\pi^\ast\zeta~+~\int\limits_{\partial\tilde{D}_1+\partial\tilde{D}_2}f\pi^\ast\zeta.
\end{eqnarray*}
It follows that 
\[
\int\limits_{X-{D_1\cup D_2}} \zeta\wedge\eta=-\int\limits_{\gamma_\eta}\zeta~+~\int\limits_{\partial\tilde{D}_1+\partial\tilde{D}_2}f\pi^\ast\zeta~.
\]
We would like to know what happens as $D_i\rightarrow\{\infty_i\}$. Write
\[
\int\limits_{\partial\tilde{D}_i}f\pi^\ast\zeta=\int\limits_{\partial\tilde{D}_i}f(\tilde{\infty}_i)\pi^\ast\zeta~+~\int\limits_{\partial\tilde{D}_i}\left(f-f(\tilde{\infty}_i)\right)\pi^\ast\zeta.
\] 
Since $\zeta$ is holomorphic on $\tilde{D}_i-\tilde{\infty}_i$ with a pole of order 1 at $\infty_i$, and $f-f(\tilde{\infty}_i)$ is smooth and vanishes at $\tilde{\infty}_i$, the second term goes to zero as $D_i\rightarrow \{\infty_i\}$. The first term is equal to $2\pi i f(\tilde{\infty}_i)\text{res}_{\infty_i}(\zeta)$. Thus 
\begin{equation}\label{eq1111}
\int\limits_{X} \zeta\wedge\eta=-\int\limits_{\gamma_\eta}\zeta~+~a(f(\tilde{\infty}_1)-f(\tilde{\infty}_2)).
\end{equation}
The second term on the right is real as $a$ and $f$ are real. The first term is also real because the cohomology class of $\zeta$ in $H^1(X-\{\infty_1,\infty_2\})$ is real. Thus the claim is established.\\

Now it is easy to conclude the lemma. Let $\eta$ be as described in the statement. Without loss of generality we may assume that \eqref{eq5ch3revision1} is surjective, and that $\gamma_\eta$ has a simple representative loop. Then we know \eqref{eq1111}, and hence
\[
\int\limits_{X} \zeta\wedge\eta\stackrel{\ZZ}{\equiv}-\int\limits_{\gamma_\eta}\zeta~+~a\int\limits_{\infty_2}^{\infty_1}\eta.
\]
The left hand side (which is $h(\eta)$) is an integer.
\end{proof}

\subsection{} When $n=2$, Theorem \ref{main1} asserts that
\begin{equation}\label{eqDRSthmv2}
\Psi(\mathbb{E}^\infty_{2,e})=h_2(-\Delta_{2,e}+Z^\infty_{2,e}).
\end{equation} 
This is a slightly stronger version of a theorem of Darmon, Rotger, and Sols \cite[Theorem 2.5]{DRS}. Their result can be stated as to assert that, for every Hodge class $\xi$ of $(H^1)^{\otimes 2}$, one has
\begin{equation}\label{DRSraw}
\xi^{-1}\left(\Psi(\mathbb{E}^\infty_{2,e})\right)=\xi^{-1}\left(h_2(-\Delta_{2,e}+Z^\infty_{2,e})\right),
\end{equation}
where $\xi^{-1}:J((H^1)^{\otimes 3})^\vee\rightarrow J(H^1)^\vee$ is the map that sends $[f]\mapsto [f(\xi\otimes-)]$ for any $f\in \left((H_\CC^1)^{\otimes 3}\right)^\vee$. (This is well-defined because $\xi$ is a Hodge class.)\\

Let $\{\beta_j\}_j\subset\pi_1(U,e)$ be such that $\{[\beta_j]\}_j$ forms a basis of $H_1(X,\ZZ)$. For each $j$, let $\eta_j$ be the harmonic form on $X$ such that $\PD([\eta_j])=[\beta_j]$. In view of our description of $\Psi(\mathbb{E}^\infty_{2,e})$ given in Proposition \ref{harmonicvolume}, \eqref{DRSraw} is equivalent to that if $\xi=\sum [\omega_i]\otimes[\rho_i]$ with $\omega_i$ and $\rho_i$ harmonic forms on $X$ with integral periods, then the two maps $H^1_\CC\rightarrow \CC$ given by
\[
[\eta_j]\mapsto \int\limits_{\partial^{-1}\Delta_{2,e}}\sum\limits_i\omega_i\otimes\rho_i\otimes\eta_j 
\]
and 
\[
[\eta_j]\mapsto -\left(\int\limits_{\beta_j} \sum\omega_i\rho_i +\nu(\xi)\right)~+~\int\limits_{\Delta^{(2)}(X)} \xi\int\limits_{\gamma_e^\infty} \eta_j,
\]
represent the same class in $J(H^1)^\vee$. For this it suffices to verify that the restrictions of the two maps to $F^1H^1_\CC$ differ by (the restriction of) an element of $(H_1)_\ZZ$, and this is what Darmon, Rotger and Sols do in \cite{DRS}.\\

The argument given in \cite{DRS} combined with Lemma \ref{indep_from_infty} indeed implies \eqref{eqDRSthmv2}. To see this, let us start with an obvious observation. Suppose $A$, $B$, and $C$ are abelian groups. Then a map $f:A\otimes B\rightarrow C$ is zero if and only if, for every $a\in A$, the map $B\rightarrow C$ defined by $b\mapsto f(a\otimes b)$ is zero. Now suppose we have a map $f:(H^1_\CC)^{\otimes3}\rightarrow \CC$, defined over $\RR$. Then $[f]$ is trivial in $J((H^1)^{\otimes3})^\vee$ if and only if $\Phi([f])=0$, and since $f$ is defined over the reals, the latter amounts to that $\pi\circ f|_{(H^1_\ZZ)^{\otimes 3}}=0$, where $\pi:\RR\rightarrow\RR/\ZZ$ is the natural map. This is equivalent to that for every $\xi\in (H^1_\ZZ)^{\otimes 2}$, the map $H_\ZZ^1\rightarrow \RR/\ZZ$ given by $c\mapsto \pi\circ f(\xi\otimes c)$ is zero, or equivalently, the map $\xi^{-1}f:H^1_\CC\rightarrow\CC$ defined by $c\mapsto f(\xi\otimes c)$ is integer-valued on $H_\ZZ^1$. The latter by the remark in Paragraph \ref{RmodZ} is equivalent to that the restriction of $\xi^{-1}f$ to $F^1H^1$ coincides with that of an element of $(H_1)_\ZZ$.\\

In view of the above observation, \eqref{eqDRSthmv2} is equivalent to that, for every $\xi=[\omega]\otimes[\rho] \in (H^1_\ZZ)^{\otimes 2}$, where the $\omega$ and $\rho$ are harmonic forms on $X$ with integral periods, the restriction to $F^1H^1_\CC$ of the map $H^1_\CC\rightarrow \CC$ defined by
\[
[\eta_j]\mapsto \int\limits_{\partial^{-1}\Delta_{2,e}}\omega\otimes\rho\otimes\eta_j~+~\left(\int\limits_{\beta_j} \omega\rho +\nu(\xi)\right)~-~\int\limits_{\Delta^{(2)}(X)} \xi\int\limits_{\gamma_e^\infty} \eta_j
\]  
is equal to that of of an element of $(H_1)_\ZZ$. This is exactly Theorem 2.5 of \cite{DRS}, except that here $\xi$ is not necessarily a Hodge class, but rather merely an integral class. However, the argument in \cite{DRS} works just as well here too, as long as one can replace the point $\infty$ by a point at which certain technical conditions\footnote[2]{on the ``positioning" of $\infty$ relative to $\partial^{-1}\Delta_{2,e}$} hold. Lemma \ref{indep_from_infty} allows one to do this\footnote[3]{In \cite{DRS}, a similar task is performed by Lemma 1.3, which asserts that our Lemma \ref{indep_from_infty} holds after applying $\xi^{-1}$.}.

\subsection{} We close this section by noting that applying the map $\Phi$ to \eqref{eqDRSthmv2}, we see that, if $\omega,\rho,\eta$ are harmonic forms on $X$ with integral periods, and $\gamma_\eta\in\pi_1(U,e)$ is such that $\PD([\eta])=[\gamma_\eta]$ in homology of $X$, then
\begin{equation}\label{eqDRSv3}
\int\limits_{\partial^{-1}\Delta_{2,e}}\omega\otimes\rho\otimes\eta\stackrel{\ZZ}{\equiv}-\int\limits_{\gamma_\eta}(\omega\rho+\nu(\omega\otimes\rho))~+~\int\limits_X \omega\wedge\rho ~\int\limits_e^\infty \eta.
\end{equation}
(See Proposition \ref{harmonicvolume}(b).)

\section{Proof of the general case of Theorem \ref{main1}}\label{proofgeneralcase}

Our goal here is to use the contents of the previous sections to prove Theorem \ref{main1} in $n\geq 3$ case. We will equivalently show that the two sides of \eqref{eqmainthm1} have equal images under $\Phi$. Let $\omega_1,\ldots,\omega_n$ and $\eta_1,\ldots,\eta_{n-1}$ be harmonic forms on $X$ with integral periods, and for each $i$, $\gamma_i\in\pi_1(U,e)$ be such that $[\gamma_i]=\PD([\eta_i])$ in $H_1(X,\ZZ)$. All equalities below take place in $\RR/\ZZ$. We use the notation $[\ldots|\ldots]$ for $\ldots\otimes\ldots$, and for brevity denote $\frac{(n-3)(n-2)}{2}$ by $m$. The reader can refer to Section \ref{construction_of_cycles} to recall the definition of the chains and permutations that appear in the calculations.\\

We have   
\begin{eqnarray*}
\Phi(h_n(\Delta_{n,e}))[\omega_1|\ldots|\omega_n|\eta_1|\ldots|\eta_{n-1}]&=&\int\limits_{\partial^{-1}\Delta_{n,e}}[\omega_1|\ldots|\omega_n|\eta_1|\ldots|\eta_{n-1}]\\
&=& \sum\limits_{i=1}^{n-1}(-1)^{i-1} \int\limits_{(\sigma_i)_\ast (\partial^{-1}(P_e)_\ast(\Lambda_n))}[\omega_1|\ldots|\omega_n|\eta_1|\ldots|\eta_{n-1}].
\end{eqnarray*}
We also have
\[
\int\limits_{(\sigma_i)_\ast (\partial^{-1}(P_e)_\ast(\Lambda_n))}[\omega_1|\ldots|\omega_n|\eta_1|\ldots|\eta_{n-1}]=\int\limits_{\partial^{-1}(P_e)_\ast(\Lambda_n)}(\sigma_i)^\ast ([\omega_1|\ldots|\omega_n|\eta_1|\ldots|\eta_{n-1}]),
\]
and recalling how $\sigma_i$ permutes coordinates of $X^{2n-1}$, we see this is
\begin{eqnarray*}
&=&(-1)^{n+i-1+m}\int\limits_{\partial^{-1}(P_e)_\ast(\Lambda_n)}[\omega_i|\omega_{i+1}|\eta_i|\omega_1|\eta_1|...|\omega_{i-1}|\eta_{i-1}|\omega_{i+2}|\eta_{i+1}|\ldots|\omega_n|\eta_{n-1}]\\
&=&(-1)^{n+i-1+m}\int\limits_{(\partial^{-1}\Delta_{2,e})\times((P_e)_\ast\Delta^{(2)}(X))^{n-2}}[\omega_i|\omega_{i+1}|\eta_i|\omega_1|\eta_1|...|\omega_{i-1}|\eta_{i-1}|\omega_{i+2}|\eta_{i+1}|\ldots|\omega_n|\eta_{n-1}]\\
&=&(-1)^{n+i-1+m}\int\limits_{\partial^{-1}\Delta_{2,e}}[\omega_i|\omega_{i+1}|\eta_i]~\prod\limits_{j=1}^{i-1}\int\limits_{(P_e)_\ast\Delta^{(2)}(X)}[\omega_j|\eta_j]~\prod\limits_{j=i+2}^{n}\int\limits_{(P_e)_\ast\Delta^{(2)}(X)}[\omega_{j}|\eta_{j-1}]\\
&=&(-1)^{n+i-1+m}\int\limits_{\partial^{-1}\Delta_{2,e}}[\omega_i|\omega_{i+1}|\eta_i]~\prod\limits_{j=1}^{i-1}\int\limits_{\Delta^{(2)}(X)}[\omega_j|\eta_j]~\prod\limits_{j=i+2}^{n}\int\limits_{\Delta^{(2)}(X)}[\omega_{j}|\eta_{j-1}],
\end{eqnarray*}
as the other summands in $(P_e)_\ast\Delta^{(2)}(X)$ do not contribute to the integrals. In view of \eqref{eqDRSv3}, the last expression is 
\begin{eqnarray*}
&=&(-1)^{n+i-1+m} \left(-\int\limits_{\gamma_i} \omega_i\omega_{i+1}+\nu([\omega_i|\omega_{i+1}])~+
\int\limits_X\omega_i\wedge\omega_{i+1} \int\limits_e^\infty \eta_i\right) \prod\limits_{j=1}^{i-1}\int\limits_X\omega_j\wedge\eta_j~\prod\limits_{j=i+2}^{n}\int\limits_{X}\omega_{j}\wedge\eta_{j-1}\\
&=&(-1)^{i-1+m} \left(-\int\limits_{\gamma_i} \omega_i\omega_{i+1}+\nu([\omega_i|\omega_{i+1}])~+\int\limits_X\omega_i\wedge\omega_{i+1}\int\limits_e^\infty \eta_i\right) \prod\limits_{j=1}^{i-1}\int\limits_{\gamma_j}\omega_j~\prod\limits_{j=i+2}^{n}\int\limits_{\gamma_{j-1}}\omega_{j}.
\end{eqnarray*}
It follows that 
\begin{equation}\label{eq1proofmain1}
(-1)^m\Phi(h_n(\Delta_{n,e}))[\omega_1|\ldots|\omega_n|\eta_1|\ldots|\eta_{n-1}]=-(I)~+~(II),
\end{equation}
where 
\[
(I)~=~\sum\limits_{i=1}^{n-1}\left(\int\limits_{\gamma_i} \omega_i\omega_{i+1}+\nu([\omega_i|\omega_{i+1}])\right) \prod\limits_{j=1}^{i-1}\int\limits_{\gamma_j}\omega_j~\prod\limits_{j=i+2}^{n}\int\limits_{\gamma_{j-1}}\omega_{j}
\]
and
\[
(II)~=~\sum\limits_{i=1}^{n-1} \int\limits_X\omega_i\wedge\omega_{i+1}~\int\limits_e^\infty \eta_i ~\prod\limits_{j=1}^{i-1}\int\limits_{\gamma_j}\omega_j~\prod\limits_{j=i+2}^{n}\int\limits_{\gamma_{j-1}}\omega_{j}.
\]
In view of \eqref{itintprop}, 
\begin{eqnarray}
(I)&=&\sum\limits_{i=1}^{n-1} \prod\limits_{j=1}^{i-1}\int\limits_{\gamma_j}\omega_j~\int\limits_{\gamma_i} \omega_i\omega_{i+1}~\prod\limits_{j=i+2}^{n}\int\limits_{\gamma_{j-1}}\omega_{j}~+~\sum\limits_{i=1}^{n-1} \prod\limits_{j=1}^{i-1}\int\limits_{\gamma_j}\omega_j~\int\limits_{\gamma_i}\nu([\omega_i|\omega_{i+1}])~\prod\limits_{j=i+2}^{n}\int\limits_{\gamma_{j-1}}\omega_{j}\notag\\
&=&\int\limits_{(\gamma_1-1)\ldots(\gamma_{n-1}-1)}\omega_1\ldots\omega_n~+~\sum\limits_{i=1}^{n-1}\int\limits_{(\gamma_1-1)\ldots(\gamma_{n-1}-1)}\omega_1\ldots\omega_{i-1}\nu([\omega_i|\omega_{i+1}])\omega_{i+2}\ldots\omega_n\notag\\
&=&\Phi(\Psi(\mathbb{E}^\infty_{n,e}))([\omega_1|\ldots|\omega_n|\eta_1|\ldots|\eta_{n-1}]),\label{eq2proofmain1} 
\end{eqnarray}
by Proposition \ref{harmonicvolume}(b).\\

On the other hand, for $1\leq i\leq n-1$, in view of \eqref{eqpartialZ},
\begin{eqnarray*}
\Phi(h_n(Z^\infty_{n,i}-Z^e_{n,i}))([\omega_1|\ldots|\omega_n|\eta_1|\ldots|\eta_{n-1}])&=&\int_{(\tau_i)_\ast (C^\infty_{n,e})}[\omega_1|\ldots|\omega_n|\eta_1|\ldots|\eta_{n-1}]\\
&=&\int_{C^\infty_{n,e}}(\tau_i)^\ast[\omega_1|\ldots|\omega_n|\eta_1|\ldots|\eta_{n-1}],
\end{eqnarray*}
which, in view of the definition of $C^\infty_{n,e}$ and on recalling how $\tau_i$ permutes the coordinates of $X^{2n-1}$, is
\begin{eqnarray*}
&=& (-1)^{m+n+i-1} \int\limits_X\omega_i\wedge\omega_{i+1}~\int\limits_{\gamma^\infty_e}\eta_i~\prod\limits_{j=1}^{i-1}\int\limits_{X}\omega_j\wedge\eta_j~\prod_{j=i+2}^{n} \int\limits_{X}\omega_j\wedge\eta_{j-1}\\
&=& (-1)^{m+i-1} \int\limits_X\omega_i\wedge\omega_{i+1}~\int\limits_{\gamma^\infty_e}\eta_i~\prod\limits_{j=1}^{i-1}\int\limits_{\gamma_j}\omega_j~\prod_{j=i+2}^{n} \int\limits_{\gamma_{j-1}}\omega_j.
\end{eqnarray*}
Thus
\begin{eqnarray}
\Phi(h_n(Z^\infty_{n,e}))([\omega_1|\ldots|\omega_n|\eta_1|\ldots|\eta_{n-1}])&=&\sum\limits_{i=1}^{n-1}(-1)^{i-1}\Phi(h_n(Z^\infty_{n,i}-Z^e_{n,i}))([\omega_1|\ldots|\omega_n|\eta_1|\ldots|\eta_{n-1}])\notag\\
&=&\sum\limits_{i=1}^{n-1}(-1)^m\int\limits_X\omega_i\wedge\omega_{i+1}~\int\limits_{\gamma^\infty_e}\eta_i~\prod\limits_{j=1}^{i-1}\int\limits_{\gamma_j}\omega_j~\prod_{j=i+2}^{n} \int\limits_{\gamma_{j-1}}\omega_j\notag\\
&=&(-1)^m~(II).\label{eq3proofmain1}
\end{eqnarray}
\noindent Finally, combining equations \eqref{eq1proofmain1}, \eqref{eq2proofmain1}, and \eqref{eq3proofmain1}, we have 
\begin{eqnarray*}
(-1)^m\Phi(h_n(\Delta_{n,e}))[\omega_1|\ldots|\omega_n|\eta_1|\ldots|\eta_{n-1}]&=&-\Phi(\Psi(\mathbb{E}^\infty_{n,e}))([\omega_1|\ldots|\omega_n|\eta_1|\ldots|\eta_{n-1}])\\
&+&(-1)^m\Phi(h_n(Z^\infty_{n,e}))([\omega_1|\ldots|\omega_n|\eta_1|\ldots|\eta_{n-1}]),
\end{eqnarray*}
as desired.

\section{Two corollaries of Theorem \ref{main1}}\label{geometriccor}

In this section we give two corollaries of Theorem \ref{main1}. First we establish a lemma.

\begin{lemma}\label{lemjuly17}
The map 
\begin{equation}\label{eq1corsection}
\CH_{0}^\hhom(X)\rightarrow  J((H^1)^{\otimes 2n-1})^\vee\hspace{.2in}\infty-e~\mapsto h_n(Z^\infty_{n,e})
\end{equation}
is injective.
\end{lemma}

\begin{proof}
It is clear from the definition of $Z^\infty_{n,e}$ that \eqref{eq1corsection} is a (well-defined) group map. Now suppose 
\[
\sum\limits_j h_n(Z^{\infty_j}_{n,e})=0.
\]
We will show that $\sum\limits_j (\infty_j-e)$ is zero in $\CH_{0}^\hhom(X)$. Let $\eta$ be a harmonic 1-form on $X$ with integral periods. In view of the isomorphisms 
\[\CH_{0}^\hhom(X)\stackrel{\text{AJ}=h_1}{\cong} J(H^1)^\vee\cong Hom(H^1_\ZZ,\RR/\ZZ),\]
it suffices to show that
\[
\sum\limits_j \int\limits_e^{\infty_j} \eta~\in\ZZ.
\]
We may assume that $\int \eta:H_1(X,\ZZ)\rightarrow\ZZ$ is surjective. Let $\omega$ be a harmonic 1-form with integral periods such that $\int\limits_X \omega\wedge\eta=1$. We shall use the notation as in Paragraph \ref{pardefZ} and write 
\[
Z^{\infty_j}_{n,e}=\sum\limits_i(-1)^{i-1}(Z^{\infty_j}_{n,i}-Z^{e}_{n,i}).
\]
On recalling the definition of the cycles involved in the equation above, one easily sees that in $\RR/\ZZ$,
\begin{eqnarray*}
\Phi(h_n(Z^{\infty_j}_{n,e}))(\omega\otimes\eta^{\otimes n}\otimes\omega^{\otimes n-2})&=&\Phi(h_n(Z^{\infty_j}_{n,1}-Z^{e}_{n,1}))(\omega\otimes\eta^{\otimes n}\otimes\omega^{\otimes n-2})\\
&=&(-1)^{\frac{(n-3)(n-2)}{2}}\int\limits_e^{\infty_j}\eta.
\end{eqnarray*}
The result follows from that $\sum\limits_j \Phi(h_n(Z^{\infty_j}_{n,e}))=0$.
\end{proof}

We now give two consequences of Theorem \ref{main1}. The first is in the spirit of Corollary 5.4 of Pulte \cite{Pulte}.

\begin{cor}
The function
\[
X(\CC)-\{e\}\rightarrow \Ext((H^1)^{\otimes n}, (H^1)^{\otimes n-1})\hspace{.2in} \infty\mapsto \mathbb{E}^\infty_{n,e}
\]
is injective.
\end{cor}
\begin{proof}
Let $\infty_1,\infty_2\in X(\CC)-\{e\}$. By Theorem \ref{main1},
\[
(-1)^{\frac{n(n-1)}{2}}\Psi(\mathbb{E}^{\infty_1}_{n,e}-\mathbb{E}^{\infty_2}_{n,e})=h_n(Z^{\infty_2}_{n,e}-Z^{\infty_1}_{n,e})=h_n(Z^{\infty_2}_{n,\infty_1}).
\]
The result follows from the previous lemma.
\end{proof}

\begin{cor}
Suppose $X$ has genus 1. Then $\mathbb{E}^\infty_{n,e}$ is torsion if and only if $\infty-e$ is torsion in $\CH_0^\hhom(X)$.
\end{cor}

\begin{proof}
By a result of Gross and Shoen \cite[Corollary 4.7]{GS}, $\Delta_{2,e}$ is torsion in genus 1 case. It follows that $\Delta_{n,e}$ is torsion for all $n$. (See the remark at the end of Paragraph \ref{boundry_inverse_Delta_2}.) Thus by Theorem \ref{main1}, $\mathbb{E}^\infty_{n,e}$ is torsion if and only if $h_n(Z^\infty_{n,e})$ is torsion. The desired conclusion follows from Lemma \ref{lemjuly17}.
\end{proof}

\section{$\mathbb{E}^\infty_{n,e}$ and Rational Points on the Jacobian}\label{ch4}

In the remainder of the paper we assume that $X, e,\infty$ are defined over a subfield $K\subset\CC$. Our goal is to give some applications of Theorem \ref{main1} in number theory. In this section, we show that one can associate to the extension $\mathbb{E}^\infty_{n,e}$ a family of rational points on the Jacobian of $X$. This generalizes Theorem 1 and Corollary 1 of \cite{DRS}. Our approach follows the ideas leading to those results, and generally speaking, is in line with Darmon's philosophy of trying to construct rational points on Jacobian varieties using higher dimensional varieties.

\subsection{Recollection: Maps between intermediate Jacobians induces by correspondences} Let $Y$ (resp. $Y'$) be a smooth projective variety over $\CC$ of dimension $d$ (resp. $d'$) over $\CC$. Suppose $l\leq d+d'$. One has natural isomorphisms
\begin{eqnarray}
H^{2l}(Y\times Y')^\vee&\cong& \left(\bigoplus_{r}H^r(Y)\otimes H^{2l-r}(Y')\right)^\vee\notag\\
&\cong&\bigoplus_{r}H^r(Y)^\vee \otimes H^{2l-r}(Y')^\vee \notag\\
&\cong&\bigoplus_{r}\inhom\left(H^r(Y),H^{2l-r}(Y')^\vee\right)\notag\\
&\stackrel{\text{Poincare duality}}{\cong}&\bigoplus_{r} \inhom\left(H^{2d-r}(Y)^\vee(-d),H^{2l-r}(Y')^\vee\right)\notag\\
&\cong&\bigoplus_{r} \inhom\left(H^{2d-r}(Y)^\vee,H^{2l-r}(Y')^\vee\right)(d).\notag
\end{eqnarray}

\noindent Let $Z\in\CH_l(Y\times Y')$. Then the class $cl(Z)$ of $Z$ is a Hodge class in 
\[
H^{2l}(Y\times Y')^{\vee},
\] 
which is given by integration over $Z$ (or more precisely, the smooth locus of $Z$) if $Z$ is an irreducible closed subset. In view of the isomorphisms above, $cl(Z)$ decomposes as a sum of Hodge classes in 
\[
\inhom\left(H^{2d-r}(Y)^\vee,H^{2l-r}(Y')^\vee\right).
\]
It follows that for each $r$, $cl(Z)$ gives a morphism of Hodge structures 
\begin{equation}\label{eq1arith}
H^{2d-r}(Y)^\vee(l-d)\rightarrow H^{2l-r}(Y')^\vee.
\end{equation}
If $r$ is odd, this induces a map 
\begin{equation}\label{eq2arith}
J H^{2d-r}(Y)^\vee=JH^{2d-r}(Y)^\vee(l-d) \rightarrow JH^{2l-r}(Y')^\vee.
\end{equation}
With abuse of notation we denote the maps \eqref{eq1arith} and \eqref{eq2arith} also by $cl(Z)$.\\

Let $m\leq d$. The push-forward map
\[
Z_\ast: \CH_m(Y)\rightarrow \CH_{m+l-d}(Y')
\]
restricts to a map
\[
Z_\ast: \CH^\hhom_m(Y)\rightarrow \CH^\hhom_{m+l-d}(Y').
\]
One has a commutative diagram
\[
\begin{tikzpicture}
  \matrix (m) [matrix of math nodes, column sep=2em, row sep=2.5em]
    {	  \CH^\hhom_m(Y)& J H^{2m+1}(Y)^\vee\\
		\CH^\hhom_{m+l-d}(Y')& J H^{2m+2l-2d+1}(Y')^\vee.\\};
  { [start chain] \chainin (m-1-1);
    \chainin (m-1-2) [join={node[above,labeled] {AJ}}];}
  {[start chain] \chainin (m-2-1);
    \chainin (m-2-2) [join={node[above,labeled] {AJ}}];}
	{[start chain] \chainin (m-1-2);
		\chainin (m-2-2) [join={node[right,labeled] {cl(Z)}}];}
	{[start chain] \chainin (m-1-1);
		\chainin (m-2-1) [join={node[left,labeled] {Z_\ast}}];}
  \end{tikzpicture}
\]

\subsection{} Fix a subfield $K\subset\CC$. From now on, we assume that the curve $X$ and the points $e,\infty$ are defined over $K$. More precisely, suppose $X=X_0\times_K\Spec(\CC)$, where $X_0$ is a projective curve over $K$, and that $e,\infty\in X_0(K)$. Let $\Jac=\Jac(X_0)$ be the Jacobian of $X_0$. Throughout, we identify 
\[
\Jac(\CC)=\CH_0^\hhom(X)\stackrel{\AJ}{\cong}J(H^1)^\vee.
\]
Thus in particular, $\Jac(K)$ is identified as a subgroup of $J(H^1)^\vee$. For a Hodge class $\xi$ in $(H^1)^{\otimes 2n-2}$, let
\[
\xi^{-1}: J((H^1)^{\otimes 2n-1})^\vee\rightarrow J(H^1)^\vee
\]
be the map $[f]\mapsto [f(\xi\otimes-)]$. For an algebraic cycle $Z\in\CH_{n-1}(X_0^{2n-2})$, we denote by $\xi_Z$ the $(H^1)^{\otimes 2n-2}$ Kunneth component of 
\[
cl(Z)\in H_\CC^{2n-2}(X^{2n-2})^\vee\stackrel{\text{Poincare duality}}{\cong} H_\CC^{2n-2}(X^{2n-2}).
\]
We have the following result.

\begin{thm}\label{rationalpointjac}
Let $Z\in\CH_{n-1}(X_0^{2n-2})$. Then 
\[
\xi_Z^{-1}(\Psi(\mathbb{E}^\infty_{n,e}))\in\Jac(K).
\] 
\end{thm} 
We should point out that this is not a priori obvious, as to get the extension $\mathbb{E}^\infty_{n,e}$ one first extends the scalars to $\CC$. Note that varying $Z$, we get a family of points in $\Jac(K)$ associated to $\mathbb{E}^\infty_{n,e}$ parametrized by $\CH_{n-1}(X_0^{2n-2})$. In other words, the weight filtration on (the mixed Hodge structure associated to) $\pi_1(X-\{\infty\},e)$ is giving rise to families of points in $\Jac(K)$ parametrized by algebraic cycles on powers of $X_0$.\\

With abuse of notation, we denote the compositions
\[
\CH_{n-1}^\hhom(X_0^{2n-1})\stackrel{\text{natural map}}{\rightarrow} \CH_{n-1}^\hhom(X^{2n-1})\stackrel{\AJ}{\rightarrow} JH^{2n-1}(X^{2n-1})^\vee
\]
and
\[
\CH_{n-1}^\hhom(X_0^{2n-1})\stackrel{\text{natural map}}{\rightarrow} \CH_{n-1}^\hhom(X^{2n-1})\stackrel{h_n}{\rightarrow} J((H^1)^{\otimes 2n-1})^\vee 
\]  
by $\AJ$ and $h_n$ respectively. In view of Theorem \ref{main1} and the fact that both $\Delta_{n,e}$ and $Z^\infty_{n,e}$ are defined over $K$ (see Paragraph \ref{remcyclesoverk}), Theorem \ref{rationalpointjac} follows immediately from the following lemma.

\begin{lemma}\label{sec8thm}
Let $Z\in\CH_{n-1}(X_0^{2n-2})$. Then the image of the composition
\[
\CH_{n-1}^\hhom(X_0^{2n-1})\stackrel{h_n}{\rightarrow} J((H^1)^{\otimes 2n-1})^\vee\stackrel{\xi_Z^{-1}}{\rightarrow}J(H^1)^\vee
\]
lies in the subgroup $\Jac(K)$.
\end{lemma}

\begin{proof}
Denote the diagonal of $X_0$ by $\Delta(X_0)$. Let $Z'\in\CH_n(X_0^{2n})$ be such that its class in 
\[
H^{2n}(X^{2n})^\vee
\]
is the $((H^1)^{\otimes 2n})^\vee$ Kunneth component of 
\[
cl(Z\times \Delta(X_0))\in H^{2n}(X^{2n})^\vee.
\]
Such $Z'$ can be explicitly constructed using the fact that the Kunneth components of the class of the diagonal $\Delta(X_0)\in\CH_{1}(X_0^2)$ are algebraic. We will show that the diagram
\begin{equation}\label{proofsec8thm}
\begin{tikzpicture}
  \matrix (m) [matrix of math nodes, column sep=2em, row sep=2.5em]
    {	  \CH_{n-1}^\hhom(X_0^{2n-1})& J((H^1)^{\otimes 2n-1})^\vee\\
		\CH_{0}^\hhom(X_0)& J(H^1)^\vee\\};
  { [start chain] \chainin (m-1-1);
    \chainin (m-1-2) [join={node[above,labeled] {h_n}}];}
  {[start chain] \chainin (m-2-1);
    \chainin (m-2-2) [join={node[above,labeled] {h_1=\AJ}}];}
	{[start chain] \chainin (m-1-2);
		\chainin (m-2-2) [join={node[right,labeled] {\xi_Z^{-1}}}];}
	{[start chain] \chainin (m-1-1);
		\chainin (m-2-1) [join={node[left,labeled] {Z'_\ast}}];}
  \end{tikzpicture}
\end{equation}
commutes. This will prove the assertion, as $h_1$ is the map that identifies $\Jac(K)=\CH_{0}^\hhom(X_0)$ as a subgroup of $J(H^1)^\vee$.\\

By functoriality of the Abel-Jacobi maps with respect to correspondences, one has a commutative diagram
\begin{equation}\label{proofsec8thmdiag2}
\begin{tikzpicture}
  \matrix (m) [matrix of math nodes, column sep=2em, row sep=2.5em]
    {	  \CH_{n-1}^\hhom(X_0^{2n-1})& J H^{2n-1}(X^{2n-1})^\vee\\
		\CH_{0}^\hhom(X_0)& J(H^1)^\vee.\\};
  {[start chain] \chainin (m-1-1);
    \chainin (m-1-2) [join={node[above,labeled] {\AJ}}];}
  {[start chain] \chainin (m-2-1);
    \chainin (m-2-2) [join={node[above,labeled] {\AJ}}];}
	{[start chain] \chainin (m-1-2);
		\chainin (m-2-2) [join={node[right,labeled] {cl(Z')}}];}
	{[start chain] \chainin (m-1-1);
		\chainin (m-2-1) [join={node[left,labeled] {Z'_\ast}}];}
	\end{tikzpicture}
\end{equation}
Thus to establish commutativity of \eqref{proofsec8thm}, it suffices to show that
\begin{equation}
\begin{tikzpicture}
  \matrix (m) [matrix of math nodes, column sep=3em, row sep=2.5em]
    {	  J H^{2n-1}(X^{2n-1})^\vee&J((H^1)^{\otimes 2n-1})^\vee\\
		J(H^1)^\vee&\\};
  {[start chain] \chainin (m-1-1);
    \chainin (m-1-2) [join={node[above,labeled] {\text{natural}} node[below,labeled] {\text{projection}}}];}
	{[start chain] \chainin (m-1-2);
		\chainin (m-2-1) [join={node[below,labeled] {\xi_Z^{-1}}}];}
	{[start chain] \chainin (m-1-1);
		\chainin (m-2-1) [join={node[left,labeled] {cl(Z')}}];}
	\end{tikzpicture}
\end{equation}
commutes. This in turn will be established if we verify the commutativity of
\begin{equation}\label{proofsec8thmdiag3}
\begin{tikzpicture}
  \matrix (m) [matrix of math nodes, column sep=3em, row sep=2.5em]
    {	  H_\CC^{2n-1}(X^{2n-1})^\vee&((H_\CC^1)^{\otimes 2n-1})^\vee\\
		(H_\CC^1)^\vee,&\\};
  {[start chain] \chainin (m-1-1);
    \chainin (m-1-2) [join={node[above,labeled] {\text{natural}} node[below,labeled] {\text{projection}}}];}
	{[start chain] \chainin (m-1-2);
		\chainin (m-2-1) [join={node[below,labeled] {\xi_Z^{-1}}}];}
	{[start chain] \chainin (m-1-1);
		\chainin (m-2-1) [join={node[left,labeled] {cl(Z')}}];}
	\end{tikzpicture}
\end{equation}
where with abuse of notation $\xi_Z^{-1}$ denotes the map $f\mapsto f(\xi_Z\otimes-)$. Note that since 
\[cl(Z')\in ((H_\CC^1)^{\otimes 2n})^\vee\subset H_\CC^{2n}(X^{2n})^\vee,\]
we only need to verify commutativity on the direct summand
\[
((H_\CC^1)^{\otimes 2n-1})^\vee\subset H_\CC^{2n-1}(X^{2n-1})^\vee.
\] 
Let $f\in ((H^1_\CC)^{\otimes 2n-1})^\vee$. Suppose $f$ is the Poincare dual of $\alpha\in H_\CC^{2n-1}(X^{2n-1})$, i.e. 
\[f(-)=\int\limits_{X^{2n-1}}\alpha\wedge -.\]
Then $\alpha$ lies in the Kunneth component $(H_\CC^1)^{\otimes 2n-1}$. Let $\beta\in H^1_\CC$. Unwinding definitions, in view of the fact that $cl(Z')$ is the $((H^1)^{\otimes 2n})^\vee$ component of $cl(Z\times\Delta(X_0))$, we have
\[
cl(Z')(f)(\beta)=cl(Z')(\alpha\otimes\beta)=cl(Z\times \Delta(X_0))(\alpha\otimes \beta). 
\] 
Let 
\[
\alpha=\sum\limits_i \alpha^{(i)}_1\otimes\ldots\otimes\alpha^{(i)}_{2n-1}.
\]
Then 
\begin{eqnarray*}
cl(Z')(f)(\beta)&=&\sum\limits_i cl(Z\times\Delta(X_0)) (\alpha^{(i)}_1\otimes\ldots\otimes\alpha^{(i)}_{2n-1}\otimes\beta)\\
&=& \sum\limits_i cl(Z)(\alpha^{(i)}_1\otimes\ldots\otimes\alpha^{(i)}_{2n-2})\int\limits_X\alpha^{(i)}_{2n-1}\wedge\beta\\
&=&\sum\limits_i \int\limits_{X^{2n-2}}\xi_Z\wedge (\alpha^{(i)}_1\otimes\ldots\otimes\alpha^{(i)}_{2n-2})~\int\limits_X\alpha^{(i)}_{2n-1}\wedge\beta\\
&=&\sum\limits_i \int\limits_{X^{2n-1}}\left(\xi_Z\wedge (\alpha^{(i)}_1\otimes\ldots\otimes\alpha^{(i)}_{2n-2})\right)\otimes (\alpha^{(i)}_{2n-1}\wedge\beta)\\
&=&\int\limits_{X^{2n-1}}\alpha\wedge(\xi_Z\otimes \beta)\\
&=&f(\xi_Z\otimes \beta).
\end{eqnarray*}
Thus $cl(Z')(f)=\xi_Z^{-1}(f)$ as desired.
\end{proof}
From now on, in the interest of simplifying the notation, for a Hodge class $\xi\in (H^1)^{\otimes 2n-2}$, we write $P_\xi$ for $\xi^{-1}(\Psi(\mathbb{E}^\infty_{n,e}))$. For $Z\in\CH_{n-1}(X_0^{2n-2})$, we simply write $P_Z$ for $P_{\xi_Z}$.\\

{\bf Remark.} It was pointed out to me by Darmon that the idea of constructing points on the Jacobian of $X_0$ using Hodge classes in $H^2(X^2)$ first arose in the work \cite{YZZ} of W. Yuan, S. Zhang, and W. Zhang in the setting of modular curves.\\

\subsection{An analytic description of $P_Z$}\label{analytic_description_par1}

Proposition \ref{harmonicvolume}(a) gives us a description of $\Psi(\mathbb{E}^\infty_{n,e})$, and hence can be used to give an analytic description of points of the form $P_Z$, or more generally $P_\xi$. The issue with this description will be that it involves the forms $\nu$. More precisely, to do computations with it one needs to know $\nu(\omega_1\otimes\omega_2)$ for harmonic forms $\omega_1,\omega_2$ on $X$. In this paragraph, we try to give a different description of $\Psi(\mathbb{E}^\infty_{n,e})$, and hence $P_Z$ and $P_\xi$, which does not have this issue, as it uses differentials of the second kind as opposed to harmonic forms.\\

Recall that in view of Carlson's theorem (see Paragraph \ref{Carlsoniso}), a Hodge section of $\mathfrak{q}$ and an integral retraction of $\mathfrak{i}$ (see \eqref{seq2}) will give us a description of  
\[\mathbb{E}^\infty_{n,e}\in\Ext((H^1)^{\otimes n},(H^1)^{\otimes n-1})\cong J\inhom((H^1)^{\otimes n},(H^1)^{\otimes n-1}).\] 
We will use the same retraction $r_\ZZ$ of $\mathfrak{i}$ as in Section \ref{section_extension_E}, but seek for a different, rather more simple, Hodge section of $\mathfrak{q}$.\\

Recall that $g$ is the genus of $X$. From now on (to the end of the paper), we fix the following set of data:

\begin{itemize}
\item[(i)] $\alpha_1,\ldots,\alpha_{2g}$ as in Paragraph \ref{Hodge_filt_L_n_paragraph}: For $1\leq i\leq g$, $\alpha_i$ is holomorphic on $X$, and for $g+1\leq i\leq 2g$, $\alpha_i$ is meromorphic on $X$ and holomorphic on $X-\{\infty\}$, and the cohomology classes of the $\alpha_i$ form a basis of $H^1_\CC$.
\item[(ii)] a basis $d_1,\ldots,d_{2g}$ of $H^1_\ZZ$
\item[(iii)] $\beta_1,\ldots,\beta_{2g}\in\pi_1(X-\{\infty\},e)$ such that $[\beta_i]=\PD(d_i)$, i.e.
\[
\int\limits_{\beta_i}-=\int\limits_X d_i\wedge-.
\] 
\end{itemize}
As in Paragraph \ref{Hodge_filt_L_n_paragraph}, let 
\[R^1=\sum\limits_i\CC\alpha_i\subset \Omega^1_{\hol}(X-\{\infty\}),\] 
where for any Riemann surface $M$, by $\Omega^1_\hol(M)$ we denote the space of holomorphic 1-forms on $M$.\\

The map 
\[
(H^1)^{\otimes n}_\CC\rightarrow (L_n)_\CC
\]
defined by
\[
[\alpha_{i_1}]\otimes\ldots\otimes[\alpha_{i_{n}}]\mapsto \int\alpha_{i_1}\ldots\alpha_{i_n}, 
\]
or equivalently by
\begin{equation}\label{eq1ch4revision}
[\omega_{i_1}]\otimes\ldots\otimes [\omega_{i_{n}}]\mapsto\int \omega_{i_1}\ldots\omega_{i_n} \hspace{.2in}(\omega_i\in R^1),
\end{equation}
is a section of $q_\CC$; this is clear from \eqref{eqq}. Thus the composition
\[
\sigma_F:(H^1)^{\otimes n}_\CC\stackrel{\eqref{eq1ch4revision}}{\rightarrow} (L_n)_\CC\stackrel{\text{quotient}}{\rightarrow}(\frac{L_n}{L_{n-2}})_\CC
\]
is a section of $\mathfrak{q}$ (over $\CC$).\\

\noindent {\bf\underline{Hypothesis $\star$}}: We say that the $\alpha_i$ satisfy {\it Hypothesis $\star$} if the map $\sigma_F$ above is compatible with the Hodge filtrations.\\

Recall from Paragraph \ref{intretpar} that our choice of the $\beta_i$ leads to an integral retraction $r_\ZZ$ of $\mathfrak{i}$ given by \eqref{eqr_Z}. In view of Carlson's theorem, if the $\alpha_i$ satisfy Hypothesis $\star$, the extension $\mathbb{E}^\infty_{n,e}\in J\inhom((H^1)^{\otimes n},(H^1)^{\otimes n-1})$ is represented by the map $r_\ZZ\circ\sigma_F$. Thus we have the following description of $\Psi(\mathbb{E}^\infty_{n,e})$. (See the argument for Proposition \ref{harmonicvolume}(a).)
\begin{prop}
If the $\alpha_i$ satisfy Hypothesis $\star$, then $\Psi(\mathbb{E}^\infty_{n,e})$ is represented by the map
\[
(H^1_\CC)^{\otimes 2n-1}\rightarrow\CC\]
given by
\[
[\omega_1]\otimes\ldots\otimes [\omega_n]\otimes d_{i_1}\otimes\ldots\otimes d_{i_{n-1}} \mapsto \int\limits_{(\beta_{i_1}-1)\ldots(\beta_{i_{n-1}}-1)} \omega_1\ldots\omega_n\hspace{.3in}(\omega_i\in R^1).
\]
\end{prop}

Let 
\[s:H^1_\CC\rightarrow\Omega^1_{\hol}(X-\{\infty\})\]
be the map that sends $c\in H^1_\CC$ to the unique element of $R^1$ representing $c$. From now on, $\omega_i:=s(d_i)$; it is in particular a linear combination of the $\alpha_i$ with integral periods.\\
For $c\in H^1_\CC$, we write
\[
c=\sum\limits_i p_{i}(c) d_i,
\]
which is equivalent to
\[
s(c)=\sum\limits_i p_{i}(c) \omega_i.
\]
For a {\it multi-index} 
\[
I=(i_m,\ldots,i_m)\subset\{1,\ldots,2g\}^{m},
\]
let
\[
d_I=d_{i_1}\otimes\ldots\otimes d_{i_m}\in (H^1)_\ZZ^{\otimes m}.
\]
For a Hodge class $\xi\in (H^1)^{\otimes 2n-2}$, we write
\[
\xi=\sum\limits_{I\subset\{1,\ldots,2g\}^{2n-2}} \lambda_I(\xi) d_I.
\]
Note that the $\lambda_I(\xi)$ are integers. For $Z\in \CH_{n-1}(X_0^{2n-2})$, let $\lambda_I(Z)=\lambda_I(\xi_Z)$.\\

Denote the map of the previous proposition tentatively by $f$. If the $\alpha_i$ satisfy Hypothesis $\star$, by definition, $P_\xi$ is the class of the map
\[
f_\xi:H^1_\CC\rightarrow\CC\hspace{.3in}\text{defined by}\hspace{.3in} c\mapsto f(\xi\otimes c).
\]
We have
\begin{eqnarray*}
f(\xi\otimes c)&=&\sum\limits_j p_j(c)f(\xi\otimes d_j)\\
&=&\sum\limits_j\sum\limits_{I} p_j(c)\lambda_I(\xi) f(d_I\otimes d_j)\\
&=&\sum\limits_j\sum\limits_{I} p_j(c)\lambda_I(\xi)\int\limits_{(\beta_{i_{n+1}}-1)\ldots(\beta_{i_{2n-2}}-1)(\beta_j-1)} \omega_{i_1}\ldots\omega_{i_n},
\end{eqnarray*}  
where in all summations $1\leq j\leq 2g$ and $I=(i_1,\ldots,i_{2n-2})\in \{1,\ldots,2g\}^{2n-2}$. We record the conclusion as a proposition.	
\begin{prop}\label{analytic_description_z(e)}
Suppose the $\alpha_i$ satisfy Hypothesis $\star$. Then $P_\xi$ is the class of the map $f_\xi:H^1_\CC\rightarrow\CC$ defined by
\[
f_\xi(c)=\sum\limits_j\sum\limits_{I} p_j(c)\lambda_I(\xi)\int\limits_{(\beta_{i_{n+1}}-1)\ldots(\beta_{i_{2n-2}}-1)(\beta_j-1)} \omega_{i_1}\ldots\omega_{i_n}.
\]
In particular, for $Z\in\CH_{n-1}(X_0^{2n-2})$, $P_Z$ is the class of $f_{\xi_Z}$.
\end{prop}
We finish this paragraph by rewriting the formula for $f_\xi$ in a form that will be useful later. For future reference, we record it as a proposition.
\begin{prop}\label{analytic_description_z(e)_V2}
For $i,j,k\leq 2g$, let
\[
\mu'_{ijk}(\xi;c)=\sum\limits_{r=1}^{n-1}~\sum\limits_{\stackrel{i_1,\ldots,i_{2n-1}\leq2g}{(i_r,i_{r+1},i_{r+n})=(i,j,k)}}~\lambda_{(i_1,\ldots,i_{2n-2})}(\xi)~p_{i_{2n-1}}(c)\prod\limits_{l=1}^{r-1}\int\limits_{\beta_{i_{l+n}}}\omega_{i_l}~\prod\limits_{l=r+2}^{n}\int\limits_{\beta_{i_{l+n-1}}}\omega_{i_l}.
\]
Then
\[
f_\xi(c)=\sum\limits_{i,j,k\leq2g}\mu'_{ijk}(\xi;c)\int\limits_{\beta_k}\omega_i\omega_j.
\]
\end{prop}
\begin{proof}
This follows from the previous formula for $f_\xi$ on noting that by \eqref{itintprop},
\[
\int\limits_{(\beta_{i_{n+1}}-1)\ldots(\beta_{i_{2n-2}}-1)(\beta_{i_{2n-1}}-1)} \omega_{i_1}\ldots\omega_{i_n}=\sum\limits_{r=1}^{n-1}~~ \prod\limits_{l=1}^{r-1}\int\limits_{\beta_{i_{l+n}}}\omega_{i_l} ~\int\limits_{\beta_{i_{r+n}}}\omega_{i_r}\omega_{i_{r+1}}~ \prod\limits_{l=r+2}^{n}\int\limits_{\beta_{i_{l+n-1}}}\omega_{i_l}.
\]
\end{proof}

For $Z\in\CH_{n-1}(X_0^{2n-2})$, to simplify the notation we simply write $f_Z$ for $f_{\xi_Z}$. 

\subsection{More on Hypothesis $\star$} 

In this paragraph, we show that in the case of elliptic curves, one can indeed choose the $\alpha_i$ such that they satisfy Hypothesis $\star$. Note that when $g=1$, by assumption, $\alpha_2$ has a pole at $\infty$ and is holomorphic elsewhere. The order of the pole of $\alpha_2$ at $\infty$ is thus $\geq 2$. The form $\alpha_1$ is holomorphic on $X$.

\begin{prop}\label{prop_Hypothesis_star_genus1}
Let $g=1$. Suppose the order of $\infty$ as a pole of $\alpha_2$ is 2. Then the $\alpha_i$ satisfy Hypothesis $\star$. 
\end{prop}

Before we prove the proposition, we state an easy lemma.

\begin{lemma}
Let $D$ be the open unit disc in $\CC$. Suppose $\alpha$ is a holomorphic 1-form on $D-\{0\}$ with a pole of order 2 at $0$, and $\eta$ is a smooth closed 1-form on $D$. Let $f$ be a smooth function on $D-\{0\}$ such that $df=\alpha-\eta$ on $D-\{0\}$. Then $z^2f(z)\rightarrow 0$ as $z\rightarrow 0$.
\end{lemma}
\begin{proof}
Write $\alpha=(\frac{C}{z^2}+h)dz$, where $C\neq0$ is a constant and $h$ is a holomorphic function on $D$. Let $F$ be a smooth function on $D$ such that $dF=\eta$. Then
\[
df=(\frac{C}{z^2}+h)dz+dF=d\left(-\frac{C}{z}+H+F\right),
\] 
where $H$ is an anti-derivative of $h$. Thus 
\[
f=-\frac{C}{z}+H+F+\text{constant}.
\]
The desired conclusion follows.
\end{proof}

\underline{Proof of Proposition \ref{prop_Hypothesis_star_genus1}}: For convenience, we adopt the following temporary notation. For $i=1,2$, $\eta_i$ denotes the harmonic 1-form on $X$ whose cohomology class coincides with that of $\alpha_i$. In particular, $\eta_1=\alpha_1$. For each $i$, we write $\alpha_i=\eta_i+df_i$, where $f_i$ is a smooth function on $X-\{\infty\}$ satisfying $f_i(e)=0$. (Thus $f_1$ is just $0$.) Let $a_i=[\alpha_i]$. Note that $a_1\in F^1H^1_\CC$. We will be using the multi-index notation 
\[
a_I=a_{i_1}\otimes\ldots\otimes a_{i_n}
\]
for $I=(i_1,\ldots,i_n)\in \{1,2\}^{n}$.\\

To verify Hypothesis $\star$, we need to show that $\sigma_F$ is compatible with the Hodge filtrations. Since $\mathfrak{s}_F$ is known to be compatible with the Hodge filtrations (see Lemma \ref{s_F}), we can equivalently show that $\sigma_F-\mathfrak{s}_F$ respects the Hodge filtrations. In view of the fact that both $\sigma_F$ and $\mathfrak{s}_F$ are sections of $\mathfrak{q}$, we see that $\sigma_F-\mathfrak{s}_F$ actually maps into the subspace
\[
\left(\frac{L_{n-1}}{L_{n-2}}\right)_\CC=\ker(\mathfrak{q})_\CC\subset\left(\frac{L_n}{L_{n-2}}\right)_\CC
\]
(see Paragraph \ref{defe}). Thus we need to show that
\[(\sigma_F-\mathfrak{s}_F)(F^p(H^1_\CC)^{\otimes n})\subset F^p\left(\frac{L_{n-1}}{L_{n-2}}\right)_\CC.\] 
Equivalently, in view of Proposition \ref{qbar}, we will be done if we show that
\[
\overline{q}_{n-1}\circ(\sigma_F-\mathfrak{s}_F)(F^p(H^1_\CC)^{\otimes n})\subset F^p(H^1_\CC)^{\otimes n-1}.
\]
(Here $\overline{q}_{n-1}$ is the isomorphism $\displaystyle{\frac{L_{n-1}}{L_{n-2}}\rightarrow (H^1_\CC)^{\otimes n-1}}$ given by Proposition \ref{qbar}.)\\
Let
\[
I=(i_1,\ldots,i_n)\in \{1,2\}^{n}
\]
be such that at least $p$ of the $i_r$ are 1. It suffices to show that  
\[
\overline{q}_{n-1}\circ(\sigma_F-\mathfrak{s}_F) (a_I)\in F^p(H^1_\CC)^{\otimes n-1}.
\]
By Lemma \ref{s_F},
\begin{eqnarray*}
\mathfrak{s}_F(a_I)&=&\int~\eta_{i_1}\ldots\eta_{i_n}+\sum\limits_{r=1}^{n-1}\eta_{i_1}\ldots\nu(\eta_{i_r}\otimes\eta_{i_{r+1}})\ldots\eta_{i_n}\\
&+&\text{terms of length at most $n-2$}~\mod L_{n-2}.
\end{eqnarray*}
On the other hand,
\begin{eqnarray*}
\sigma_F(a_I)&=&\int \alpha_{i_1}\ldots\alpha_{i_n} \mod L_{n-2}\\
&=&\int (\eta_{i_1}+df_{i_1})\ldots (\eta_{i_n}+df_{i_n}) \mod L_{n-2}.
\end{eqnarray*}
The integral above expands as the integral of
\begin{eqnarray*}
&&\eta_{i_1}\ldots\eta_{i_n}+\sum\limits_{r} \eta_{i_1}\ldots(df_{i_r})\ldots\eta_{i_n}+\sum\limits_{r<s} \eta_{i_1}\ldots(df_{i_r})\ldots(df_{i_s})\ldots\eta_{i_n}\\
&+&\text{terms with three or more appearances of $df$.}
\end{eqnarray*}
In view of the relations \eqref{itrel} satisfied by iterated integrals, every summand in which two factors $df_{i_r}$ and $df_{i_s}$ with $s>r+1$ appear, can be replaced by terms of length at most $n-2$. In particular, this can be done for terms with three or more appearances of $df$. We get 
\begin{eqnarray*}
\sigma_F(a_I)&=&\int \eta_{i_1}\ldots\eta_{i_n}+\underbrace{\sum\limits_{r} \eta_{i_1}\ldots(df_{i_r})\ldots\eta_{i_n}}_{(I)}\\
&+&\underbrace{\sum\limits_{r<n} \eta_{i_1}\ldots(df_{i_r})(df_{i_{r+1}})\ldots\eta_{i_n}}_{(II)}\\
&+&\text{terms of length at most $n-2$}\hspace{.2in}\mod L_{n-2}. 
\end{eqnarray*}
On recalling $f_j(e)=0$, straightforward computations using \eqref{itrel} show 
\[
\int (I)=\int \sum\limits_{r<n}\eta_{i_1}\ldots\eta_{i_{r-1}}(f_{i_r}\eta_{i_{r+1}}-f_{i_{r+1}}\eta_{i_{r}})\eta_{i_{r+2}}\ldots\eta_{i_n}
\]
and 
\begin{eqnarray*}
\int (II)&=&\int \sum\limits_{r<n} \eta_{i_1}\ldots\eta_{i_{r-1}}(f_{i_r}df_{i_{r+1}})\eta_{i_{r+2}}\ldots\eta_{i_n}\\
&+&\text{terms of length at most $n-2$}.
\end{eqnarray*}
Thus
\begin{eqnarray*}
(\sigma_F-\mathfrak{s}_F)(a_I)&=&\int~ \sum\limits_{r<n}\eta_{i_1}\ldots\eta_{i_{r-1}}\bigm(f_{i_r}\eta_{i_{r+1}}-f_{i_{r+1}}\eta_{i_{r}}-\nu(\eta_{i_r}\otimes\eta_{i_{r+1}})\bigm)\eta_{i_{r+2}}\ldots\eta_{i_n}\\
&+&\sum\limits_{r<n} \eta_{i_1}\ldots\eta_{i_{r-1}}(f_{i_r}df_{i_{r+1}})\eta_{i_{r+2}}\ldots\eta_{i_n}\\
&+&\text{terms of length at most $n-2$}\hspace{.2in}\mod L_{n-2}.
\end{eqnarray*}
Note that each term on the right that appears on the first two lines, has length $n-1$. The integral on the right (which is closed) lives in $L_{n-1}$. We claim that both $f_{i_r}\eta_{i_{r+1}}-f_{i_{r+1}}\eta_{i_{r}}-\nu(\eta_{i_r}\otimes\eta_{i_{r+1}})$ and $f_{i_r}df_{i_{r+1}}$ are closed. This is clear for the latter element. As for the former, if $i_r=i_{r+1}$, then
\[
d\bigm(f_{i_r}\eta_{i_{r+1}}-f_{i_{r+1}}\eta_{i_{r}}-\nu(\eta_{i_r}\otimes\eta_{i_{r+1}})\bigm)=-d\nu(\eta_{i_r}\otimes\eta_{i_{r+1}})=-\eta_{i_r}\wedge\eta_{i_{r}}=0.
\]
On the other hand, if $i_r\neq i_{r+1}=1$, then on recalling $f_1=0$, one has  
\[
f_{i_r}\eta_{i_{r+1}}-f_{i_{r+1}}\eta_{i_{r}}-\nu(\eta_{i_r}\otimes\eta_{i_{r+1}})=f_2\eta_1-\nu(\eta_2\otimes\eta_1),
\]
the latter easily seen to be closed. The case $i_r\neq i_{r+1}=2$ is similar.\\

It follows that
\begin{eqnarray*}
\overline{q}_{n-1}(\sigma_F-\mathfrak{s}_F)(a_I)&=& \sum\limits_{r<n}a_{i_1}\otimes\ldots \otimes a_{i_{r-1}}\otimes b_r \otimes a_{i_{r+2}}\otimes\ldots\otimes a_{i_n}\\
&+&\sum\limits_{r<n} a_{i_1}\otimes\ldots\otimes a_{i_{r-1}}\otimes[f_{i_r}df_{i_{r+1}}]\otimes a_{i_{r+2}}\otimes\ldots\otimes a_{i_n},
\end{eqnarray*}
where 
\[
b_r=[f_{i_r}\eta_{i_{r+1}}-f_{i_{r+1}}\eta_{i_{r}}-\nu(\eta_{i_r}\otimes\eta_{i_{r+1}})].
\]
To complete the proof, it suffices to show that every term in the expansion of $\overline{q}_{n-1}(\sigma_F-\mathfrak{s}_F)(a_I)$ above belongs to $F^p(H^1_\CC)^{\otimes n-1}$. The element
\begin{equation}
a_{i_1}\otimes\ldots\otimes a_{i_{r-1}}\otimes[f_{i_r}df_{i_{r+1}}]\otimes a_{i_{r+2}}\otimes\ldots\otimes a_{i_n}\label{eq1ch4seriousrevision}
\end{equation}
is zero if $i_r$ or $i_{r+1}$ is 1. If both  $i_r$ and $i_{r+1}$ are 2, then by assumption at least $p$ of 
\begin{equation}
i_1,\ldots,i_{r-1},i_{r+2},\ldots i_{n}\label{eq2ch4seriousrevision}
\end{equation}
are 1, and hence \eqref{eq1ch4seriousrevision} belongs to $F^p(H^1_\CC)^{\otimes n-1}$. We show that
\begin{equation*}
a_{i_1}\otimes\ldots \otimes a_{i_{r-1}}\otimes b_r \otimes a_{i_{r+2}}\otimes\ldots\otimes a_{i_n}
\end{equation*} 
%\label{eq3ch4seriousrevision}
is also in $F^p(H^1_\CC)^{\otimes n-1}$. If $i_r=i_{r+1}=1$, then $b_r=0$. (In fact, the differential $f_{i_r}\eta_{i_{r+1}}-f_{i_{r+1}}\eta_{i_{r}}-\nu(\eta_{i_r}\otimes\eta_{i_{r+1}})$ is zero, see Lemma \ref{lemlift1}(ii).) If $i_r=i_{r+1}=2$, then again by assumption as least $p$ of \eqref{eq2ch4seriousrevision} are 1. Finally, suppose $i_r\neq i_{r+1}$. Then at least $p-1$ of \eqref{eq2ch4seriousrevision} are 1, so that
it is enough to show that $b_r\in F^1H^1_\CC$. We consider the case $i_r=1$. (The other case is similar.) The 1-form 
\begin{equation}
f_{i_r}\eta_{i_{r+1}}-f_{i_{r+1}}\eta_{i_{r}}-\nu(\eta_{i_r}\otimes\eta_{i_{r+1}})=-f_2\eta_1-\nu(\eta_1\otimes\eta_2)\label{eq4ch4seriousrevision}
\end{equation}
on $X-\{\infty\}$ is of type (1,0), as $\nu$ preserves the Hodge filtration. It is also closed, and hence is holomorphic on $X-\{\infty\}$. By the previous lemma and the fact that $\nu$ takes values in $E^1(X\log\infty)$, \eqref{eq4ch4seriousrevision} (is meromorphic at $\infty$ and) has a pole of order at most 1 at $\infty$. It follows from the residue theorem that indeed \eqref{eq4ch4seriousrevision} is holomorphic on $X$, and hence $b_r\in F^1H^1_\CC$, as desired.\hfill$\square$\\  

We close this section with a few remarks.
\begin{rem}
(1) Note that by Riemann-Roch, there exists a meromorphic form on $X$ with divisor $\geq-2\infty$, so that by the previous proposition there always exist $\alpha_1,\alpha_2$ satisfying Hypothesis $\star$. More explicitly, if 
$X_0$ is given by the affine equation
\[
y^2=4x^3-g_2x-g_3,
\]
and $\infty$ is the point at infinity, we can take $\alpha_2=\frac{xdx}{y}$. If $\infty$ in not the point at infinity, we can take $\alpha_2$ to be the pullback of $\frac{xdx}{y}$ along a translation. (Meanwhile, $\alpha_1$ can be taken to be any nonzero holomorphic form on $X$.)\\
(2) It is possible that in general (and not just in $g=1$ case), any collection of $\alpha_i$ satisfies Hypothesis $\star$. In fact, it would not be surprising if $F^p(L_n)_\CC$ is the span of iterated integrals of the form
\[
\int \alpha_{i_1}\ldots\alpha_{i_l}\hspace{.3in}(l\leq n)
\]
with at least $p$ of the $\alpha_{i_r}$ of the first kind. One may hope that a similar description (now counting the number of differentials of first or third kind) exists more generally for $F^pL_n(X-S,e)$, where $S$ is any finite nonempty subset of $X(\CC)$.\\      
\end{rem}

\section{Application to Periods}\label{ch5}

\subsection{Some elementary remarks}\label{sec1ch5}

For $c\in H^1_\CC$, define the space of periods of $X$ corresponding to $c$ to be 
\[
\Per_\QQ(c):=(H_1)_\QQ(c)=\sum\limits_{i\leq2g}\QQ\int\limits_{\beta_i}c.
\]
It is easy to see that
\begin{equation}\label{eq2ch5}
\Per_\QQ(c)=\sum\limits_{i\leq2g}\QQ p_i(c).
\end{equation}
Indeed, if $B=(b_{ij})$ where
\[
b_{ij}=\int\limits_{\beta_i}\omega_j, 
\]
then
\[
B\begin{pmatrix}p_1(c)\\
\vdots\\
p_{2g}(c)
\end{pmatrix}=\begin{pmatrix}
\int\limits_{\beta_1}c\\
\vdots\\
\int\limits_{\beta_{2g}}c
\end{pmatrix}.
\]
Since the Poincare pairing is non-degenerate, $B$ is invertible. Hence \eqref{eq2ch5} follows.\\

From now on we assume that the $\alpha_i$ belong to $\Omega^1(X_0)$, i.e. are regular algebraic 1-forms on $X_0$. Then the space of periods of $X_0$ is the $K$-span of the numbers
\[
\int\limits_{\beta_i}\alpha_j.
\]
We denote this space by $\Per(X_0)$. For $1\leq i,j\leq2g $, let 
\[p_{ij}=p_j([\alpha_i]),\] 
so that
\[
\alpha_i=\sum\limits_j p_{ij}\omega_j.
\]
It follows from \eqref{eq2ch5} that $\Per(X_0)$ is spanned (over $K$) by the numbers $p_{ij}$ ($i,j\leq 2g$).\\

Let $\QQ(\Per(X_0))$ be the field generated over $\QQ$ by the periods of $X_0$. It is easy to see that for any $\gamma\in\pi_1(X-\{\infty\},e)$ and $n$, the $\QQ(\Per(X_0))$-span of the numbers
\begin{equation}\label{eq7ch5}
\int\limits_\gamma\omega_{i_1}\ldots\omega_{i_n}\hspace{.2in} (i_1,\ldots,i_n\leq2g)
\end{equation}
is equal to the $\QQ(\Per(X_0))$-span of the numbers
\begin{equation}\label{eq8ch5}
\int\limits_\gamma\alpha_{i_1}\ldots\alpha_{i_n}\hspace{.2in} (i_1,\ldots,i_n\leq2g).
\end{equation}
In fact, each number in \eqref{eq8ch5} (resp. \eqref{eq7ch5}) can be written as a linear combination of the elements of the other set with coefficients being explicit polynomials (resp. rational functions) in the $p_{ij}$. 

\subsection{Methodology}\label{ch5philosophy}
It is well-known that algebraic cycles on products of $X$, or rather Hodge classes in tensor powers of $H^1$, give rise to algebraic relations between periods of $X_0$. In short, this is because these Hodge classes cut down the dimension of the Mumford-Tate group of $X$, which in turn cuts down the transcendence degree over $K$ of the field obtained by adjoining the periods of $X_0$ to $K$.\footnote[2]{It is known that the transcendence degree over $K$ of the field obtained by adjoining the periods of $X_0$ to $K$ is less than or equal to the dimension of the Mumford-Tate group of $X$. It is conjectured that the two quantities are indeed equal. (See \cite{LNM900}.)} Our main objective here is to show how Hodge classes in tensor powers of $H^1$, and hence algebraic cycles on products of $X$, might also give rise to non-trivial relations among the periods of the fundamental group of $X_0-\{\infty\}$ that lie deeper in the weight filtration, at least among the periods of $L_2(X_0-\{\infty\},e)$ (i.e. iterated integrals of length $\leq 2$ in the forms $\alpha_i$).\\  

Throughout, to simplify the notation, we identify 
\[
\Omega^1_\hol(X)=H^{1,0}
\]
via the distinguished isomorphism between them.\\

In the previous section, for each Hodge class $\xi\in(H^1)^{\otimes 2n-2}$ we defined a point
\[
P_\xi=\xi^{-1}(\Psi(\mathbb{E}^\infty_{n,e}))\in J(H^1)^\vee=\Jac(\CC).
\]
We identify
\[
J(H^1)^\vee\cong\frac{\Omega^1_\hol(X)^\vee}{H_1(X,\ZZ)}
\]
via the isomorphism given by
\[
[f]\mapsto[f\bigm|_{\Omega^1_\hol(X)}].
\]
If the $\alpha_i$ satisfy Hypothesis $\star$, the point 
\[
P_\xi\in\frac{\Omega^1_\hol(X)^\vee}{H_1(X,\ZZ)}
\]
is $\displaystyle{[f_\xi\bigm|_{\Omega^1_\hol(X)}]}$. (See Proposition \ref{analytic_description_z(e)}.) 

\begin{lemma}\label{lem1ch5}
Suppose the $\alpha_i$ satisfy Hypothesis $\star$. If $P_\xi$ is torsion, then for every $\alpha\in\Omega^1_\hol(X)$, $f_\xi(\alpha)\in\Per_\QQ(\alpha)$.
%(b) If $g=1$, then $P_\xi$ is torsion if and only if for all (or equivalently for some nonzero) $\alpha\in\Omega^1_\hol(X)$, $f_\xi(\alpha)\in\Per_\QQ(\alpha)$.
\end{lemma}
\begin{proof}
This is immediate from that $P_\xi$ is torsion if and only if $f_\xi\bigm|_{\Omega^1_\hol(X)}$ coincides with an element of $H_1(X,\QQ)$.
\end{proof}

Suppose the $\alpha_i$ satisfy Hypothesis $\star$, and a Hodge class $\xi\in (H^1)^{\otimes 2n-2}$ is such that $P_\xi$ is torsion. Then by the previous lemma and Proposition \ref{analytic_description_z(e)_V2}, for every $\alpha\in\Omega_\hol^1(X)$ one has
\begin{equation}\label{eq5ch5}
\sum\limits_{i,j,k\leq2g}\mu'_{(i,j,k)}(\xi;\alpha)\int\limits_{\beta_k}\omega_i\omega_j\in\Per_\QQ(\alpha).
\end{equation}  
The $\mu'$ are integral linear combinations of the $p_l(\alpha)$, and by \eqref{eq2ch5} they belong to $\Per_\QQ(\alpha)$. Setting $\alpha=\alpha_1,\ldots,\alpha_g$ , we get linear relations between 
\begin{equation}\label{eq4ch5}
1, \int\limits_{\beta_k}\omega_i\omega_j\hspace{.2in} (i,j,k\leq 2g)
\end{equation}
with coefficients in $\Per(X_0)$.\\ 

One has the formal relations among \eqref{eq4ch5} of the form
\begin{equation}\label{eq6ch5}
\int\limits_{\beta_k}\omega_i \int\limits_{\beta_k}\omega_j=\int\limits_{\beta_k}\omega_i\omega_j+\int\limits_{\beta_k}\omega_j\omega_i,
\end{equation}  
which come from the shuffle product property of iterated integrals. These will enable us to write the relations \eqref{eq5ch5} in fewer ``variables". For $\alpha\in\Omega^1_\hol(X)$ and a Hodge class $\xi\in (H^1)^{\otimes 2n-2}$, and $i,j,k$ such that
\[
i,j,k\leq 2g, i<j,
\]
let
\[
\mu_{(i,j,k)}(\xi;\alpha)=\mu'_{(i,j,k)}(\xi;\alpha)-\mu'_{(j,i,k)}(\xi;\alpha).
\]
\\

\begin{prop}\label{mainpropsec2ch5}
Suppose a Hodge class $\xi\in (H^1)^{\otimes 2n-2}$ is such that $P_\xi$ is torsion. If the $\alpha_i$ satisfy Hypothesis $\star$, then for every $\alpha\in\Omega^1_\hol(X)$,  
\begin{equation*}
\sum\limits_{\stackrel{i,j,k\leq2g}{i<j}}\mu_{(i,j,k)}(\xi;\alpha)\int\limits_{\beta_k}\omega_i\omega_j\in\Per_\QQ(\alpha).
\end{equation*}  
\end{prop}
\begin{proof}
We know \eqref{eq5ch5} is true. Now note that by \eqref{eq6ch5},
\begin{eqnarray*}
\sum\limits_{i,j,k\leq2g}\mu'_{(i,j,k)}(\xi;\alpha)\int\limits_{\beta_k}\omega_i\omega_j&=&\sum\limits_{\stackrel{i,j,k\leq2g}{i<j}}\mu'_{(i,j,k)}(\xi;\alpha)\int\limits_{\beta_k}\omega_i\omega_j\\
&+&\sum\limits_{i,k\leq2g}\frac{1}{2}\mu'_{(i,i,k)}(\xi;\alpha)(\int\limits_{\beta_k}\omega_i)^2\\
&+&\sum\limits_{\stackrel{i,j,k\leq2g}{i>j}}\mu'_{(i,j,k)}(\xi;\alpha)\left(\int\limits_{\beta_k}\omega_i \int\limits_{\beta_k}\omega_j-\int\limits_{\beta_k}\omega_j\omega_i\right)\\
&\stackrel{\Per_\QQ(\alpha)}{\equiv}&\sum\limits_{\stackrel{i,j,k\leq2g}{i<j}}\mu_{(i,j,k)}(\xi;\alpha)\int\limits_{\beta_k}\omega_i\omega_j~,
\end{eqnarray*}
since the $\mu'$ belong to $\Per_\QQ(\alpha)$.
\end{proof}

Suppose the $\alpha_i$ satisfy Hypothesis $\star$, and that $P_\xi$ is torsion. Taking $\alpha=\alpha_1,\ldots,\alpha_g$, we get $g$ linear relations between
\begin{equation}\label{eq9ch5}
1, \int\limits_{\beta_k}\omega_i\omega_j\hspace{.2in} (i,j,k\leq 2g, i<j)
\end{equation}
with coefficients in $\Per(X_0)$. In view of the last comment in Paragraph \ref{sec1ch5} and the shuffle product property of iterated integrals, each of these relations can be rewritten as a linear relation in
\begin{equation}\label{eq9'ch5}
1, \int\limits_{\beta_k}\alpha_i\alpha_j\hspace{.2in} (i,j,k\leq 2g, i<j)
\end{equation}
with coefficients in $\QQ(\Per(X_0))$.

\begin{rem}
(1) Suppose the $\alpha_i$ satisfy Hypothesis $\star$. Recall that if $\xi=\xi_Z$ for an algebraic cycle $Z\in\CH_{n-1}(X_0^{2n-2})$, then $P_\xi$ is in $\Jac(K)$. (See Theorem \ref{rationalpointjac}.) If the Mordell-Weil group $\Jac(K)$ is finite, then $P_\xi$ will automatically be torsion, and hence in view of Proposition \ref{mainpropsec2ch5} we get relations among \eqref{eq9ch5}. We will pursue this further in the next section.\\
(2) As it was mentioned earlier, it may be the case that Hypothesis $\star$ in fact always holds. Recall that at least we know it does hold if $g=1$ and $\alpha_2$ has order 2 at $\infty$. (See Proposition \ref{prop_Hypothesis_star_genus1}.)
\end{rem}

\section{Relations between periods- Some explicit calculations}\label{ch5section_examples}
Here we carry out the method of the previous section in some cases. In order to simplify the calculations, we will assume from now on that the cohomology classes $d_i$ are chosen in such a way that 
\[
\int\limits_X d_i\wedge d_j=1 \hspace{.2in}\text{if}\hspace{.2in}i<j.
\] 

\subsection{Relations coming from the diagonal of $X_0$}
In this paragraph, we show that interestingly, the diagonal $\Delta(X_0)$ of $X_0$ can give rise to relations between \eqref{eq9ch5} that do not seem to be trivial. This is in particular interesting, because $\Delta(X_0)$ does not give rise to a relation between the periods of $X_0$ itself.  The following lemma, whose proof we postpone until the appendix, describes $\xi_{\Delta(X_0)}$. Recall that in our notation $\xi_Z=\sum\limits_I\lambda_I(Z)d_I$.

\begin{lemma}\label{lem_xi_Delta}
We have
\[
\lambda_{ij}(\Delta(X_0))=\begin{cases}
(-1)^{i+j}\hspace{.2in}&\text{if $i<j$}\\
0&\text{if $i=j$}\\
(-1)^{i+j+1}&\text{if $i>j$}.
\end{cases}
\]
\end{lemma}

Let $\alpha\in\Omega^1_\hol(X)$. One has 
\begin{equation}\label{mu'_diagonal_X}
\mu'_{ijk}(\xi_{\Delta(X_0)};\alpha)=\lambda_{ij}(\Delta(X_0))p_k(\alpha),
\end{equation}
and hence for $i<j$,
\begin{eqnarray*}
\mu_{ijk}(\xi_{\Delta(X_0)};\alpha)&=&p_k(\alpha)\left(\lambda_{ij}(\Delta(X_0))-\lambda_{ji}(\Delta(X_0))\right)\\
&=&2(-1)^{i+j}p_k(\alpha).
\end{eqnarray*}

\begin{prop}\label{prop_relations_given_by_diagonal}
Suppose the $\alpha_i$ satisfy Hypothesis $\star$. If $P_{\Delta(X_0)}$ is torsion, then 
\begin{equation}\label{eq10ch5}
\sum\limits_{\stackrel{i,j,k\leq2g}{i<j}}(-1)^{i+j}p_{lk}\int\limits_{\beta_k}\omega_i\omega_j\in \Per_\QQ(\alpha_l)\hspace{.3in}(l=1,\ldots,g). 
\end{equation}
Moreover, these, as linear relations among \eqref{eq9ch5} with coefficients in $\QQ(\Per(X_0))$, are independent.
\end{prop}
\begin{proof}
The first assertion is a special case of Proposition \ref{mainpropsec2ch5}. As for the independence of the relations, note that the $g\times 2g$ matrix whose $lk$-entry is the coefficient in the relation corresponding to $\alpha_l$ of 
\[
\int\limits_{\beta_k}\omega_1\omega_2,
\] 
is minus the top half of the matrix $(p_{ij})_{i,j\leq 2g}$ of periods. The latter matrix is invertible and hence the former has rank $g$.\\
\end{proof}

By Theorem \ref{rationalpointjac}, $P_{\Delta(X_0)}$ is in $\Jac(K)$, so that the torsion condition automatically holds if the Mordell-Weil group $\Jac(K)$ is finite. In particular, one obtains: 
\begin{cor}\label{cor1sec3ch5}
Let $K=\QQ$ and $X_0$ be either the hyper-elliptic curve given by the affine equation
\[
y^2=x(x-3)(x-4)(x-6)(x-7),
\]
or the Fermat curve given by the affine equation
\[
x^p+y^p=1,
\]
where $p$ is an odd prime $\leq 7$. Suppose (in each case) the $\alpha_i$ satisfy Hypothesis $\star$. Then one has $g$ (the genus of $X_0$ in each case) independent relations as in \eqref{eq10ch5}.
\end{cor}

Indeed, in each of these situations $\Jac(\QQ)$ is known to be finite. See \cite{Gordon-Grant} for the hyper-elliptic curve and and \cite{Faddeev} for the given Fermat curves. Note that the points $e,\infty$ must be in $X_0(\QQ)$.
   
%Of course, if $X_0$ is an elliptic curve over $K$ with finite $X_0(K)$, then again a relation of the form \eqref{eq10ch5} is guaranteed. The reason we did not include this as a part of the corollary is that one may not need $X_0(K)$ to be finite in this case, as we will see in the next paragraph.

\subsection{More on the genus one case} One can state a more precise variation of Proposition \ref{prop_relations_given_by_diagonal} in the case $g=1$. We first prove a lemma.

\begin{lemma}
Let $g=1$. Then $2h_2(\Delta_{2,e})=0$.
\end{lemma}

\begin{proof}
We will equivalently show that
\[2\Phi(h_2(\Delta_{2,e}))\in\Hom((H^1_\ZZ)^{\otimes 3},\RR/\ZZ)\]
is zero. For a permutation $\sigma\in S_3$, denote the map
\[
X^3\longrightarrow X^3\hspace{.5in}(x_1,x_2,x_3)\mapsto (x_{\sigma(1)},x_{\sigma(2)},x_{\sigma(3)})
\]
also by $\sigma$. It is easy to see that $\sigma_\ast(\Delta_{2,e})=\Delta_{2,e}$. Let $\partial^{-1}(\Delta_{2,e})$ be as in Section \ref{construction_of_cycles}, i.e. a chain whose boundary is $\Delta_{2,e}$. Then 
\[
\partial\sigma_\ast(\partial^{-1}(\Delta_{2,e}))=\sigma_\ast\partial(\partial^{-1}(\Delta_{2,e}))=\sigma_\ast(\Delta_{2,e})=\Delta_{2,e},
\] 
so that $\sigma_\ast\partial^{-1}(\Delta_{2,e})$ can also be used to calculate $\Phi(h_2(\Delta_{2,e}))$.\\
 
Let $\eta_1,\eta_2$ be harmonic 1-forms on $X$ with integral periods whose images in cohomology form a basis of $H^1_\ZZ$. Then
\begin{eqnarray*}
\int\limits_{\partial^{-1}(\Delta_{2,e})}\eta_{i_1}\otimes\eta_{i_2}\otimes\eta_{i_3}&\stackrel{\ZZ}{\equiv}&\int\limits_{\sigma_\ast\partial^{-1}(\Delta_{2,e})}\eta_{i_1}\otimes\eta_{i_2}\otimes\eta_{i_3}\\
&=&\int\limits_{\partial^{-1}(\Delta_{2,e})}\sigma^\ast(\eta_{i_1}\otimes\eta_{i_2}\otimes\eta_{i_3})\\
&=&\text{sgn}(\sigma)\int\limits_{\partial^{-1}(\Delta_{2,e})}\eta_{i_{\sigma(1)}}\otimes\eta_{i_{\sigma(2)}}\otimes\eta_{i_{\sigma(3)}}.
\end{eqnarray*}
So far $\sigma$ was arbitrary. Now given a triple $(i_1,i_2,i_3)$, take $\sigma$ to be a transposition that fixes the triple. (Such transposition exists because $g=1$.) Then it follows from the above that 
\[
\int\limits_{\partial^{-1}(\Delta_{2,e})}\eta_{i_1}\otimes\eta_{i_2}\otimes\eta_{i_3}\in \frac{1}{2}\ZZ.
\] 
Thus the image of $\Phi(h_2(\Delta_{2,e})$ lies in $(\frac{1}{2}\ZZ)/\ZZ$, i.e. $2\Phi(h_2(\Delta_{2,e})=0$.\\
\end{proof}

\begin{rem}
Gross and Schoen \cite[Corollary 4.7]{GS} showed that when $g=1$, $6\Delta_{2,e}$ is zero in $\CH^\hhom_1(X^3)$. 
\end{rem}

\begin{thm}\label{thmch5}
Let $g=1$. Suppose the $\alpha_i$ satisfy Hypothesis $\star$. (Recall that this is guaranteed for instance if $\alpha_2$ has order 2 at $\infty$). Then 
\begin{equation}\label{eq16ch5}
p_{11}\int\limits_{\beta_1}\omega_1\omega_2+p_{12}\int\limits_{\beta_2}\omega_1\omega_2 \equiv \int\limits_e^\infty \alpha_1\hspace{.1in}\mod \frac{1}{4}\Per_\ZZ(\alpha_1),
\end{equation}
where $\Per_\ZZ(\alpha_1)=(H_1)_\ZZ(\alpha_1)$. In particular, 
\[
p_{11}\int\limits_{\beta_1}\omega_1\omega_2+p_{12}\int\limits_{\beta_2}\omega_1\omega_2\in \Per_\QQ(\alpha_1)
\]
if and only if $\infty-e$ is torsion in $\CH^\hhom_0(X_0)$ (or equivalently, in $X_0(K)$).
\end{thm}

\begin{proof}
In view of the previous lemma and Theorem \ref{main1},
\begin{equation}\label{aug27eq1}
2\xi_{\Delta(X_0)}^{-1}(\Psi(\mathbb{E}^\infty_{2,e}))=2\xi_{\Delta(X_0)}^{-1}(h_2(Z^\infty_{2,e}))\in J(H^1)^\vee\cong\frac{\Omega^1_\hol(X)^\vee}{H_1(X,\ZZ)}.
\end{equation}
Fix a path $\gamma_e^\infty$ from $e$ to $\infty$ in $X$. The elements $\xi_{\Delta(X_0)}^{-1}(\Psi(\mathbb{E}^\infty_{2,e}))$ and $\xi_{\Delta(X_0)}^{-1}(h_2(Z^\infty_{2,e}))$ of $\displaystyle{\frac{\Omega^1_\hol(X)^\vee}{H_1(X,\ZZ)}}$ are respectively represented by $f_{\Delta(X_0)}$ and the map
\[
\alpha\mapsto \int\limits_{\Delta(X)}\xi_{\Delta(X_0)}\int\limits_{\gamma_e^\infty}\alpha.
\]
Thus \eqref{aug27eq1} gives 
\[
2f_{\Delta(X_0)}(\alpha_1) \equiv 2\int\limits_{\Delta(X)}\xi_{\Delta(X_0)}\int\limits_{\gamma_e^\infty}\alpha_1\hspace{.1in}\mod \Per_\ZZ(\alpha_1).
\]
Straightforward calculations using \eqref{mu'_diagonal_X}, Lemma \ref{lem_xi_Delta}, and Proposition \ref{analytic_description_z(e)_V2} show 
\[\int\limits_{\Delta(X)}\xi_{\Delta(X_0)}=-2\hspace{.2in}\text{and}\hspace{.2in} f_{\Delta(X_0)}(\alpha_1)=-2 \left(p_{11}\int\limits_{\beta_1}\omega_1\omega_2+p_{12}\int\limits_{\beta_2}\omega_1\omega_2\right).
\]
The first assertion follows. The second assertion follows from the first and the classical Abel-Jacobi theorem.
\end{proof}

\begin{rem}
(1) Using 
\[
p_{l1}=-\int\limits_{\beta_2}\alpha_l~,~~p_{l2}=\int\limits_{\beta_1}\alpha_l\hspace{.2in}(l=1,2)
\]
and the shuffle product property of iterated integrals, the left hand side of \eqref{eq16ch5} can be rewritten as
\[
\frac{1}{2(\int_{\beta_1}\alpha_1\int_{\beta_2}\alpha_2-\int_{\beta_1}\alpha_2\int_{\beta_2}\alpha_1)}\left(\int\limits_{\beta_1}\alpha_1\int\limits_{\beta_2} (\alpha_1\alpha_2-\alpha_2\alpha_1)~-~\int\limits_{\beta_2}\alpha_1\int\limits_{\beta_1} (\alpha_1\alpha_2-\alpha_2\alpha_1)\right).
\]
(2) Suppose $X_0$ is given by the affine equation
\[
y^2=4x^3-g_2x-g_3.
\] 
Let $\infty$ be the point at infinity. Take $\alpha_1=\frac{dx}{y}$ and $\alpha_2=\frac{xdx}{y}$. One then has the {\it Legendre relation}
\[
\int_{\beta_1}\alpha_1\int_{\beta_2}\alpha_2-\int_{\beta_1}\alpha_2\int_{\beta_2}\alpha_1=2\pi i.
\]
Equation \eqref{eq16ch5} can be rewritten as
\[
\int\limits_{\beta_1}\alpha_1\int\limits_{\beta_2} (\alpha_1\alpha_2-\alpha_2\alpha_1)~-~\int\limits_{\beta_2}\alpha_1\int\limits_{\beta_1} (\alpha_1\alpha_2-\alpha_2\alpha_1)\equiv 4\pi i\int\limits_e^\infty \alpha_1\hspace{.1in}\mod \pi i\cdot\Per_\ZZ(\alpha_1).
\]
\end{rem}

\subsection{Relations coming from the diagonal of $X_0^2$}
So far in this section we considered relations that can arise from a Hodge class in $(H^1)^{\otimes 2}$ (namely, the class of the diagonal of $X_0$), and hence only used $n=2$ case of Theorem \ref{main1} and Theorem \ref{rationalpointjac}. In fact, we did not even need the full machinery of the former: We only needed \eqref{DRSraw} of Darmon, Rotger, and Sols. Our goal in this paragraph is to provide evidence for that, applying the method of Section \ref{ch5} to Hodge classes in higher tensor powers of $H^1$, or algebraic cycles in higher powers of $X$, and hence using the results of the previous sections in $n>2$ setting, one may indeed obtain new information about the periods. To this end, we will study the relations that can arise from $\Delta(X^2_0)\in \CH_2(X_0^4)$, where $\Delta(X^2_0)$ is the diagonal of $X^2_0$. We will then show that at least in $g=2$ case, these relations are not the same as the ones arising from $\Delta(X_0)$.\\

Throughout, for simplicity, we write $\lambda_{ij}$ for $\lambda_{ij}(\Delta(X_0))$ (given in Lemma \ref{lem_xi_Delta}). 
\begin{lemma}\label{lem_mu_diagonal_X^2}
Let $\alpha\in\Omega^1_\hol(X)$. Then for $i,j,k\leq 2g$, $i<j$,
\[
\mu_{ijk}(\xi_{\Delta(X_0^2)};\alpha)=\lambda_{jk}p_i(\alpha)-\lambda_{ik}p_j(\alpha)-2(-1)^{i+j}p_k(\alpha).
\]
\end{lemma}
The proof of this lemma is a fairly long computation. We postpone it to the appendix.\\

Suppose the $\alpha_i$ satisfy Hypothesis $\star$ and $P_{\Delta(X^2_0)}$ is torsion. (The latter for instance will automatically hold if $\Jac(K)$ is finite, e.g. in the cases as in Corollary \ref{cor1sec3ch5}.) Then by Proposition \ref{mainpropsec2ch5}, 
\begin{equation}\label{eq15ch5}
\sum\limits_{\stackrel{i,j,k}{i<j}}\biggm(\lambda_{jk}p_{li}-\lambda_{ik}p_{lj}-2(-1)^{i+j}p_{lk}\biggm)\int\limits_{\beta_k}\omega_i\omega_j~\in\Per_\QQ(\alpha_l)\hspace{.3in}(l\leq g).
\end{equation}

\begin{prop}  
The relations \eqref{eq15ch5} are independent (as linear relations among \eqref{eq9ch5} with coefficients in $\QQ(\Per(X_0)$).
\end{prop}
\begin{proof}
Let $A$ be the matrix formed by the coefficients of 
\[
\int\limits_{\beta_1}\omega_1\omega_2,\hspace{.1in}\text{and}\hspace{.3in} \int\limits_{\beta_j}\omega_1\omega_j\hspace{.2in} (1<j\leq 2g)
\]
in the relations. (In other words, the $l1$-entry of $A$ is the coefficient of $\int\limits_{\beta_1}\omega_1\omega_2$ in the relation corresponding to $\alpha_l$, and for $j>1$, its $lj$-entry is the coefficient of $\int\limits_{\beta_j}\omega_1\omega_j$ in the relation corresponding to $\alpha_l$.) It is enough to show that $A$ has rank $g$. But this is clear, since one has
\[
\mu_{1,2,1}(\Delta(X_0^2);\alpha_l)=3p_{l1}
\]
and
\[
\mu_{1,j,j}(\Delta(X_0^2);\alpha_l)=3(-1)^jp_{lj},
\]
so that the $j^\text{th}$ column of $A$ is $\pm3$ the $j^\text{th}$ column of the top half of the period matrix $(p_{ij})_{i,j\leq 2g}$.
\end{proof}

Suppose both $P_{\Delta(X_0)}$ and $P_{\Delta(X_0^2)}$ are torsion, and that the $\alpha_i$ satisfy Hypothesis $\star$. Then one has two sets of $g$ independent relations given in \eqref{eq10ch5} and \eqref{eq15ch5}. In $g=1$ case, the two relations are trivially dependent. On the other hand, one has:

\begin{prop}
Let $g=2$ and the $\alpha_i$ satisfy Hypothesis $\star$. If $P_{\Delta(X_0)}$ and $P_{\Delta(X_0^2)}$ are both torsion, then among relations \eqref{eq10ch5} and \eqref{eq15ch5}, there are at least 3 (i.e. $g+1$) independent ones.
\end{prop}

\begin{proof}
In view of \eqref{eq10ch5}, we can replace \eqref{eq15ch5} by 
\begin{equation}\label{eq17ch5}
\sum\limits_{\stackrel{i,j,k}{i<j}}\biggm(\lambda_{jk}p_{li}-\lambda_{ik}p_{lj}\biggm)\int\limits_{\beta_k}\omega_i\omega_j~\in\Per_\QQ(\alpha_l)\hspace{.3in}(l=1,2).
\end{equation}
We refer to the relations \eqref{eq10ch5} by $R_1$, $R_2$ ($R_l$ for the one corresponding to $\alpha_l$), and to the relations \eqref{eq17ch5} by $R'_1$, $R'_2$. Suppose both $\{R_1,R_2,R'_1\}$ and $\{R_1,R_2,R'_2\}$ are dependent. We claim that for all distinct $i,j,k\leq 4$, $i<j$, and $l\leq 2$,
\begin{equation}\label{eq18ch5}
(\lambda_{jk}p_{li}-\lambda_{ik}p_{lj}+\lambda_{ij}p_{lk})(p_{1i}p_{2j}-p_{1j}p_{2i})=0. 
\end{equation}
Indeed, given $i,j,k,l$ as above, form the $3\times3$ matrix whose columns are the coefficients of 
\[
\int\limits_{\beta_i}\omega_i\omega_j,\hspace{.2in} \int\limits_{\beta_j}\omega_i\omega_j,\hspace{.2in} \int\limits_{\beta_k}\omega_i\omega_j
\] 
in $E_1, E_2, E'_l$. One easily calculates its determinant to be 
\[
(\lambda_{jk}p_{li}-\lambda_{ik}p_{lj}+\lambda_{ij}p_{lk})(p_{1i}p_{2j}-p_{1j}p_{2i}),
\]
so that the claim follows.\\

Next, we show that the equations \eqref{eq18ch5} contradict the fact that the matrix 
\[
P:=(p_{ij})_{i\leq g, j\leq 2g}
\]
has rank $g$. This will be done in two steps. Let $P_j$ be the $j^{\text{th}}$ column of $P$.\\

\underline{Step 1}: Consider the following situations:
\[
(i) \det\begin{pmatrix}
p_{11}&p_{12}\\
p_{21}&p_{22}
\end{pmatrix}\neq 0,\hspace{.2in}(ii) \det\begin{pmatrix}
p_{11}&p_{14}\\
p_{21}&p_{24}
\end{pmatrix}\neq 0,\]
\[(iii) \det\begin{pmatrix}
p_{13}&p_{14}\\
p_{23}&p_{24}
\end{pmatrix}\neq 0,\hspace{.2in}(iv) \det\begin{pmatrix}
p_{12}&p_{13}\\
p_{22}&p_{23}
\end{pmatrix}\neq 0.
\]
Suppose (i) holds. Then in view of \eqref{eq18ch5},
\begin{equation}\label{eq19ch5}
\lambda_{2k}P_1-\lambda_{1k}P_2+\lambda_{12}P_k=0 \hspace{.2in}(k=3,4).
\end{equation}
Setting $k=3,4$ it follows $P_3=-P_4$. On the other hand, \eqref{eq19ch5} gives $P_3=-(P_1+P_2)$, so that
\[
P=\begin{pmatrix}
p_{11}&p_{12}&-(p_{11}+p_{12})&p_{11}+p_{12}\\
p_{21}&p_{22}&-(p_{21}+p_{22})&p_{21}+p_{22}\end{pmatrix}.
\]
It follows that (ii) holds.\\

Similarly, one can check that 
\begin{itemize}
\item[-] (ii) implies (iii) and that $P_2=-P_3$,
\item[-] (iii) implies (iv) and that $P_1=-P_2$, and finally
\item[-] (iv) implies (i) and that $P_1=P_4$.
\end{itemize}
Since $P$ has rank 2, it follows none of $(i)-(iv)$ hold, i.e.
\[
\det\begin{pmatrix}
p_{11}&p_{12}\\
p_{21}&p_{22}
\end{pmatrix}=\det\begin{pmatrix}
p_{11}&p_{14}\\
p_{21}&p_{24}
\end{pmatrix}=\det\begin{pmatrix}
p_{13}&p_{14}\\
p_{23}&p_{24}
\end{pmatrix}=\det\begin{pmatrix}
p_{12}&p_{13}\\
p_{22}&p_{23}
\end{pmatrix}=0.
\]

\underline{Step 2}: Since 3rd and 4th columns of $P$ are linearly dependent and $P$ has rank 2, one of the first two columns must be nonzero. We assume the first column is not zero; the other case is similar. By the previous step, $P$ must look like
\[
\begin{pmatrix}
p_{11}&0&p_{13}&0\\
p_{21}&0&p_{23}&0\end{pmatrix}.
\]
Indeed, $P_2$ and $P_4$ are scalar multiples of $P_1$, so that $\rank(P)=2$ forces $P_1,P_3$ to be linearly independent. Each of $P_2,P_4$ is a scalar multiple of both $P_1$ and $P_3$, and hence is zero. Now taking $(i,j,k)=(1,3,2)$ in \eqref{eq18ch5} we see $P_1=-P_3$, contradicting $\rank(P)=2$.  
\end{proof}

\appendix
\section{Proofs of Lemmas \ref{lem_xi_Delta} and \ref{lem_mu_diagonal_X^2}}

\begin{proof}[Proof of Lemma \ref{lem_xi_Delta}] Let $\{c_i\}$ be the basis of $H^1_\ZZ$ that is dual to $\{d_i\}$ with respect to Poincare duality, i.e.
\[
\int\limits_X c_i\wedge d_j=\delta_{ij}:=\begin{cases}
1\hspace{.2in}&\text{if $i=j$}\\
0&\text{otherwise}.\end{cases}
\]
We use the multi-index notation for the $\{c_i\}$ as well: $c_{ij}$ means $c_i\otimes c_j$. For simplicity, write $\lambda_{ij}$ for $\lambda_{ij}(\Delta(X_0))$. One has for each $i,j$, 
\[
\int\limits_{\Delta(X)}c_{ij}=\int\limits_{X^2} \sum\limits_{k,l}\lambda_{kl}d_{kl} \wedge c_{ij},
\]
which can be rewritten as
\[
\int\limits_{X}c_i\wedge c_j=-\sum\limits_{k,l}\lambda_{kl}\int\limits_{X^2} (d_k\wedge c_i)\otimes (d_l\wedge c_j),
\]
the latter being clearly equal to 
\[
-\sum\limits_{k,l}\lambda_{kl}\int\limits_{X}(d_k\wedge c_i) \int\limits_{X}(d_l\wedge c_j).
\]
It follows that 
\begin{equation}\label{eq11ch5}
\lambda_{ij}=\int\limits_{X}c_j\wedge c_i.
\end{equation}
Let $A=(a_{ij})$, where 
\[
a_{ij}=\int\limits_X d_i\wedge d_j,
\]
so that $A$ is a $2g$ by $2g$ skew-symmetric matrix with the entries above the diagonal all equal to 1. For each $i$, let
\[
c_i=\sum\limits_j b_{ij}d_j.
\]
Let $B=(b_{ij})$. One has
\[
\delta_{ij}=\int\limits_X c_i\wedge d_j=\int\limits_X \sum\limits_k b_{ik}d_k\wedge d_j=\sum\limits_k b_{ik}a_{kj},
\]
so that $BA$ is identity, $B=A^{-1}$. It follows that
\[
b_{ij}=\begin{cases}
(-1)^{i+j}\hspace{.2in}&\text{if $i<j$}\\
0&\text{if $i=j$}\\
(-1)^{i+j+1}&\text{if $i>j$}.
\end{cases}
\]
On the other hand, by \eqref{eq11ch5},
\[
\lambda_{ij}=\sum\limits_{k,l} b_{jk}b_{il}a_{kl},
\]
which is the $ij$-entry of the matrix $B(BA)^t=B$. The result follows.
\end{proof}

\begin{proof}[Proof of Lemma \ref{lem_mu_diagonal_X^2}] For the moment, let $\xi\in (H^1)^{\otimes 4}$ be an arbitrary Hodge class. For simplicity, we write $\lambda_{ijkl}$ for $\lambda_{ijkl}(\xi)$. Let $\alpha\in\Omega^1_\hol(X)$. We will simply write $p_j$ for $p_j(\alpha)$. One easily sees
\[
\mu_{ijk}'(\xi;\alpha)=\sum\limits_{l,m}~p_m\lambda_{ijlk}a_{ml}+p_k\lambda_{lijm}a_{ml},
\]
where $\displaystyle{a_{ml}=\int\limits_{\beta_m}\omega_l}$. Thus for $i<j$,
\begin{eqnarray*}
\mu_{ijk}(\xi;\alpha)&=&\sum\limits_{l,m}~a_{ml}\biggm(p_m\lambda_{ijlk}+p_k\lambda_{lijm}-p_m\lambda_{jilk}-p_k\lambda_{ljim}\biggm)\notag\\
&=&\sum\limits_{l,m} a_{ml}\biggm(p_m\left(\lambda_{ijlk}-\lambda_{jilk}\right)+p_k\left(\lambda_{lijm}-\lambda_{ljim}\right)\biggm).
\end{eqnarray*} 
In view of $a_{ml}=-a_{lm}$ and $a_{ml}=1$ if $m<l$, this can be rewritten as
\[
\sum\limits_{m<l} \biggm(p_m\left(\lambda_{ijlk}-\lambda_{jilk}\right)+p_k\left(\lambda_{lijm}-\lambda_{ljim}-\lambda_{mijl}+\lambda_{mjil}\right)+p_l\left(\lambda_{jimk}-\lambda_{ijmk}\right)\biggm),
\]
which can again be rewritten as
\begin{equation}\label{eq21ch5}
\sum\limits_{m}p_m\left(\sum\limits_{l=m+1}^{2g}(\lambda_{ijlk}-\lambda_{jilk})+\sum\limits_{l=1}^{m-1}(\lambda_{jilk}-\lambda_{ijlk})\right)~+~p_k\sum\limits_{m<l} \left(\lambda_{lijm}-\lambda_{ljim}-\lambda_{mijl}+\lambda_{mjil}\right).
\end{equation}
Now let $\xi=\xi_{\Delta(X_0^2)}$. We will simply write $\mu_{ijk}$ for $\mu_{ijk}(\xi;\alpha)$, and continue to write $\lambda_{ij}$ (resp. $\lambda_{ijkl}$) for $\lambda_{ij}(\Delta(X_0)$ (resp. $\lambda_{ijkl}(\Delta(X_0^2)$). Since $\Delta(X_0^2)$ is obtained from $\Delta(X_0)\times \Delta(X_0)$ by switching the 2nd and 3rd coordinates, one has
\[
\lambda_{ijkl}=-\lambda_{ik}\lambda_{jk}.
\]
In view of $\lambda_{ij}=-\lambda_{ji}$, \eqref{eq21ch5} simplifies to
\[
\sum\limits_{m}p_m\left(\sum\limits_{l=m+1}^{2g}(\lambda_{ijlk}-\lambda_{jilk})+\sum\limits_{l=1}^{m-1}(\lambda_{jilk}-\lambda_{ijlk})\right)~+~2p_k\sum\limits_{m<l} (\lambda_{lijm}-\lambda_{ljim}),
\]
Thus so far we know
\[
\mu_{ijk}=\sum\limits_{m}a_mp_m~+~2p_k \sum\limits_{m<l} (\lambda_{lijm}-\lambda_{ljim}),
\]
where
\begin{eqnarray*}
a_m&=&\sum\limits_{l=m+1}^{2g}(\lambda_{ijlk}-\lambda_{jilk})+\sum\limits_{l=1}^{m-1}(\lambda_{jilk}-\lambda_{ijlk})\\
&=&\left(\sum\limits_{l=m+1}^{2g}-\sum\limits_{l=1}^{m-1}\right)(\lambda_{ijlk}-\lambda_{jilk}).
\end{eqnarray*}
Thus we will be done if we show
\begin{equation}\label{eq22ch5}
\sum\limits_{m<l} (\lambda_{lijm}-\lambda_{ljim})=(-1)^{i+j+1} \hspace{.3in}(\text{for all $i<j$})
\end{equation}
and
\begin{equation*}
\left(\sum\limits_{l=m+1}^{2g}-\sum\limits_{l=1}^{m-1}\right) \lambda_{ijlk}=\begin{cases}\lambda_{jk}\hspace{.3in}&\text{if $m=i$}\\
0&\text{if $m\neq i$}.\end{cases}\hspace{.3in}(\text{for all distinct $i,j$})
\end{equation*}
The latter is equivalent to that for all $i$ and $m$,
\begin{equation}\label{eq23ch5}
\left(\sum\limits_{l=m+1}^{2g}-\sum\limits_{l=1}^{m-1}\right) \lambda_{li}=\begin{cases}1\hspace{.3in}&\text{if $m=i$}\\
0&\text{if $m\neq i$}.\end{cases}
\end{equation}
Before we try to verify these, note that for any fixed $i$ and $r$, one has:
\begin{itemize}
\item[(i)] If $r<i$, then 
\[
\sum_{l\leq r}\lambda_{li}=\begin{cases}
\lambda_{1i}=\lambda_{ri}\hspace{.3in}&(r\stackrel{2}{\not\equiv}0)\\
0&(r\stackrel{2}{\equiv}0)\end{cases}.
\]
\item[(ii)] If $r\geq i$, then 
\[
\sum_{i<l\leq r}\lambda_{li}=\begin{cases}
\lambda_{(i+1)i}=\lambda_{ri}\hspace{.3in}&(r\stackrel{2}{\not\equiv}i)\\
0&(r\stackrel{2}{\equiv}i)\end{cases}.
\]
\end{itemize}
For $r\geq i$, writing
\[
\sum_{l\leq r}\lambda_{li}=\left(\sum\limits_{l\leq i-1}+\sum\limits_{i<l\leq r}\right)\lambda_{li},
\]
we see that for any $r,i$,  
\[
\sum\limits_{l\leq r}\lambda_{li}=\begin{cases}
\lambda_{1i}=(-1)^{i+1}\hspace{.2in}&(r<i, r\stackrel{2}{\not\equiv} 0)\\
0&(r<i, r\stackrel{2}{\equiv} 0)\\
\lambda_{1i}=(-1)^{i+1}&(r\geq i, r\stackrel{2}{\equiv}i\stackrel{2}{\equiv}0)\\
\lambda_{1i}+\lambda_{ri}=0&(r\geq i, r\stackrel{2}{\not\equiv}i\stackrel{2}{\equiv}0)\\
0&(r\geq i, r\stackrel{2}{\equiv}i\stackrel{2}{\not\equiv}0)\\
\lambda_{ri}=(-1)^{i+1}&(r\geq i, r\stackrel{2}{\not\equiv}i\stackrel{2}{\not\equiv}0),
\end{cases}
\]
or in short,
\begin{equation}\label{eq24ch5}
\sum\limits_{l\leq r}\lambda_{li}=\begin{cases}
(-1)^{i+1}\hspace{.2in}&(r<i, r\stackrel{2}{\not\equiv} 0)~\text{or}~(r\geq i, r\stackrel{2}{\equiv}0)\\
0&(r<i, r\stackrel{2}{\equiv} 0)~\text{or}~(r\geq i, r\stackrel{2}{\not\equiv}0).
\end{cases}
\end{equation}

Now we verify \eqref{eq22ch5} and \eqref{eq23ch5}. Writing
\[
\left(\sum\limits_{l=m+1}^{2g}-\sum\limits_{l=1}^{m-1}\right) \lambda_{li}=-\lambda_{mi}+\left(\sum\limits_{l\leq 2g}-2\sum\limits_{l\leq m-1}\right)\lambda_{li},
\]
a straightforward computation using \eqref{eq24ch5} gives \eqref{eq23ch5}.\\ 

Turning our attention to \eqref{eq22ch5}, start by breaking the sum as
\begin{equation*}
\sum\limits_{m<l} (\lambda_{lijm}-\lambda_{ljim})=\sum\limits_{m<l}\lambda_{lijm}~-~\sum\limits_{m<l}\lambda_{ljim}.
\end{equation*}
We have
\begin{equation*}
\sum\limits_{m<l}\lambda_{lijm}=\sum\limits_{l}\lambda_{lj}\sum\limits_{m=1}^{l-1}\lambda_{mi}=(-1)^j\left(\overbrace{\sum\limits_{l<j}}^{(I)}-\overbrace{\sum\limits_{l>j}}^{(II)}\right)(-1)^{l}\sum\limits_{m=1}^{l-1}\lambda_{mi}.
\end{equation*}
Before we proceed any further, it is convenient to use the following notation. Given a subset $S\subset\RR$, we denote by $E(S)$ (resp. $O(S)$) the number of even (resp. odd) numbers in $S$. In view of \eqref{eq24ch5},
\[
(I)=(-1)^{i+1}\bigm(E((0,i])-O((i,j))\bigm)
\]
and 
\[
(II)=(-1)^{i}O((j,2g]).
\]
(since $i<j$). Thus
\begin{equation}\label{eq25ch5}
\sum\limits_{m<l}\lambda_{lijm}=(-1)^{i+j+1}\bigm(E((0,i])-O((i,j))+O((j,2g])\bigm).
\end{equation}
Similarly,
\begin{eqnarray*}
\sum\limits_{m<l}\lambda_{ljim}&=&\sum\limits_{l}\lambda_{li}\sum\limits_{m=1}^{l-1}\lambda_{mj}\\
&=&(-1)^i\left(\sum\limits_{l<i}-\sum\limits_{l>i}\right)(-1)^{l}\sum\limits_{m=1}^{l-1}\lambda_{mj}.
\end{eqnarray*}
In view of \eqref{eq24ch5}, keeping in mind $i<j$, we get
\begin{equation}\label{eq26ch5}
\sum\limits_{m<l}\lambda_{ljim}=(-1)^{i+j+1}\bigm(E((0,i))-E((i,j])+O((j,2g])\bigm).
\end{equation}
Now \eqref{eq22ch5} follows from \eqref{eq25ch5} and \eqref{eq26ch5} on noting that
\[
E((0,i])-O((i,j))-E((0,i)))+E((i,j])=E([i,j])-O((i,j))=1.
\]
\end{proof}

\end{document}